\documentclass[]{article}

\usepackage[linesnumbered,ruled,vlined]{algorithm2e}

\usepackage{amsthm}
\usepackage{amsmath}
\usepackage{amssymb}
\usepackage{amsgen}

\usepackage{array}
\usepackage{authblk}

\usepackage[english]{babel}
\usepackage{booktabs}

\usepackage[font=footnotesize]{caption}
\usepackage[commandnameprefix=always, markup=nocolor]{changes}
\usepackage{comment}

\usepackage{dsfont}

\usepackage{etoolbox}

\usepackage{fancyhdr}
\usepackage{float}
\usepackage{fontawesome}
\usepackage[T1]{fontenc}

\usepackage[a4paper, portrait, margin=1.1811in]{geometry}
\usepackage{graphicx}

\usepackage{helvet}
\usepackage[colorlinks, citecolor=teal, linkcolor=teal, urlcolor=teal]{hyperref}
\usepackage[acronym, nonumberlist, style=long, toc]{glossaries}

\usepackage[misc]{ifsym}
\usepackage[utf8]{inputenc}

\usepackage{lmodern}

\usepackage{makecell} 
\usepackage{makeidx}

\usepackage{mathtools}

\usepackage[framemethod=TikZ]{mdframed}

\usepackage{multicol}
\usepackage{multirow}

\usepackage{natbib}   

\usepackage{orcidlink}

\usepackage{rotating}

\usepackage{scrextend}
\usepackage{subfig}

\usepackage{titlesec}

\usepackage{varwidth}
\usepackage{verbatim}

\usepackage{xcolor}

\usepackage{xfor}
\usepackage{xkeyval}

\SetCommentSty{mycommfont}

\SetKwInput{KwInput}{Input}                
\SetKwInput{KwOutput}{Output}              

\newcounter{theo}[section]
\renewcommand{\thetheo}{\arabic{section}.\arabic{theo}}
\newenvironment{theo}[2][]{%
\refstepcounter{theo}%
\ifstrempty{#1}%
{\mdfsetup{%
frametitle={%
\tikz[baseline=(current bounding box.east),outer sep=0pt]
\node[anchor=east,rectangle,fill=gray!20]
{\strut Theorem~\thetheo};}}
}%
{\mdfsetup{%
frametitle={%
\tikz[baseline=(current bounding box.east),outer sep=0pt]
\node[anchor=east,rectangle,fill=gray!20]
{\strut Theorem~\thetheo:~#1};}}%
}%
\mdfsetup{innertopmargin=10pt,linecolor=gray!20,%
linewidth=2pt,topline=true,%
frametitleaboveskip=\dimexpr-\ht\strutbox\relax
}
\begin{mdframed}[]\relax%
\label{#2}}{\end{mdframed}}

\newcounter{prf}[section]

\newenvironment{prf}[2][]{%
\refstepcounter{prf}%
\ifstrempty{#1}%
{\mdfsetup{%
frametitle={%
\tikz[baseline=(current bounding box.east),outer sep=0pt]
\node[anchor=east,rectangle,fill=gray!20]
{\strut Proof};}}
}%
{\mdfsetup{%
frametitle={%
\tikz[baseline=(current bounding box.east),outer sep=0pt]
\node[anchor=east,rectangle,fill=gray!20]
{\strut Proof:~#1};}}%
}%
\mdfsetup{innertopmargin=10pt,linecolor=gray!20,%
linewidth=2pt,topline=true,%
frametitleaboveskip=\dimexpr-\ht\strutbox\relax
}
\begin{mdframed}[]\relax%
\label{#2}}{\qed\end{mdframed}}

\newcounter{prop}[section]
\renewcommand{\theprop}{\arabic{section}.\arabic{prop}}
\newenvironment{prop}[2][]{%
\refstepcounter{prop}%
\ifstrempty{#1}%
{\mdfsetup{%
frametitle={%
\tikz[baseline=(current bounding box.east),outer sep=0pt]
\node[anchor=east,rectangle,fill=gray!20]
{\strut Proposition~\theprop};}}
}%
{\mdfsetup{%
frametitle={%
\tikz[baseline=(current bounding box.east),outer sep=0pt]
\node[anchor=east,rectangle,fill=gray!20]
{\strut Proposition~\theprop:~#1};}}%
}%
\mdfsetup{innertopmargin=10pt,linecolor=gray!20,%
linewidth=2pt,topline=true,%
frametitleaboveskip=\dimexpr-\ht\strutbox\relax
}
\begin{mdframed}[]\relax%
\label{#2}}{\end{mdframed}}

\newcounter{defi}[section]
\renewcommand{\thedefi}{\arabic{section}.\arabic{defi}}
\newenvironment{defi}[2][]{%
\refstepcounter{defi}%
\ifstrempty{#1}%
{\mdfsetup{%
frametitle={%
\tikz[baseline=(current bounding box.east),outer sep=0pt]
\node[anchor=east,rectangle,fill=gray!20]
{\strut Definition~\thedefi};}}
}%
{\mdfsetup{%
frametitle={%
\tikz[baseline=(current bounding box.east),outer sep=0pt]
\node[anchor=east,rectangle,fill=gray!20]
{\strut Definition~\thedefi:~\textbf{#1}};}}%
}%
\mdfsetup{innertopmargin=10pt,linecolor=gray!20,%
linewidth=2pt,topline=true,%
frametitleaboveskip=\dimexpr-\ht\strutbox\relax
}
\begin{mdframed}[]\relax%
\label{#2}}{\end{mdframed}}

\newcounter{rema}[section]
\renewcommand{\therema}{\arabic{section}.\arabic{rema}}

\newcounter{lem}[section]
\renewcommand{\thelem}{\arabic{section}.\arabic{lem}}

\makeglossaries

\newacronym{ATD}{ATD}{Average Trimmed Distance}

\newacronym{BDCA}{BDCA}{Boosted Difference of Convex Functions Algorithm}

\newacronym[longplural={Confidence Intervals}, \glsshortpluralkey={CIs}]{CI}{CI}{Confidence Interval}
\newacronym{CNPP}{CNPP}{Continuous Nonlinear Programming Problem}
\newacronym[longplural={Conditional Value-at-Risks}, \glsshortpluralkey={CVaRs}]{CVaR}{CVaR}{Conditional Value-at-Risk}
\newacronym{C-step}{C-step}{Concentration step}
\newacronym{CDF}{CDF}{Cumulative Distribution Function}

\newacronym{DBI}{DBI}{Davies Bouldin Index}

\newacronym{DC}{DC}{Difference of Convex Functions}
\newacronym{DCA}{DCA}{Difference of Convex Functions Algorithm}

\newacronym{dkabs}{$\mathbf{d}_t$-abs}{Absolute $\Delta$ in Solution Vector}
\newacronym{dkrel}{$\mathbf{d}_t$-rel}{Relative $\Delta$ in Solution Vector}

\newacronym{DMLP}{DMLP}{Data Model with Leverage Points}
\newacronym{DJ}{DJ}{Dow Jones Industrial Average}

\newacronym{FMI}{FMI}{Fowlkes-Mallows Index}

\newacronym{FTSE}{FTSE}{Financial Times Stock Exchange 100}
\newacronym{FF49}{FF49}{Fama \& French 49 Industrial Portfolios}

\newacronym{fctabs}{$\phi$-abs}{Absolute $\Delta$ in Objective Function}
\newacronym{fctcrel}{$\phi$-rel}{Relative $\Delta$ in Objective Function}

\newacronym{iid}{iid}{independent and identically distributed}

\newacronym{KKT}{KKT}{Karush–Kuhn–Tucker}

\newacronym{LTS}{LTS}{Least Trimmed Squares}

\newacronym[longplural={Market Values}, \glsshortpluralkey={MaVas}]{MaVa}{MaVa}{Market Value}

\newacronym{MCD}{MCD}{Minimum Covariance Determinant}
\newacronym{MFCQ}{MFCQ}{Mangasarian-Fromovitz Constraint Qualification}
\newacronym{MIPP}{MIPP}{Mixed Integer Programming Problem}
\newacronym{MSD}{MSD}{Million Song Dataset}

\newacronym{OLS}{OLS}{Ordinary Least Squares}
\newacronym{OSCM}{OSCM}{One Sided Contamination Model}

\newacronym{PCC}{PCC}{Pearson Correlation Coefficient}
\newacronym{PDF}{PDF}{Probability Density Function}

\newacronym{QP}{QP}{Quadratic Program}
  
\newacronym{RANSAC}{RANSAC}{Random Sample Consensus}

\newacronym{sBDCA}{sBDCA}{successive Boosted Difference of Convex Functions Algorithm}

\newacronym{SCC}{SCC}{Spearman's Rank Correlation Coefficient}

\newacronym{SCD}{SCD}{Superconductivity Data}

\newacronym{TS}{TS}{Theil-Sen}

\newacronym{WLS}{WLS}{Weighted Least Squares}

\setcitestyle{authoryear, open={(},close={)}}

\begin{document}

\pagestyle{fancy}

\lhead{Thormann et al.}
\rhead{Robust Regression with DC Programming}

\setlength{\headheight}{35pt}

\newpage
\setcounter{page}{1}
\renewcommand{\thepage}{\arabic{page}}
	
\captionsetup[figure]{labelfont={bf},
                      labelformat={default}, 
                      labelsep=period, 
                      name={Figure }
                      }	
                      
\captionsetup[table]{labelfont={bf}, 
                     labelformat={default}, 
                     labelsep=period, 
                     name={Table }
                     }

\setlength{\parskip}{0.5em}



\fancypagestyle{plain}{
	\fancyhf{}
	\renewcommand{\headrulewidth}{0pt}
	\renewcommand{\familydefault}{\sfdefault}
	\lhead{\color{teal}\normalsize \textbf{Working Paper} (v2)\\ \color{black}
	\text{Operational Research Group, University of Southampton}\\ }
	
}

\makeatletter

\patchcmd{\@maketitle}{\LARGE \@title}{\fontsize{19}{24}\selectfont\@title}{}{}
\makeatother

\setlength{\affilsep}{2em}  
\newsavebox\affbox

\author{
    {\large \textbf{Marah-Lisanne Thormann}} {\small \href{mailto:m.-l.thormann@soton.ac.uk}{\Letter}  \href{https://www.southampton.ac.uk/people/5ztphy/ms-marah-thormann}{\faHome} \protect\\ \vspace{0.1cm}
     School of Mathematical Sciences, \protect\\ University of Southampton, \protect\\ SO17 1BJ Southampton, UK} \\ \vspace{-0.25cm}
    {\large \textbf{Phan Tu Vuong}} {\small \href{mailto:t.v.phan@soton.ac.uk}{\Letter} \href{https://www.southampton.ac.uk/people/5y2ds9/doctor-vuong-phan}{\faHome} {\large \orcidlink{0000-0002-1474-994X}} \protect\\ \vspace{0.1cm}
     School of Mathematical Sciences,  \protect\\ University of Southampton, \protect\\ SO17 1BJ  Southampton, UK}  \vspace{0.3cm}
    
    {\large \textbf{Alain B. Zemkoho}} {\small \href{mailto:a.b.zemkoho@soton.ac.uk}{\Letter} \href{https://www.southampton.ac.uk/~abz1e14/index.html}{\faHome} {\large \orcidlink{0000-0003-1265-4178}} \protect\\
     \vspace{0.1cm} School of Mathematical Sciences, \protect\\ University of Southampton, \protect\\ SO17 1BJ  Southampton, UK} \vspace{0.3cm}

    {\large \textbf{Tri-Dung Nguyen}} {\small \href{mailto:t.nguyen-264@kent.ac.uk}{\Letter} \href{https://www.kent.ac.uk/kent-business-school/people/4056}{\faHome} {\large \orcidlink{0000-0002-4158-9099}} \protect\\
    \vspace{0.1cm} Kent Business School, \protect\\ University of Kent, \protect\\ CT2 7NZ Canterbury, UK} \vspace{-0.1cm}
}

\titlespacing\section{0pt}{12pt plus 4pt minus 2pt}{0pt plus 2pt minus 2pt}
\titlespacing\subsection{12pt}{12pt plus 4pt minus 2pt}{0pt plus 2pt minus 2pt}
\titlespacing\subsubsection{12pt}{12pt plus 4pt minus 2pt}{0pt plus 2pt minus 2pt}

\titleformat{\section}{\normalfont\fontsize{10}{15}\bfseries}{\thesection.}{1em}{}
\titleformat{\subsection}{\normalfont\fontsize{10}{15}\bfseries}{\thesubsection.}{1em}{}
\titleformat{\subsubsection}{\normalfont\fontsize{10}{15}\bfseries}{\thesubsubsection.}{1em}{}

\title{\vspace{-0.3cm} \textbf{Faster than Fast-LTS: Robust Regression and  Outlier Detection with DC Programming} \\ \vspace{0.6cm}}
\date{}  
	
\begingroup
\let\center\flushleft
\let\endcenter\endflushleft

\maketitle

\vspace{-0.7cm}

	

\noindent\rule{1.25cm}{0.4pt}  \rlap{\color{gray}\vrule}%
  \fboxsep1.5mm\colorbox[rgb]{1,1,1}{\raisebox{-0.4ex}{%
    \large\selectfont\sffamily\bfseries\abstractname}} \noindent\rule{11.6cm}{0.4pt}
    
    \noindent
    {\small When datasets contain outliers, robust regression is a well-established alternative to Ordinary Least Squares. 
    A commonly employed robust estimator is \gls{LTS}, which computes the regression coefficients from a subset of observations. 
    Determining the exact solution corresponds to a combinatorial problem with prohibitive computational costs, even for instances of moderate dimension.
    Thus, the most prevalent approach in practice remains a heuristic known as Fast-\gls{LTS}.  
    Although the heuristic often performs effectively, certain elements of the approach remain open to improvement.
    In particular, its core procedure provides robust results only when initialised with a large number of starting points.
    To address the heuristic's limitations, this paper reformulates the \gls{LTS} problem as a concave minimisation problem subject to a capped simplex constraint, and proposes the \gls{sBDCA} as a solution method. 
    Theoretically, we establish via the \L ojasiewicz property that \gls{sBDCA} converges to a local solution with a linear rate in the fastest case.  
    To ensure robustness from a single initialisation in practice, we derive and integrate a problem-specific preconditioning matrix into the algorithmic setup. 
    Building on this theoretical foundation, we 
    conduct numerical studies on various synthetic and real-world datasets to demonstrate the effectiveness of \gls{sBDCA} with preconditioning. 
    Specifically, we show that our approach is up to 3.25 times faster than Fast-\gls{LTS} and achieves up to 90\% lower objective function values, particularly in high-dimensional settings.
    In addition to this performance comparison, we introduce a novel application of \gls{LTS} for outlier detection in sports analytics, where the market values of football players are predicted to identify undervalued athletes. 
    As all code is openly available, this paper further provides a practical guide to robust regression in Python.
     
    \vspace{0.2cm}
    
    \noindent \textbf{Keywords}: Successive Boosted Difference of Convex Functions Algorithm; Difference of Convex Functions Programming; Least Trimmed Squares; Robust Regression; Outlier Detection

    \vspace{0.2cm}
  
    \noindent \textbf{Mathematics Subject Classification (2020)}: 62F30; 62F35; 62J20; 65K05; 65K10; 90C26   

    \vspace*{\fill}

    \noindent \rule{15cm}{0.2pt}
    {\footnotesize \noindent 
    \textbf{Funding Acknowledgement}: 
    The first author receives a PhD studentship from the School of Mathematical Sciences at the University of Southampton.}
    
    }

\endgroup

\glsresetall


\tableofcontents

\newpage


\section{Introduction}\label{sec:Intro}



``Bad data yields bad models. This is often referred to as the GIGO principle: garbage .. in, garbage .. out'' [\citet[p. 17]{Baesens2009}]. 
Thus, modelling can quickly backfire if attention is not given to the quality of the data. 
However, measuring and improving data quality can be a time-consuming task, often involving several steps before the model can even be estimated [cf.~\citet[p. 1781]{Schelter2018}; \citet[pp. 39--40]{garcia2015}]. 
In fact, the time spent preparing data usually far outweighs the time required for the modelling itself [cf. \citet[p.~1]{fernandes2023}]. 
To overcome this challenge, both practitioners and theoreticians have been investigating robust methods capable of handling suboptimal data quality for over 60 years [cf.~\citet[p.~179]{Fahrmeir2021}; \citet[p. 43]{Tukey1962}]. 
In particular, for the Linear Regression model, several robust estimators have emerged that grapple with the trade-off between statistical efficiency and robustness against model deviations [cf. \citet[pp. 64--65]{Maronna2019}; \citet[pp. 441--445]{Verardi2009}].
Among these estimators, \gls{LTS} offers superior statistical properties but entails a combinatorial burden [cf. \citet[p. 227]{Alfons2013}; \citet[p. 45]{FLORES2015}; \citet[p. 15]{rousseeuw1987}].
In this paper, we address the aforementioned challenge of computing this estimator, and demonstrate how it can be approximated more effectively from a single starting point than with existing methods.


To better contextualize the problem setting and our proposed methodology, the remainder of this introduction is divided into several subsections. 
We begin by revisiting the Linear Regression model and explain how its coefficients are estimated using \gls{OLS}. 
Subsequently, we address the limitations of \gls{OLS} by focusing on its sensitivity to outliers in the underlying data.
We then discuss the statistical properties an estimator should satisfy to ensure robustness.
Afterwards, we introduce \gls{LTS} as a commonly used robust estimator and highlight its advantages as well as its computational complexity. 
Next, we review existing solution approaches for \gls{LTS} and emphasize the unexplored potential of \gls{DC} programming to outperform the state-of-the-art approach. 
Ultimately, we summarise the key contributions of this paper, which lays the foundation for a practical guide to robust regression and outlier detection based on \gls{LTS} and \gls{DC} programming in Python.


\subsection{Linear Regression Model and Ordinary Least Squares}


In many practical problems, we aim to estimate a model that describes the relationship between a set of explanatory variables and a dependent variable [cf.~\citet[p. 23]{Fahrmeir2021}]. 
If we assume linear and uncorrelated effects of the covariates and a continuous response, one of the most commonly applied statistical models is Linear Regression [cf. \citet[p. 44]{hastie2009}]. 
The core idea of the approach is to model the relationship as a non-deterministic linear function such that for a dataset with $n \in \mathbb{N}$ observations, and $p  \in \mathbb{N}$ independent variables (including an intercept), we obtain a system of equations with the form
\begin{equation*}
    \mathbf{y} = \mathbf{X} \boldsymbol{\beta} + \boldsymbol{\epsilon},
\end{equation*}
where $\mathbf{y} := \begin{bmatrix} y_1, \ldots, y_n \end{bmatrix}^\top \in \mathbb{R}^n$ denotes the dependent variable, and $\boldsymbol{\epsilon} := \begin{bmatrix} \epsilon_1, \ldots, \epsilon_n \end{bmatrix} ^\top \in \mathbb{R}^n$ is the random component [cf. \citet[pp. 23--24]{Fahrmeir2021}]. 
The systematic component $\mathbf{X} \boldsymbol{\beta} \in \mathbb{R}^n$, consists of the vector of unknown regression coefficients $\boldsymbol{\beta} := \begin{bmatrix} \beta_1, \ldots, \beta_p \end{bmatrix} ^{\top} \in \mathbb{R}^p$, and the  design matrix given by 
\begin{equation*}
    \mathbf{X} := \begin{bmatrix}
                       x_{11} & x_{12} & \cdots & x_{1p} \\ 
                       x_{21} & x_{22} & \cdots & x_{2p} \\
                       \vdots & \vdots & \ddots & \vdots \\
                       x_{n1} & x_{n2} & \cdots & x_{np} 
                    \end{bmatrix}
                =  
    \begin{bmatrix} \mathbf{x}_1, \ldots, \mathbf{x}_n \end{bmatrix}^\top \in \mathcal{M}_{n \times p} \left(\mathbb{R} \right),
\end{equation*}
whose $n$-rows are the feature vectors $\mathbf{x}_i := \begin{bmatrix} x_{i 1}, \ldots, x_{ip} \end{bmatrix}^\top \in \mathbb{R}^{p}$ for all instances $i \in \{1, \ldots, n \}$. 
For the random component, the model assumes that the individual elements are \gls{iid}, i.e., $\mathbb{E} \left(\boldsymbol{\epsilon} \right) = \mathbf{0}_n \ \text{and} \ \mathbb{V} \left(\boldsymbol{\epsilon} \right) = \sigma^2 \cdot \mathbf{I}_n$ where $\mathbf{0}_n \in \{0\}^n$, $\sigma \in \mathbb{R}_+$ denotes the constant standard deviation, and $\mathbf{I}_n \in \mathcal{M}_{n \times n} \left ( \mathbb{R} \right)$ is the identity matrix [cf. \citet[p. 30]{Fahrmeir2021}]. 
Subsequently, the unknown regression coefficients can be estimated using \gls{OLS}, a method first developed by Gauss and Legendre around 1800 [cf. \citet[p. 44]{hastie2009}; \citet[p. 239]{Plackett1972}]. 
According to this method, the coefficients $\widehat{\boldsymbol{\beta}} \in \mathbb{R}^p$ are estimated by minimizing the sum of squared residuals, i.e.,
\begin{align}\label{eq:classic_OLS}
    \underset{\boldsymbol{\beta}}{\text{minimise}} \ \sum_{i=1}^n r_i^2 = \sum_{i=1}^n \left(y_i -  \hat{y}_i\right)^2,
\end{align}
where $r_i \in \mathbb{R}$ corresponds to the residual for instance $i$ calculated as the difference between target variable $y_i \in \mathbb{R}$ and model prediction $\hat{y}_i := \mathbf{x}_i^\top \widehat{\boldsymbol{\beta}}$. 
The solution to this minimisation problem can be computed using matrix algebra, which made \gls{OLS} especially attractive at the time of its invention, as it was the only feasible approach not requiring a computer [cf. \citet[p.~2]{rousseeuw1987}]. 
Even today,  \gls{OLS} remains popular due to its computational speed and optimality under the additional assumption of normally distributed errors [cf. \citet[p. 75]{rousseeuw2011}]. 
However, the method exhibits limited robustness under poor data quality, which is addressed in the next subsection [cf. \citet[p. 787]{Western1995}; \citet[p. 2]{rousseeuw1987}]. 


\subsection{Sensitivity of Ordinary Least Squares to Outliers}\label{sec:outliers_in_regression}


One of the main reasons for poor data quality is the presence of outliers [cf. \citet[p. 237]{aggarwal2015}]. 
Formally, an outlier can be defined as an instance that is so different from the remaining data that it is suspected to have been generated by a different mechanism [cf. \citet[p. 1]{hawkins1980}]. 
Therefore, outliers can be errors, extreme values, or ``... discordant data that may arise from the natural variation within the population or process'' [\citet[p.~163]{Salgado2016}].  
While outliers can provide useful information in applications like fraud detection, in regression analysis they often cause the estimator to break down [cf. \citet[p.~237]{aggarwal2015}; \citet[p.~163]{Salgado2016}]. 
Moreover, \textit{masking} and \textit{swamping} effects may occur, where outlying data points hide each other and go undetected, or valid observations are misclassified as outliers, respectively  [cf. \citet[p. 626]{She2011}; \citet[p. 145]{pena1995}]. 
Simplified illustrations of both effects are shown in the left and right columns of Figure~\ref{fig:outlier_types}, using artificially generated data with one independent variable. 
In the upper subplots, the outliers are clearly visible due to the univariate nature of the data. 
However, the diagnostic subplots highlight the challenge of detecting outliers in multivariate settings, where univariate plots may fail to reveal anomalies.


\vspace{-0.2cm}

\begin{figure}[H]%
    \centering
    \captionsetup{justification=centering}
    \includegraphics[width=1\textwidth]{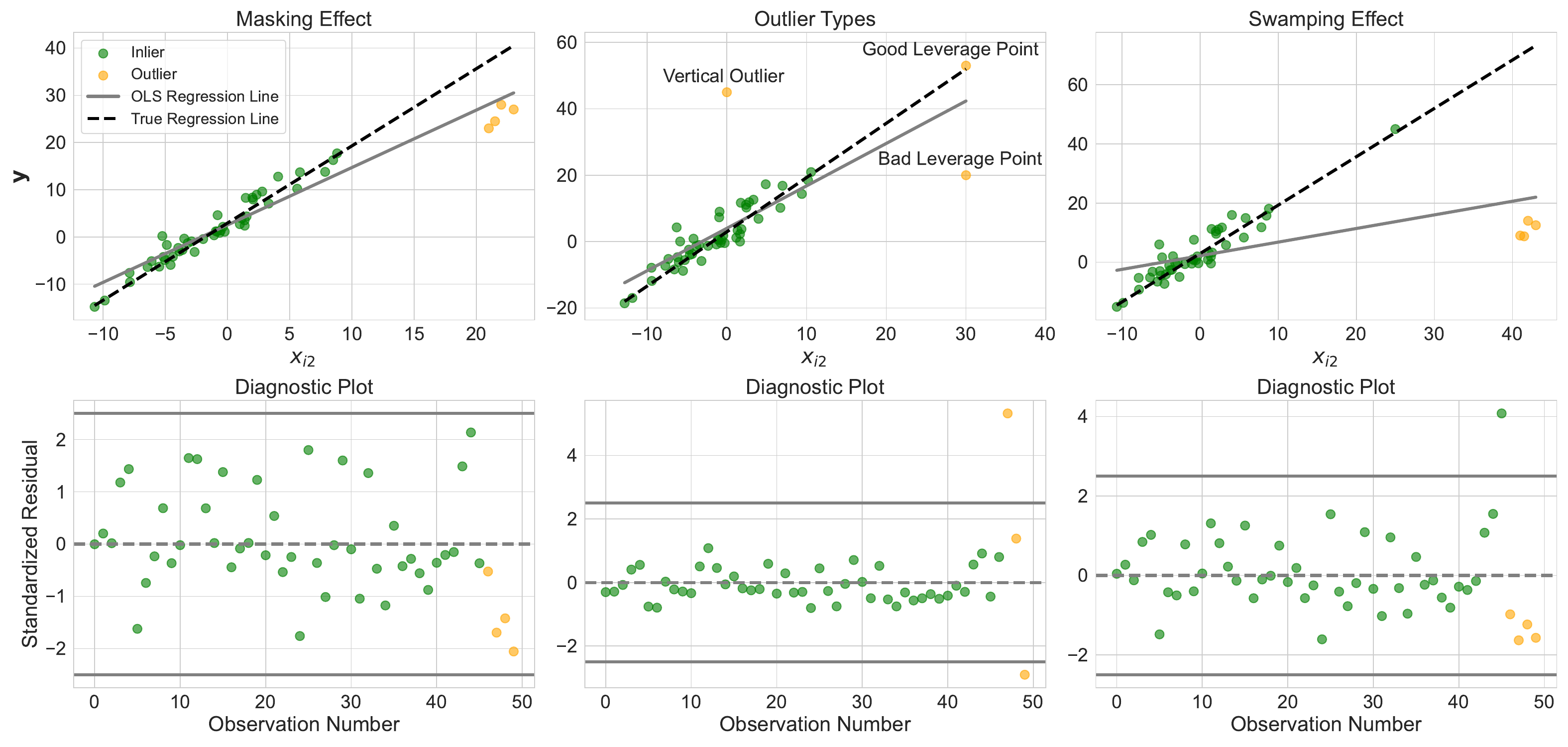}
    \caption{Masking, Swamping, and Types of Outliers in Regression Analysis.}%
    \label{fig:outlier_types}%
\end{figure}

\vspace{-0.2cm}


In regression analysis, different types of outliers can be distinguished. 
If an observation deviates significantly from the true underlying regression fit, it is classified as a \textit{regression outlier} [cf.~\citet[p.~636]{Rousseeuw1990}]. 
In contrast, an instance is considered a \textit{leverage point} if it is unusual with respect to the covariates [cf. \citet[p. 264]{reimann2008}]. 
Observations that meet both criteria are called \textit{bad leverage points}, while those that meet only the first or second definition are referred to as  \textit{vertical outliers} and \textit{good leverage points}, respectively [cf. \citet[pp. 433--434]{beliakov2012}; \citet[p. 410]{alma2011}; \citet[p. 440]{Verardi2009}; \citet[p. 636]{Rousseeuw1990}; \citet[pp. 3--7]{rousseeuw1987}]. 
A simple visualisation of the different outlier types is provided in the subplots located in the middle column of Figure~\ref{fig:outlier_types}.
In the next subsection, we explore strategies for handling these outliers and examine the statistical requirements that a regression estimator must satisfy to ensure robustness.


\subsection{Robust Estimators in Regression Analysis}\label{sec:robust_estimators_in_reg_analysis}


According to \citet[p. 88]{hampel1973} and \citet[p. 26]{hampel1986}, a routine dataset may contain up to 10\% gross errors. 
Therefore, in most practical modelling problems, outlying observations appear to be the rule rather than the exception, and it is essential to consider how best to mitigate their impact [cf. \citet[p. 3]{stuart2011}]. 
The academic literature suggests three general approaches in this regard [cf. \citet[p. 1]{SABZEKAR2021}]. 
First, one can preprocess the data to filter out conspicuous instances [cf. \citet[p. 108]{garcia2015}]. 
Second, sample polishing methods can be applied to correct anomalies prior to model training [cf. \citet[p. 1]{SABZEKAR2021}]. 
Lastly, one can employ robust training algorithms whose results are less sensitive to outliers [cf. \citet[p. 108]{garcia2015}]. 
As robust training methods, to some extent, perform preprocessing and modelling simultaneously, they can offer a time advantage over approaches that separate outlier detection and modelling into two distinct, consecutive tasks.
These methods are also referred to as indirect approaches to outlier detection, as their primary objective is to safeguard the robustness of model parameters, while the identification of suspicious observations typically occurs as a by-product -- often based on residual analysis [cf. \citet[p. 8]{huber2009}; \citet[p. 1265]{Hadi1993}]. 


Particularly in regression analysis, the use of robust methods has grown substantially since the 1960s, accompanied by the development of essential criteria that robust estimators are expected to meet [cf. \citet[p. 179]{Fahrmeir2021}].
According to \citet[p. 5]{huber2009}, a well-designed robust procedure is characterised by three desirable properties:
\begin{enumerate}
    \item \textit{Efficiency}: When all model assumptions are fulfilled, the estimator should deliver optimal or near-optimal results. 
    \item \textit{Stability}: The performance should only be marginally affected by smaller deviations from the assumed model [cf. \citet[p. 2]{stuart2011}].
    \item \textit{Breakdown}: The breakdown point of an estimator is defined as the smallest fraction of bad observations that can make the estimator’s bias arbitrarily large [cf. \citet[p. 434]{beliakov2012}; \citet[p. 8]{huber2009}; \citet[pp. 9--10]{rousseeuw1987}]. 
    Thus, an estimator should have a high breakdown point to prevent larger deviations from the assumed model leading to a collapse.
\end{enumerate}
All three of these statistical properties are important, but robustness is inherently based on compromise [cf. \citet[pp. 64--65]{Maronna2019}]. 
Protecting performance against model deviations often comes at the cost of sacrificing some efficiency [cf. \citet[p. 5]{huber2009}]. 
In the academic literature, several robust alternatives to \gls{OLS} were suggested that deal with this trade-off in different ways [cf. \citet[pp. 441--445]{Verardi2009}]. 
For example, popular choices are the Median Regression, M-Estimators, S-Estimators, MM-Estimators, Huber Regression, \gls{RANSAC}, and \gls{TS} estimator [cf. \citet[pp. 1014--1016]{QU2021}; \citet[p. 367]{Huang2016}; \citet[p. 451]{yu2014}; \citet[p. 1836]{Shen2013}; \citet[pp. 194--203]{Wilcox2010}; \citet[p. 61]{owen2007}; \citet[p. 1]{Subbarao2006}]. 
Beyond these approaches, one of the most widely applied robust estimators is \gls{LTS} as it can obtain the highest possible breakdown point, and has a better statistical efficiency than competing methods [cf. \citet[p. 227]{Alfons2013}; \citet[p. 15]{rousseeuw1987}].
Thus, this paper focuses on the practical application of this estimator, which is discussed in more detail in the next subsection.


\subsection{Least Trimmed Squares: Challenges and Heuristic Approaches}\label{sec:LTS_Challenges_Heuristic}


Originally proposed by \citet{rousseeuw1984}, the main idea of \gls{LTS} is to compute $\widehat{\boldsymbol{\beta}}$ by only considering the sum of the $h \in \mathbb{N}$ (with $h < n$) smallest squared residuals, i.e.,
\begin{align}\label{eq:Classic_LTS}
    \underset{\boldsymbol{\beta}}{\text{minimise}} \ q \left(\boldsymbol{\beta} \right) := \sum_{i=1}^h (r^2)_{i:n},
\end{align}
where $\frac{n}{2} \leq h < n$, $q: \mathbb{R}^p \rightarrow \mathbb{R}_+$, and $(r^2)_{1:n} \leq \cdots \leq (r^2)_{n:n}$ are the ordered squared estimation errors. 
Thus, the estimation procedure is in line with \eqref{eq:classic_OLS} as soon as the subset of $h$ observations is identified. 
In order to obtain this subset, the most na\"ive approach would be to construct all possible subsets of $h$ instances, and compare their objective function values, which is also known as brute-force search [cf. \citet[p. 1044]{GILONI2002}]. 
In general, this demonstrates that a solution to problem \eqref{eq:Classic_LTS} always exists, but it also emphasizes that \gls{LTS} is a combinatorial problem that is NP-hard [cf. \citet[p. 14]{bernholt2006}; \citet[p. 30]{rousseeuw2006}]. 
As a result, the brute-force search becomes intractable even in moderate dimensions due to the demanding computational complexity [cf. \citet[p. 150]{mount2014}]. 


To make the \gls{LTS} estimator practically applicable, for more than 40 years both practitioners and theoreticians have sought to develop algorithms that reduce computation time while still producing solutions that are either optimal or locally optimal, and computationally efficient [cf.~\citet{FLORES2015}; \citet{mount2014}; \citet{Kan2013}; \cite{Shen2013}; \citet{Satman2012}; \citet{Hofmann2010}; \citet{NGUYEN2010}; \citet{Zioutas2009}; \citet{rousseeuw2006}; \citet{GILONI2002}; \citet{AGULLO2001}; \citet{atkinson1999}; \citet{HAWKINS1999}; \citet{HOSSJER1995}; \citet{HAWKINS1994}]. 
To the best of our knowledge, the fastest currently available exact algorithm solves the \gls{LTS} problem globally in $\mathcal{O}(n^{p+1})$ time [cf. \citet[p.~149]{mount2016}]. 
While this represents an improvement over brute-force search, it remains computationally expensive for even moderate values of $n$ and $p$. 
As a result, the \gls{LTS} problem can still only be solved exactly for small datasets. 
In practice, it is typically approached using heuristics, which yield non-deterministic outcomes and almost always converge to at most locally optimal solutions [cf. \citet[p. 3]{XIE2022}; \citet[p. 813]{SudermannMerx2021}].


The most prominent and frequently applied heuristic to approximately solve the \gls{LTS} problem -- with implementations in statistical software like \href{https://www.sas.com/en_gb/home.html}{SAS} or \href{https://www.r-project.org/}{R} -- is called Fast-\gls{LTS} [cf. \citet[p.~4]{SMITI2020}; \citet[p. 150]{mount2014}; \citet[p. 3214]{NGUYEN2010}]. 
The approach was introduced by \citet{rousseeuw2006} and provides a robust and fast framework to estimate the regression coefficients in the presence of outliers. 
Within the heuristic, the purpose of each iteration is to update the regression coefficients based on the $h$ instances that provide the smallest squared residuals. 
This updating step is repeated as long as the subset of $h$ observations continues to change. 
Since the procedure involves only a simple sorting of residuals and basic matrix calculus, each iteration can be executed at a very low computational cost, provided that an initial starting point has been selected. 
Theoretically, \citet[pp.~32--33]{rousseeuw2006} showed that this idea results in a monotonically decreasing sequence $\{[q \left(\boldsymbol{\beta} \right)]_k\}$ which is bounded from below and converges in finite steps. 
To find smaller objective function values, the core procedure is repeated based on 500 different initial regression fits, and in the end only the solution that provided the lowest $q \left (\boldsymbol{\beta} \right )$ is kept [cf. \citet[pp.~32--33]{rousseeuw2006}]. 
For larger choices of $n$, the latter is also combined with random sampling that creates a nested system of subsets, and which ensures that the computation time remains manageable [cf.~\citet[pp. 36--37]{rousseeuw2006}].


Even though Fast-\gls{LTS} often works effectively in practice, certain aspects of the approach still offer room for improvement. 
One of the major drawbacks arises from the theoretical point of view. 
Apart from the locally convergent sequences, \citet{rousseeuw2006} did not derive other theoretical properties (e.g., convergence to a stationary point) that would also hold in a more general framework, leading to a categorisation as a heuristic [cf. \citet[p. 150]{mount2014}]. 
From a practical perspective, a major drawback lies in the fact that the underlying core procedure is not inherently robust, delivering reliable results only when executed with a substantially large number of starting points [cf. \citet[p. 3]{Heng2025}; \citet[p. 2505]{TORTI2012}]. 
Moreover, as this number remains fixed irrespective of the problem’s dimensionality, we later empirically observe a decline in solution quality, as reflected by higher~objective function values, for moderate to large values of $p$. 
This deterioration can be attributed to the decreasing quality of the initial regression fit, which stems from the use of randomly selected $p$-subsets for initialisation [cf. \citet[pp. 139--199]{Hawkins2002}]. 
In these cases, users must decide if a faster computation is more important than a higher solution quality. 
In the next subsection, we review alternatives to the Fast-\gls{LTS} heuristic and highlight how mathematical programming can be leveraged to approximately solve the \gls{LTS} problem.


\subsection{Mathematical Programming Formulations of Least Trimmed Squares}


Over the years, the lack of theoretical properties and the large number of required starting points have prompted multiple attempts to challenge Fast-\gls{LTS} in terms of both computation time and objective function value [cf. \citet{zuo2022}; \citet{mount2016}; \citet{mount2014}; \citet{Kan2013}; \citet{Shen2013}; \citet{Satman2012}; \citet{HARRINGTON2010}; \citet{Hofmann2010}; \citet{NGUYEN2010}; \citet{Salibian-Barrera2006}]. 
Although no other approach has yet managed to simultaneously outperform the state-of-the-art in both performance criteria, several researchers have demonstrated that improving one criterion at the expense of the other is achievable. 
A promising research direction involves investigating mathematical programming representations, as they offer various ways to incorporate robustness into the model formulation [cf. \citet{BARBATO2024}; \citet{Wu2023}; \citet{Molybog2020}; \citet{Chen2018}; \citet{Bertsimas2016}]. 


For the \gls{LTS} problem, \citet{Shen2013} and \cite{NGUYEN2010} have already shown that mathematical programming approaches can provide an approximate solution more quickly, with slightly higher objective function values compared to Fast-\gls{LTS}. 
While the former approximately solved the problem with a second-order cone program, the latter achieved the same through a semidefinite program.  
Both techniques are well-studied in the academic literature on mathematical optimisation and, as such, offer a broader theoretical foundation and universal validity [cf.~\citet[ch. 13]{Antoniou2021}, \citet[pp. 1--2]{cipolla2024}, \citet{anjos2011}]. 
The common basis for the derivation of both programs is the transformation of the \gls{LTS} problem into a concave minimisation over a polytope -- a representation that originally was derived by \citet{GILONI2002}, even prior to the introduction of Fast-\gls{LTS}. 
This formulation then also provides a natural transition to \gls{DC} programming, a broader problem class that includes concave minimisation, and which we propose in this paper to challenge the performance of Fast-\gls{LTS} [cf.~\citet[p. 133]{tuy2018}; \citet[p. 27]{LeThi2005}].


In general, \gls{DC} programming has become one of the most widely used approaches in non-convex optimisation, and has been extensively studied in the literature [cf. \citet[Section 2, 3, 4]{deOliveira2020}; \citet[p. 133]{tuy2018}; \citet[pp. 70--71]{Horst1995}; \citet[p. 27]{LeThi2005}]. 
Before this programming technique can be applied, a \gls{DC} decomposition of the objective must be derived.
In theory, such representations always exist for continuous optimisation problems on compact sets, but they are not uniquely defined and must be tailored to the problem [cf. \citet[pp. 149--150]{Horst1995}]. 
Although several natural candidates for constructing a \gls{DC} decomposition exist for concave minimisation problems, they have not yet been proposed for the \gls{LTS} problem. 
To the best of our knowledge, the only known application of \gls{DC} programming in the \gls{LTS} literature appears in \citet[pp. 87--88; 93--98]{Liu2019}.
However, their proposed \gls{DC} decomposition is deduced based on a modified version of the problem, as it also incorporates variable selection through a truncated $\ell_1$ regulariser. 
Moreover, the performance of their proposed algorithm was neither compared to standard approaches such as Fast-\gls{LTS}, nor evaluated on datasets containing bad leverage points -- arguably the most challenging type of outlier [cf. Subsection~\ref{sec:outliers_in_regression}]. 
As a result, the position of their method within the broader landscape of \gls{LTS} algorithms continues to be an open question.
Overall, this highlights a threefold research gap that this paper seeks to address: (i) the absence of a robust algorithm using solely a single starting point (ii) the lack of natural \gls{DC} decompositions grounded in the concave minimisation formulation of the classical \gls{LTS} problem; and (iii) the absence of a comprehensive evaluation of \gls{DC} programming in comparison with Fast-\gls{LTS}, particularly on datasets containing bad leverage points.


\subsection{Paper Contributions \& Outline}


In line with the previously identified research gaps, this paper proposes the \gls{sBDCA} to approximately solve the \gls{LTS} problem more effectively (from a single starting point) than existing approaches.
For this purpose, we first derive a powerful \gls{DC} decomposition grounded in the concave minimisation formulation of the \gls{LTS} problem.
This results in a linearly constrained \gls{DC} program, which, to the best of our knowledge, has not been considered for this problem before. 
Afterwards, the \gls{sBDCA} is proposed as a solution method that allows for a different \gls{DC} decomposition at each iteration and includes a simplified line search procedure. 
Theoretically, we then derive the convergence rate to a local solution of the problem by proving that the objective function is a multivariate rational function (i.e., a fraction of two polynomials), which fulfils the  \L ojasiewicz property.
To ensure robustness from a single initialisation, we also derive and incorporate a problem-specific preconditioning matrix into the \gls{sBDCA} framework, which aims to better balance the influence of regression outliers and leverage points on the optimisation result.
In the practical part of this paper, we then conduct extensive numerical experiments to compare the performance of our proposed algorithm with Fast-\gls{LTS} based on multiple synthetic and real-world datasets, and different strategies to create starting points. 
Overall, these comparisons demonstrate that \gls{sBDCA} with preconditioning (i) delivers more robust results under the considered initialisation schemes, (ii) is up to 3.25 times faster than the state-of-the-art, and (iii) reduces objective function values by up to 90\%, particularly in settings with many independent variables.
Separate from this benchmarking, we showcase the advantages of our robust regression approach in a novel data application for outlier detection within sports analytics.
In particular, we collect a new dataset using information provided by popular football websites to construct a prediction model for the \gls{MaVa} of players in the German 1. Bundesliga. 
The corresponding outputs serve as the basis for an outlier analysis to detect undervalued young athletes. 
Ultimately, since all the code used is publicly available, this paper further provides a practical guide for robust regression and outlier detection based on \gls{LTS} and \gls{DC} programming in Python.


The remaining sections of this paper are structured as follows. 
In Section \ref{sec:LTS_as_DC}, we demonstrate how the \gls{LTS} problem can be reformulated as a concave minimisation over a polytope. 
Based on this concave representation, we deduce several natural \gls{DC} decompositions for the \gls{LTS} problem.
In Section \ref{sec:DC_Programming_for_LTS}, we then propose the \gls{sBDCA} and derive the convergence rate to a local solution of the problem. 
Afterwards, Section~\ref{sec:Application_to_Robust_Regression} compares the performance of our algorithm with Fast-\gls{LTS} and other robust estimators based on numerical experiments with synthetic and real-world data.
In Section~\ref{sec:Application_to_Outlier_Detection}, we then showcase the benefits of the robust regression approach in a new application to outlier detection. 
Subsequently, Section~\ref{sec:conclusions} summarises all results.


\section{Least Trimmed Squares as Difference of Convex Functions}\label{sec:LTS_as_DC}



In this section, a powerful \gls{DC} representation for the classical \gls{LTS} problem is derived, which, to the best of our knowledge, has not been considered before. 
For this purpose, we recall the steps proposed by \citet{GILONI2002} to transform the \gls{LTS} problem into a concave minimisation over a polytope. 
At the beginning, we observe that one of the most natural options to reformulate problem~\eqref{eq:Classic_LTS} as a mathematical program with constraints is the introduction of indicator variables $\mathbf{z} \coloneqq \begin{bmatrix} z_1, \ldots, z_n\end{bmatrix}^\top \in \{0, 1\}^n$. 
These variables then can be used to construct a \gls{MIPP} of the form
\begin{align}\label{eq:MIP}
    \begin{split}
        \underset{\boldsymbol{\beta}, \mathbf{z}}{\text{minimise}} \quad & \Bar{f}(\mathbf{z}, \boldsymbol{\beta}) := \sum_{i=1}^n z_i \cdot (y_i - \mathbf{x}_i^\top \boldsymbol{\beta})^2 \\
        \text{subject to} \quad & \sum_{i=1}^n z_i = h, \\
         & z_i \in \{0, 1\} \quad \forall i \in \{1, \ldots, n\},
    \end{split}
\end{align}
where $\mathbf{z}$  determines which residuals are included in the summation, and $\Bar{f}: \mathbb{R}^n \times \mathbb{R}^p \rightarrow \mathbb{R}_+$ is a third-order multivariate polynomial that is bounded from below by $0 \leq \Bar{f}(\mathbf{z}, \boldsymbol{\beta})$. 
In general, such problem types are NP-hard and therefore are  difficult to solve exactly [cf. \citet[pp.~205--208]{ZHANG2023}]. 
However, for the specific \gls{MIPP} presented in \eqref{eq:MIP}, it can be shown that relaxing the integrality constraint does not alter the optimal solution, as we will present next. 


In order to transform problem \eqref{eq:MIP} into a \gls{CNPP}, we replace the integer constraint $z_i \in \{0, 1\}$ with two inequality constraints, $0 \leq z_i $ and $ z_i \leq 1$ for all $i \in \{1, \ldots, n\}$,  such that $z_1, \ldots, z_n$ become continuous variables on the interval $[0, 1]$. 
The resulting mathematical program then has the form 
\begin{align}\label{eq:CNLPP}
    \begin{split}
        \underset{\boldsymbol{\beta},  \mathbf{z}}{\text{minimise}} \quad & \Bar{f}   \left(\mathbf{z}, \boldsymbol{\beta} \right) \\
        \text{subject to} \quad &  \mathbf{z} \in \Delta_h,
    \end{split}
\end{align}
where 
\begin{equation*}
    \Delta_h := \Bigl\{ \mathbf{z} = \begin{bmatrix} z_1, \ldots, z_n \end{bmatrix}^\top \in [0, 1]^n  \ \bigl| \  \textstyle \sum_{i=1}^n z_i = h \Bigr\}
\end{equation*}
denotes the feasible set.
The latter set, known as the capped simplex, consists of $n$ box constraints restricting all weights $z_i$ to the closed interval $[0, 1]$, and one equality constraint ensuring that the weights sum up to $h$ [cf. Appendix \ref{sec:feasible_set} for a visualisation]. 
For this specific \gls{CNPP}, \citet[p. 3215]{NGUYEN2010} showed that all local and global solutions $(\mathbf{z}^*, \boldsymbol{\beta}^*)$ lie at extreme points of the feasible region.
A central component of their proof is the observation that the solution set of the \gls{MIPP} is a subset of that of the \gls{CNPP}, which implies that the optimal objective function value of the \gls{MIPP} cannot be lower than that of the \gls{CNPP}. 
Furthermore, any non-integer solution to the \gls{CNPP} can be transformed into an integer solution by selecting the $h$ smallest squared residuals. 
Note that in the presence of ties, multiple such subsets exist, and any of them can be chosen. 
Since the transformation involves selecting the smallest $h$ residuals, the resulting objective function value is necessarily less than or equal to that of the original non-integer solution. 
Based on these results, it then follows that the optimal solutions of the \gls{MIPP} and \gls{CNPP} must coincide and it can be established that the integrality constraints are not required to obtain an optimal solution to the problem. 
However, while this relaxation yields a continuously differentiable objective function, the structure remains non-concave, requiring an additional transformation step.


To ultimately convert problem~\eqref{eq:CNLPP} into a concave minimisation, we must also reduce the number of decision variables in the objective function $\Bar{f}$. 
For this purpose, we reformulate the problem based on an inner and outer minimisation of the form
\begin{align}\label{eq:double_minimisation}
    \begin{split}
        \underset{\mathbf{z} \in \Delta_h}{\text{minimise}} \  \ & \biggl ( \underset{\boldsymbol{\beta}}{\text{minimise}} \sum_{i=1}^n z_i \cdot (y_i - \mathbf{x}_i^\top \boldsymbol{\beta})^2 \biggl ),
    \end{split}
\end{align}
where the inner problem corresponds to a quadratic function that can be solved in closed form [cf. \citet[pp. 81--82]{boyd2009}]. 
In particular, let us assume that for the outer problem the solutions for $z_1, \ldots, z_n$ are already known and are stored in the diagonal matrix $\mathbf{Z} := \text{diag}(\mathbf{z}) = \text{diag}(z_1, \ldots, z_n) \in \mathcal{M}_{n \times n}(\mathbb{R})$. 
Then, the inner minimisation problem of \eqref{eq:double_minimisation} can be rewritten as
\begin{equation*}
  \underset{\boldsymbol{\beta}}{\text{minimise}} \quad (\mathbf{y} - \mathbf{X} \boldsymbol{\beta})^\top \mathbf{Z} (\mathbf{y} - \mathbf{X} \boldsymbol{\beta}),
\end{equation*}
which has the form of a \gls{WLS} problem [cf. \citet[pp.~193--195]{Fahrmeir2021}]. 
This implies that there exists a closed-form solution for the regression coefficients, given by $ \widehat{\boldsymbol{\beta}}_{\text{\gls{WLS}}} := (\mathbf{X}^\top \mathbf{Z} \mathbf{X})^{-1} \mathbf{X}^\top \mathbf{Z}\mathbf{y} \in \mathbb{R}^p$, where $(\mathbf{X}^\top \mathbf{Z} \mathbf{X})^{-1}$ denotes the inverse matrix, which exists if $\mathbf{X}^\top \mathbf{Z} \mathbf{X}$ is nonsingular. 
Note that if the latter condition is not satisfied, the Moore-Penrose pseudoinverse can be used instead. 
The closed-form solution then allows us to reduce the number of input variables in the objective function, resulting in $f: \mathbb{R}^n \rightarrow \mathbb{R}_+$ of the form 
\begin{align}\label{eq:concave_objective_func}
    \begin{split}
        f \left(\mathbf{z} \right) &:= \sum_{i=1}^n z_i \cdot (y_i - \mathbf{x}_i^\top (\mathbf{X}^\top \mathbf{Z} \mathbf{X})^{-1} \mathbf{X}^\top \mathbf{Z}\mathbf{y})^2 
        = \mathbf{y}^\top \mathbf{Z} \mathbf{y} - \mathbf{y}^\top \mathbf{Z} \mathbf{X} (\mathbf{X}^\top \mathbf{Z}\mathbf{X})^{-1} \mathbf{X}^\top \mathbf{Z} \mathbf{y},
    \end{split}
\end{align}
which is a concave function depending only on $\mathbf{z}$ [cf. Definition~\ref{defi:convex_function} and Theorem~\ref{theorem:concave_objective}]. 
Although concave minimisation problems generally remain NP-hard -- and even seemingly simple cases can have an exponential number of local minima -- they are typically more tractable than general non-convex optimisation problems, due to special mathematical properties [cf. \citet[p. 43]{Horst1995}]. 
A variety of methods have been proposed to approximately solve them, among which \gls{DC} programming has become a popular choice due to its computational efficiency.


\vspace{0.2cm}

\begin{defi}[{Convex Function {\small [cf. \citet[ch. 3]{boyd2009}]}}]{defi:convex_function}
    A function $f: \mathbb{R}^n \rightarrow \mathbb{R}$ is said to be \textbf{convex} (with modulus $\rho = 0$) or \textbf{strongly convex} (with modulus $\rho > 0$) if
        \begin{equation*}
            f(\lambda \cdot \mathbf{z}_1 + (1 - \lambda) \cdot \mathbf{z}_2) \leq \lambda \cdot f(\mathbf{z}_1) + (1 - \lambda) \cdot f(\mathbf{z}_2) - \frac{\rho}{2} \cdot \lambda \cdot (1 - \lambda) \cdot \Vert \mathbf{z}_1 - \mathbf{z}_2 \rVert^2
        \end{equation*}
    holds for all choices of  $\mathbf{z}_1, \mathbf{z}_2 \in \mathbb{R}^n$  and $\lambda \in (0, 1)$. In contrast, a function $f$ is said to be \textbf{(strongly) concave} if $-f$ is (strongly) convex.
\end{defi}


\begin{theo}{theorem:concave_objective}
    The objective function $f$ shown in \eqref{eq:concave_objective_func} is concave, twice continuously differentiable and bounded from below by zero.
\end{theo}
\begin{prf}{prf:concave_objective}
    Different approaches to prove Theorem~\ref{theorem:concave_objective} are shown in \citet[p. 1838]{Shen2013}, \citet[p. 3216]{NGUYEN2010} or \citet[pp. 81--82] {boyd2009}. 
\end{prf}

\vspace{-0.2cm}


In order to apply \gls{DC} programming to approximately solve the \gls{LTS} problem, it remains to derive a corresponding \gls{DC} representation for the objective function given in~\eqref{eq:concave_objective_func}. 
In theory, every twice continuously differentiable function defined over a convex set can be decomposed into a \gls{DC} form [cf.~\citet[pp. 709--711]{hartman1959}]. 
However, such decompositions are not uniquely defined and must be individually derived, which can be a challenging task [cf. \citet[p.~534]{leThi2024}; \citet[p.~28]{LeThi2005}; \citet[p.~291]{PhamDinh1997}]. 
Furthermore, the choice of decomposition is crucial for the performance of \gls{DC} programming algorithms [cf.~\citet[pp. 545--546]{leThi2024}; \citet[p. 8]{PhamDinh2014}; \citet[p. 291]{PhamDinh1997}]. 
For example, certain reformulations can significantly improve computation time or the found objective function value [cf. \citet[pp. 545--546]{leThi2024}]. 
Unfortunately, identifying a favorable \gls{DC} decomposition remains one of the major challenges in \gls{DC} programming and largely depends on the specific structure of the problem at hand [cf. \citet[pp. 545--546; 551]{leThi2024}]. 
In general, it is recommended to select a \gls{DC} decomposition that ensures the convex subproblem -- constructed within any \gls{DC} programming algorithm -- can be efficiently solved [cf. \citet[pp. 405--406]{LeThi2000}]. 


For the concave function shown in~\eqref{eq:concave_objective_func}, there appear to be several natural candidates for constructing a \gls{DC} decomposition of the form 
\begin{equation*}
    f \left(\mathbf{z} \right) = \phi_{1} \left( \mathbf{z} \right) - \phi_{2} \left( \mathbf{z} \right),
\end{equation*} 
where $\phi_1: \mathbb{R}^n \rightarrow \mathbb{R} \cup \{+ \infty\}$ and $\phi_2: \mathbb{R}^n \rightarrow \mathbb{R} \cup \{+ \infty\}$ are proper, closed, and strongly convex functions. 
For example, by introducing a parameter $\rho \in \mathbb{R}_{>0}$ and defining the auxiliary function $\psi \left( \mathbf{z} \right) := \mathbf{z}^\top \mathbf{Q} \mathbf{z}$, with $\mathbf{Q} \in \mathcal{M}_{n\times n}\left( \mathbb{R} \right)$ and $\mathbf{Q} \succeq \rho \cdot \mathbf{I}_n$, three general types of \gls{DC} decompositions can be constructed, with components defined as
\begin{alignat}{2}
    &\text{1. Option:} \ \ \phi_1(\mathbf{z}):= \psi \left( \mathbf{z} \right) + \mathbf{y}^\top \mathbf{Z} \mathbf{y}  &&\quad \text{and} \quad \phi_2(\mathbf{z}) := \psi \left( \mathbf{z} \right) + \mathbf{y}^\top \mathbf{Z} \mathbf{X} (\mathbf{X}^\top \mathbf{Z}\mathbf{X})^{-1} \mathbf{X}^\top \mathbf{Z} \mathbf{y}, \label{eq:dc_1} \\
    &\text{2. Option:}  \ \  \phi_1(\mathbf{z}):= \psi \left( \mathbf{z} \right)  + f(\mathbf{z})  &&\quad \text{and} \quad \phi_2(\mathbf{z}) := \psi \left( \mathbf{z} \right), \\
    &\text{3. Option:}  \  \  \phi_1(\mathbf{z}):= \psi \left( \mathbf{z} \right)    &&\quad \text{and} \quad \phi_2(\mathbf{z}) := \psi \left( \mathbf{z} \right) - f(\mathbf{z}), \label{eq:dc_3}
\end{alignat}
where each pair $(\phi_1, \phi_2)$ satisfies the desired convexity properties if $\mathbf{Q}$ and $\rho$ are chosen appropriately for the corresponding decomposition. 
It is important to note that these \gls{DC} decompositions differ substantially from the formulation proposed by \citet[p. 93]{Liu2019}, in which the $p$ regression coefficients $\beta_1, \ldots, \beta_p$ are treated as decision variables. 
In our practical implementation, we evaluated the corresponding \gls{DC} programs of the form 
\begin{align}\label{eq:general_problem}
    \begin{split}
             \underset{\mathbf{z} \in \Delta_h}{\text{minimise}} \ \ & \phi_{1} (\mathbf{z}) - \phi_{2} (\mathbf{z}),
    \end{split}
\end{align}
resulting from the options shown in~\eqref{eq:dc_1}–\eqref{eq:dc_3} under various choices of $\mathbf{Q}$. 
For the upcoming sections, we ultimately selected the third decomposition with $\mathbf{Q} = \rho \cdot \mathbf{I}_n$, primarily due to the computational efficiency it offers when solving the convex subproblem
\begin{align}\label{eq:general_subproblem}
    \begin{split}
             \underset{\mathbf{z} \in \Delta_h}{\text{minimise}} \ \ & \phi_{1} (\mathbf{z}) - \Bigl ( \phi_2 \left( \mathbf{z}_t \right) + \left(\mathbf{z} - \mathbf{z}_t \right)^\top \nabla \phi_{2} (\mathbf{z}_t) \Bigr),
    \end{split}
\end{align}
where $\mathbf{z}_t \in \Delta_h$ is the solution of the previous iteration. 
Specifically, this choice enables the use of the fast projection procedure introduced by \citet{ang2021}, which is significantly faster than state-of-the-art commercial solvers such as Gurobi.


\section{DC Programming for Least Trimmed Squares}\label{sec:DC_Programming_for_LTS}



After constructing a linearly constrained \gls{DC} program for the \gls{LTS} problem, this section derives the theoretical setup of our novel \gls{DC} programming approach, called \gls{sBDCA}. 
From a theoretical point of view, the algorithm combines the advantages of \gls{BDCA} and \gls{DCA} with successive \gls{DC} decompositions [cf. \citet{Thormann2026b}; \citet{aragon2022}; \citet{HO2021}; \citet{aragon2020}; \citet{Ho2020};  \citet{leThi2019}; \citet{aragon2018}; \citet{leThi2018}]. 
In comparison to \eqref{eq:general_problem}, \gls{sBDCA} is based on an iteration-specific \gls{DC} program of the form
\begin{align}\tag{$\mathcal{P}_t$}\label{eq:LTS_as_DC}
    \begin{split}
        \underset{\mathbf{z} \in \Delta_h}{\text{minimise}} \quad &  \phi_{1, t}(\mathbf{z}) - \phi_{2, t}(\mathbf{z}), \\
    \end{split}
\end{align} 
where $t \in \{0,\ldots, T\}$ denotes the iteration number with $T \in \mathbb{N}$.
The components of the different \gls{DC} decomposition are given by $\phi_{1, t} \left(\mathbf{z} \right) := \rho_t \cdot \lVert \mathbf{z} \rVert^2$ and $\phi_{2, t} \left(\mathbf{z} \right) := \rho_t \cdot \lVert \mathbf{z} \rVert^2 - f \left(\mathbf{z}\right) $ with varying convexity parameter $\rho_t \in \mathbb{R}_{>0}$. 
At each iteration, this setup is then used to construct a convex approximation of \eqref{eq:LTS_as_DC}, whose solution is utilised to perform a simplified line search afterwards. 
In the first subsection, the proposed optimisation procedure is described in more detail, and the convergence rate to a local solution of the \gls{LTS} problem is derived. 
The second subsection then compares the influence of fixed and successive \gls{DC} decompositions on the algorithmic output. 
In the last subsection, different initialisation strategies are presented, and a preconditioner for the gradient of the objective function is proposed to reduce the variability of the solution quality and to find solutions with smaller objective function values. 



\subsection{The successive Boosted Difference of Convex Functions Algorithm}\label{sec:sBDCA}


Building on the \gls{DC} formulation of the \gls{LTS} problem in \eqref{eq:LTS_as_DC},  Algorithm~\ref{alg:sBDCA} shows the procedure of \gls{sBDCA} for the specific problem at hand.
To initialise the algorithm, the user must provide a feasible starting point as well as parameters for the \gls{DC} decompositions and the termination condition. 
Afterwards, the \gls{sBDCA} can be applied, where each iteration is separated into two parts: (i) the \gls{DCA} step and (ii) the line search procedure.
In general, both parts build on each other and are iteratively executed until the algorithm reaches convergence.
At each iteration, the \gls{DCA} step is based on solving a strongly convex subproblem of the form
\begin{align}\tag{$\tilde{\mathcal{P}}_{t}$}\label{eq:subproblem_lts}
    \begin{split}
             \underset{\mathbf{z} \in \Delta_h}{\text{minimise}} \ \ & \Tilde{f}_{t} \left( \mathbf{z} \right) :=  \phi_{1, t} (\mathbf{z}) - \Bigl( \phi_{2, t}(\mathbf{z}_{t}) + (\mathbf{z} - \mathbf{z}_{t}) ^\top \nabla \phi_{2, t} (\mathbf{z}_{t}) \Bigr),
    \end{split}
\end{align}
to obtain the unique solution $\Tilde{\mathbf{y}}_{t} \in \mathbb{R}^n$.
This subproblem is an approximation of the original objective in \eqref{eq:LTS_as_DC}, where the concave part is replaced with its affine majorisation at the $t^{\text{th}}$ iteration point~$\mathbf{z}_{t}$.
In Algorithm~\ref{alg:sBDCA}, the unique solution $\Tilde{\mathbf{y}}_{t}$ is obtained by solving
\begin{align}\label{eq:projection_onto_capped_simplex}\tag{$\mathcal{P}_{\Delta_h}$}
 \begin{split}
      P_{\Delta_h} \left ( \Tilde{\mathbf{z}}_{t} \right) := \underset{\mathbf{z} \in \Delta_h}{\text{argmin}} \ \  \tfrac{1}{2} \cdot \bigl \lVert \mathbf{z} - \Tilde{\mathbf{z}}_{t}  \bigr \rVert^2,
 \end{split}
\end{align}
where $\Tilde{\mathbf{z}}_{t} := \mathbf{z}_{t} - \eta_{t} \cdot \nabla f(\mathbf{z}_{t})$ with $\eta_{t} := \frac{1}{2 \cdot \rho_{t}}$.
This projection of $\Tilde{\mathbf{z}}_{t}$ onto the capped simplex can be efficiently computed using algorithms from the literature and also yields the unique solution~$\Tilde{\mathbf{y}}_{t}$.
Specifically, Appendix~\ref{sec:projection_onto_capped_simplex} shows that the optimal solution of \eqref{eq:projection_onto_capped_simplex} coincides with that of~\eqref{eq:subproblem_lts}, and provides further details regarding the projection step.



\begin{algorithm}[H]
    \vspace{0.1cm}
  \KwInput{Initial $\mathbf{z}_0 \in \Delta_h$, $\varepsilon \geq 0$, $T \in \mathbb{N}$, $\rho_0, \rho_{\text{min}}, \rho_{\text{max}}, \nu > 0$} 
  \vspace{0.1cm}
  \For{$t = 0,\ldots,T-1$}
   {
   \vspace{0.1cm}
   $\eta_{t} \gets \frac{1}{2 \cdot \rho_{t}}$ \\
   \vspace{0.1cm}
   $\Tilde{\mathbf{y}}_{t} \gets P_{\Delta_h} \Bigl( \mathbf{z}_{t} - \eta_{t} \cdot \nabla f(\mathbf{z}_{t}) \Bigr)$ \\
    \vspace{0.1cm} \label{line:DCA_Step}
    $\mathbf{d}_{t} \gets \Tilde{\mathbf{y}}_{t} - \mathbf{z}_{t}$ \label{line:LS_begin} \\
    \vspace{0.1cm}
    \If{$\lVert \mathbf{d}_{t} \rVert \leq \varepsilon$}
    {\vspace{0.1cm}
    \textbf{break}
    }
    \vspace{0.1cm}
    $\bar{\lambda}_{t} \gets  \underset{( \Tilde{\mathbf{y}}_{t} + \bar{\lambda} \cdot \mathbf{d}_{t}  ) \in \Delta_h}{\text{arg max}} \  \bar{\lambda}$ \\
    \vspace{0.1cm}
    $\mathbf{z}_{t + 1} \gets \Tilde{\mathbf{y}}_{t} + \bar{\lambda}_{t} \cdot \mathbf{d}_{t}$ \label{line:LS_end} \\
    \vspace{0.1cm}
    $\rho_{t + 1} \gets \min \Bigl \{ \max  \bigl( \nu \cdot \rho_{t}, \rho_{\text{min}} \bigr), \rho_{\text{max}} \Bigr \} $ \label{line:rho_{t}}\\
   }
    \vspace{0.1cm}
   \KwOutput{$\mathbf{z}_{t}$}
   \vspace{0.1cm}
\caption{\acrfull{sBDCA}}\label{alg:sBDCA}
\end{algorithm} 



Once the unique solution $\Tilde{\mathbf{y}}_{t}$ has been computed, it is used to perform the simplified line search.
For this purpose, a descent direction $\mathbf{d}_{t} := \Tilde{\mathbf{y}}_{t} - \mathbf{z}_{t}$ is calculated. 
Subsequently, the line search is executed along this direction, where the additional parameter~$\Bar{\lambda}_{t}$ is determined using the boundary of the feasible set and the decision variables are directly updated based on $\mathbf{z}_{t + 1} = \Tilde{\mathbf{y}}_{t} + \bar{\lambda}_{t} \cdot \mathbf{d}_{t}$.
In comparison to previous \gls{BDCA} versions, this update step has changed as it no longer scales down $\bar{\lambda}_{t}$ based on a backtracking procedure using a Armijo-type condition of the form
$ f\left( \Tilde{\mathbf{y}}_{t} + \bar{\lambda}_{t} \cdot \mathbf{d}_{t} \right) > f\left( \Tilde{\mathbf{y}}_{t} \right) - \bar{\alpha} \cdot \bar{\lambda}_{t}^2 \cdot \lVert \mathbf{d}_{t} \rVert^2$, where $\bar{\alpha} \in \mathbb{R}_{>0}$.
However, since the objective function in \eqref{eq:concave_objective_func} is concave, the following result can be obtained
\begin{align}\label{eq:rho_in_Armijo}
    \begin{split}
        f \left(\Tilde{\mathbf{y}}_{t} + \bar{\lambda}_{t} \cdot \mathbf{d}_{t} \right) &\leq f \left(\Tilde{\mathbf{y}}_{t} \right) +   \left( \Tilde{\mathbf{y}}_{t} + \bar{\lambda}_{t} \cdot \mathbf{d}_{t} - \Tilde{\mathbf{y}}_{t} \right)^\top \nabla f \left( \Tilde{\mathbf{y}}_{t} \right) \\
        &= f \left(\Tilde{\mathbf{y}}_{t} \right) + \bar{\lambda}_{t}   \cdot \nabla_{\mathbf{d}_{t}} f \left( \Tilde{\mathbf{y}}_{t} \right) \\
        &\leq f \left(\Tilde{\mathbf{y}}_{t} \right) - 2 \cdot \bar{\lambda}_{t}  \cdot \rho_{t} \cdot \lVert \mathbf{d}_{t} \rVert^2,
    \end{split}
\end{align}
where the first inequality is derived from the concavity of $f$, and the last inequality follows from $\mathbf{d}_{t}$ being a descent direction [cf. Theorem~\ref{theo:algorithmic_properties_sBDCA}]. 
This shows that for all choices of $\bar{\alpha}$ satisfying $0 < \bar{\alpha} \leq 2 \cdot \rho_{t} \cdot (\bar{\lambda}_{t})^{-1}$, the line search procedure of previous \gls{BDCA} versions directly terminates at the boundary of the feasible set. 
Based on these insights, we decided to use the simplified update step for Algorithm~\ref{alg:sBDCA}, as it can be performed free of computational cost and ensures the maximum possible reduction of the objective along the descent direction $\mathbf{d}_{t}$.


After proposing the general procedure of \gls{sBDCA}, we now derive algorithmic properties and deduce the convergence rate of the sequence $\{ \mathbf{z}_{t} \}$ to a local solution of the \gls{LTS} problem.
Note that as the objective function is concave, convergence to a local solution is the best that can be expected for the problem at hand [cf. \citet[p. 1048]{GILONI2002}]. 
Despite this limitation, the convergence rate remains a key performance indicator of an algorithm [cf. \citet[p. 619]{Nocedal2006}]. 
For the state-of-the-art, Fast-\gls{LTS}, \citet{rousseeuw2006} neither derived an explicit convergence rate nor proved convergence to a local solution, but concluded that the heuristic terminates after a finite number of steps.
In \gls{DC} programming, the convergence rate remains an active area of research, as it must be assessed on a case-by-case basis depending on problem-specific properties [cf. \citet[p. 53]{leThi2018}].
To establish the convergence rate of \gls{sBDCA} for the \gls{LTS} problem, we leverage the results of Theorem~\ref{theo:cayley_hamilton}, which states that every square matrix satisfies its own characteristic polynomial.
The result is also known as the Cayley-Hamilton Theorem, and can be used, for example, to compute powers of square matrices or to solve binomial matrix equations [cf. \citet{Pop2017}].
In this paper, we use Theorem~\ref{theo:cayley_hamilton} to derive an alternative expression of $(\mathbf{X}^\top \mathbf{Z} \mathbf{X})^{-1}$, which in turn allows us to establish that the objective in \eqref{eq:concave_objective_func} is a multivariate rational function (i.e., a fraction of two polynomials).


\vspace{0.2cm}

\begin{theo}[{Cayley–Hamilton {\small [cf. \citet[p. 328]{Brualdi1991}]}}]{theo:cayley_hamilton}
    Let $\mathbf{A} \in \mathcal{M}_{{p \times p}}(\mathbb{R})$ be a square matrix with $p \in \mathbb{N}$, and $b_0, \ldots, b_{p-1} \in \mathbb{R}$. Then it holds that
    \begin{equation*}
        \mathbf{A}^{p} + \sum_{i=1}^{p-1} b_{i} \cdot \mathbf{A}^{i} + b_0 \cdot \mathbf{I}_p = \mathbf{0}_{p \times p} \in \mathcal{M}_{{p \times p}}(\mathbb{R}),
    \end{equation*}
    where $p_{\mathbf{A}}(\lambda) := \text{det} \left( \lambda \cdot \mathbf{I}_p - \mathbf{A} \right) = \lambda^p + \sum_{i=1}^{p-1} b_{i} \cdot \lambda^{i} + b_0$ is the characteristic polynomial of~$\mathbf{A}$.
\end{theo}

\vspace{-0.2cm}


Building on the result obtained via the Cayley-Hamilton Theorem, we then demonstrate that the objective also satisfies the requirements of a real-analytic function.
As specified in Definition~\ref{defi:real_analytic}, functions of this class have the property that at every point in their domain, there exists a locally convergent power series representation. 
This is an important property, as it allows for a precise analysis of the convergence behaviour of algorithms in mathematical optimisation, especially those relying on local information such as gradients.
In particular, the property implies that the function of interest also satisfies the \L{}ojasiewicz property [cf. \citet[p. 67]{Lojasiewicz1965}].
The latter is formalised in Definition~\ref{defi:lojasiewicz_property}, and plays a central role in the convergence analysis of descent algorithms, as it characterises the rate at which the function's gradient vanishes near a critical point. 
Specifically in \gls{DC} programming, the \L{}ojasiewicz property has a high relevance, since the convergence rates of \gls{DC} algorithms are typically established depending on the \L{}ojasiewicz exponent [cf. \citet[Theorem~1]{aragon2018}]. 


\vspace{0.2cm}

\begin{defi}[{Real-Analytic Functions {\small [cf. \citet[pp. 26--29]{Lewis2023}]}}]{defi:real_analytic}
   The function $f: \mathbb{R}^n \rightarrow \mathbb{R}$ is called \textbf{real-analytic} if at each point $\mathbf{u}_0 \in \mathbb{R}^n$  there exists a convergent power series that converges (absolutely) to $f$ for all $\mathbf{u} \in \mathbb{B} \left ( \mathbf{u}_0, \varepsilon \right)$ with $\varepsilon > 0$.
\end{defi}


\begin{defi}[{\L ojasiewicz Property {\small [cf. \citet[p. 99]{aragon2018}]}}]{defi:lojasiewicz_property}
Let $\varphi: \mathbb{R}^n \rightarrow \mathbb{R}$ be a differentiable function. If for any critical point $\mathbf{z}^* \in \mathbb{R}^n$, there exist constants $M \in \mathbb{R}_{> 0}, \varepsilon \in \mathbb{R}_{> 0}$, and $\vartheta \in [0, 1)$ such that
\begin{align*}
    \begin{split}
        \bigl \lvert \varphi \left(\mathbf{z} \right) - \varphi \left(\mathbf{z}^* \right) \bigr \rvert^{\vartheta} \leq M \cdot \bigl \lVert \nabla \varphi \left( \mathbf{z} \right) \bigr \rVert \quad \quad \text{for all} \ \mathbf{z} \in \mathbb{B} \left(\mathbf{z}^*, \varepsilon \right), 
    \end{split}
\end{align*}
then the function $\varphi$ possesses the \textbf{\L ojasiewicz property} with \textbf{\L ojasiewicz exponent} $\vartheta_{\varphi}$.
\end{defi}

\vspace{-0.2cm}


Based on the preceding results, Theorem~\ref{theo:algorithmic_properties_sBDCA} establishes algorithmic properties of \gls{sBDCA} and the convergence rate of the sequence $\{ \mathbf{z}_{t} \}$ to a local solution of the \gls{LTS} problem.
In particular, for the fifth bullet point, the \L{}ojasiewicz property and exponent $\vartheta_f \in [\frac{1}{2}, 1)$ are then used to derive a linear and sublinear rate in the fastest and slowest case, respectively. 


\vspace{0.2cm}

\begin{theo}[Algorithmic Properties of \gls{sBDCA}]{theo:algorithmic_properties_sBDCA}
    The sequences $\{ \mathbf{z}_{t} \}$ and $\{ \Tilde{\mathbf{y}}_{t} \}$ generated by Algorithm \ref{alg:sBDCA} satisfy the following properties 
        \begin{itemize}
        \item [(i)] For all $t \ge 0$, the strongly convex subproblem yields
        $$f(\Tilde{\mathbf{y}}_{t}) \le f(\mathbf{z}_{t}) - 2 \cdot \rho_{t} \cdot \lVert \mathbf{d}_{t}\rVert^2.$$
        \item [(ii)] The simplified line search guarantees
        $$f(\mathbf{z}_{t + 1}) 
             \le f(\Tilde{\mathbf{y}}_{t}) - 2 \cdot \rho_{t} \cdot \lambda_{t}\cdot \lVert \mathbf{d}_{t}\rVert^2.$$
        \item[(iii)] The sequence $\bigl \{f \left( \mathbf{z}_{t} \right) \bigr\}$ is monotonously decreasing and convergent to some $f^*$.
        \item [(iv)] There exists a sub-sequence of $\bigl \{ \mathbf{z}_{t} \bigr\}$ converging to a \gls{KKT} point of \eqref{eq:LTS_as_DC}. In addition, if there are no ties around the $h$th and ($h + 1$)th squared residual, then this \gls{KKT} point corresponds to a local solution of \eqref{eq:LTS_as_DC}. 
        \item[(v)] Let $\mathbf{z}^* \in \Delta_h$ be a limit point of $\{\mathbf{z}_{t}\}$. 
        Since the objective function $f$ satisfies the \L{}ojasiewicz property with exponent $\vartheta_f \in [\frac{1}{2}, 1)$, the sequence $\{\mathbf{z}_{t}\}$ converges to $\bar{\mathbf{z}}$ at a \textbf{linear rate} in the fastest case when $\vartheta_f = \frac{1}{2}$ and at a \textbf{sublinear rate} in the slowest case when $\frac{1}{2} < \vartheta_f < 1$.
        \end{itemize}
    \vspace{0.05cm}
\end{theo}
\begin{prf}{prf:convergence_rate}

    \begin{enumerate}
        \item [(i)]
            To prove the first statement, we observe that the orthogonal projection of $\mathbf{z}_{t} - \eta_{t} \cdot \nabla f(\mathbf{z})$ implies [cf. \citet[p. 201]{bertsekas1999}] that  
            \begin{equation}\label{eq:dk_descent_direction}
            \nabla_{\mathbf{d}_{t}} f(\mathbf{z}_{t}) = \mathbf{d}_{t}^\top \nabla f \left(\mathbf{z}_{t} \right)  \leq - \frac{1}{\eta_{t}} \cdot \lVert \mathbf{d}_{t} \rVert^2 = - 2 \cdot \rho_{t} \cdot \lVert \mathbf{d}_{t} \rVert^2\leq 0,
            \end{equation}
            where $\nabla_{\mathbf{d}_{t}} f(\mathbf{z}_{t})$ denotes the one-sided directional derivative of $f$ at $\mathbf{z}_{t}$ with respect to the direction $\mathbf{d}_{t}$, formally defined as
            \begin{equation*}
                \nabla_{\mathbf{d}_{t}} f \left(\mathbf{z}_{t} \right) := \lim_{t \rightarrow 0} \frac{f\left(\mathbf{z}_{t} + t \cdot \mathbf{d}_{t} \right) - f \left(\mathbf{z}_{t} \right)}{t} = \mathbf{d}_{t}^\top \nabla f \left(\mathbf{z}_{t} \right).
            \end{equation*}
            The result shown in \eqref{eq:dk_descent_direction} then proves that $\mathbf{d}_{t}$ is a descent direction at $\nabla f(\mathbf{z}_{t})$, and can be combined with the fact that $f$ is a concave function to obtain 
            \begin{equation}
            \mathbf{d}_{t}^\top \nabla f \left(\Tilde{\mathbf{y}}_{t} \right)  \leq \nabla \mathbf{d}_{t}^\top \nabla f \left(\mathbf{z}_{t} \right) \leq - 2 \cdot \rho_{t} \cdot \lVert \mathbf{d}_{t} \rVert^2 \leq 0,
            \end{equation}
            that proves that $\mathbf{d}_{t}$ is also a descent direction at $\nabla f(\Tilde{\mathbf{y}}_{t})$. 
            The previous results can be further combined with the fact that the first-order Taylor approximation of a concave function is a global overestimator [cf. \citet[ch. 3]{boyd2009}] to receive
            \begin{align*}
            f \left(\Tilde{\mathbf{y}}_{t} \right) &\leq f \left(\mathbf{z}_{t} \right) + \mathbf{d}_{t}^\top \nabla f \left(\mathbf{z}_{t} \right)\\
                        &\leq f(\mathbf{z}_{t}) - 2 \cdot \rho_{t} \cdot \lVert \mathbf{d}_{t} \rVert^2,
            \end{align*}
            which shows that the \gls{DCA} step on its own generates a monotonously decreasing sequence of objective function values.
            
        \item [(ii)]
        As we have already shown that $\mathbf{d}_{t}$ is a descent direction at $\nabla f(\Tilde{\mathbf{y}}_{t})$, we can similarly obtain the inequality
        \begin{equation}\label{eq:line_search}
            f \left(\mathbf{z}_{t + 1} \right) = f \left(\Tilde{\mathbf{y}}_{t} + \bar{\lambda}_{t} \cdot \mathbf{d}_{t} \right) \leq f \left(\Tilde{\mathbf{y}}_{t} \right) - 2 \cdot \bar{\lambda}_{t} \cdot \rho_{t} \cdot \lVert \mathbf{d}_{t} \rVert^2,
        \end{equation}
        showing that the simplified line search decreases the objective function if $\bar{\lambda}_{t} \in (0, \infty)$.
        
        \item [(iii)] 
        Based on the first and second statement, we can then deduce that the sequence $\{ f \left( \mathbf{z}_{t} \right) \}$ is monotonously decreasing and convergent to some $f^*$, i.e.,
        \begin{equation*}
        f^* \leq \ldots \leq f(\mathbf{z}_{t + 1}) = f \left(\Tilde{\mathbf{y}}_{t} + \bar{\lambda}_{t} \cdot \mathbf{d}_{t} \right) \leq f(\mathbf{z}_{t}) \leq \ldots \leq f(\mathbf{z}_0)
        \end{equation*}
        as the objective function $f$ is bounded from below [cf. Theorem \ref{theorem:concave_objective}].
        
        \item [(iv)]
        To prove the fourth statement, we first observe that $\mathbf{z}_{t} \in \Delta_h $ for all $t\ge 0$.
        Moreover, the sequence $\bigl \{ \mathbf{z}_{t} \bigr\}$ is bounded, since $\Delta_h$ is a compact set (i.e., closed and bounded).
        As a result, there exists a sub-sequence $\bigl \{ \mathbf{z}_{t_k} \bigr\}$ of $\bigl  \{ \mathbf{z}_{t} \bigr\}$ converging to some $\mathbf{z}^*$. 
        For this subsequence, we then can write down the \gls{KKT} conditions of the strongly convex subproblem \eqref{eq:subproblem_lts}, as
            \begin{align}\label{eq:KKT_cond}
            \begin{cases}
               \nabla \phi_{1,t_k}(\Tilde{\mathbf{y}}_{t_k})  + \sum_{i=1}^{2n} \mu_{t_k,i} \cdot \nabla g_i(\Tilde{\mathbf{y}}_{t_k}) + \lambda_{t_k} \cdot \mathbf{1}_n = \nabla \phi_{2,t_k}(\mathbf{z}_{t_k}), \\
                \mu_{t_k,i} \cdot g_i(\Tilde{\mathbf{y}}_{t_k}) = 0, \mu_{t_k,i} \geq 0, \ g_i(\Tilde{\mathbf{y}}_{t_k}) \leq 0, \ i = 1, \ldots 2 \cdot n, \\
                \mathbf{1}_n^\top \Tilde{\mathbf{y}}_{t_k} = h,
            \end{cases}
            \end{align}
            where $\mu_{t_k,i} > 0$ and $\lambda_{t_k} \in \mathbb{R}$ are the iteration-specific Lagrange multiplier of the subsequence.
            The inequality constraint functions are given by
            \begin{align*}
                \begin{split}
                    g_i(\mathbf{z}) :=
                    \begin{cases}
                        - z_i \quad &\text{if} \quad i \in \{ 1, \ldots, n\}, \\
                        z_{i - n} - 1 \quad &\text{if} \quad i \in \{ n + 1, \ldots, 2n\},
                    \end{cases}
                \end{split}
            \end{align*}            
            where the gradients have the following form
            \begin{align*}
                \nabla g_i(\mathbf{z}) :=
                \begin{cases}
                    - \mathbf{e}_i \quad &\text{if} \quad i \in \{ 1, \ldots, n\}, \\
                    \mathbf{e}_{i - n} \quad &\text{if} \quad i \in \{ n + 1, \ldots, 2n\}
                \end{cases}
            \end{align*}
            with $\mathbf{e}_i \in \mathbb{R}^n$ denoting the $i^{\text{th}}$ standard basis for all $i \in \{1, \ldots, n\}$. 
            We then observe that as $\lVert \Tilde{\mathbf{y}}_{t_k} -  \mathbf{z}_{t_k} \rVert \rightarrow 0$, we have $\Tilde{\mathbf{y}}_{t_k} \rightarrow \mathbf{z}^*$, which is the basis to  derive 
            \begin{equation*}
                \nabla \phi_{1, t_k} \left ( \Tilde{\mathbf{y}}_{t_k} \right) \rightarrow \nabla \phi_{1} \left( \mathbf{z}^*\right) \quad \text{and}  \quad \nabla \phi_{2, t_k} \left ( \mathbf{z}_{t_k} \right) \rightarrow \nabla \phi_{2} \left( \mathbf{z}^*\right),
            \end{equation*}
            with 
            \begin{equation*}
            \nabla \phi_1 \left( \mathbf{z} \right):= 2 \cdot \bar{\rho} \cdot \mathbf{z} \quad \text{and} \quad \nabla \phi_{2} \left( \mathbf{z} \right):= 2 \cdot \bar{\rho} \cdot \mathbf{z} - \nabla f \left( \mathbf{z} \right),
            \end{equation*}
            based on the continuity of $\nabla \phi_{1, t}$ and $\nabla \phi_{2, t}$, and the fact that the sequence of convexity parameters $\{ \rho_{t} \}$ is monotonously decreasing or increasing, and bounded from both sides [cf. Line \ref{line:rho_{t}} in Algorithm \ref{alg:sBDCA}]. 
            Now, it only remains to show that the subsequences of Lagrange multipliers $\{ \mu_{t_k, i}\}$ and $\{ \lambda_{t_k}\}$ are also bounded. 
            The latter can be proven based on the dual form of the \gls{MFCQ} that can be derived based on Motzkin's Transposition Theorem [cf. \citet[p. 4]{Solodov2011}; \citet[p. 44]{MANGASARIAN1967}]. 
            For the sake of brevity, we here only show that the \gls{LTS} problem fulfils \gls{MFCQ}. As the feasible set has only one equality constraint $\Tilde{h}(\mathbf{z}) := \mathbf{z}^\top \mathbf{1}_n - h = 0$ where $\mathbf{1}_n \in \{1\}^n$, it is trivial to show that the following equation
            \begin{equation}\label{eq:mfcq_lts_1}
                a \cdot \nabla \Tilde{h}(\mathbf{z}) = a \cdot\mathbf{1}_n = \mathbf{0}_n  \quad \text{with} \quad a \in \mathbb{R},
            \end{equation}
            has only one solution at $a = 0$, and therefore the independence requirement with respect to the gradients of the equality constraints is always fulfilled. 
            By selecting the elements of $\mathbf{a}:= \begin{bmatrix} a_1, \ldots, a_n \end{bmatrix}^\top \in \mathbb{R}^n$ such that $\sum_{i=1}^n a_i = 0$ and
            \begin{align*}
                    a_i \in
                    \begin{cases}
                         \mathbb{R}_{>0} \quad &\text{if} \quad z_i = 0, \\
                        \mathbb{R}_{<0} \quad &\text{if} \quad z_i = 1, \\
                        0 \quad &\text{if} \quad 0 < z_i < 1, 
                    \end{cases}
            \end{align*}
            is fulfilled, we can automatically obtain the following results 
            \begin{equation}\label{eq:mfcq_lts_2}
                \nabla g_i(\mathbf{z})^\top \mathbf{a}  = \sum_{a_i > 0} - a_i  + \sum_{a_j < 0} a_j < 0 \quad \forall i, j \in I(\mathbf{z}) \quad \text{and} \quad \nabla \Tilde{h}(\mathbf{z})^\top \mathbf{a} = \sum_{i=1}^n a_i = 0,
            \end{equation}
            as the box constraints can never be active on both sides for the same variable. 
            The combination of \eqref{eq:mfcq_lts_1} and \eqref{eq:mfcq_lts_2} then shows that all feasible points satisfy the \gls{MFCQ}. 
            Based on the previous results we then can take the limit ($t \rightarrow \infty$) in \eqref{eq:KKT_cond} to obtain 
            \begin{align}\label{eq:KKT_cond2}
                \begin{cases}
                   \nabla \phi_{1}(\mathbf{z}^*)  + \sum_{i=1}^{2n} \mu_{i} \cdot \nabla g_i(\mathbf{z}^*) + \lambda \cdot \mathbf{1}_n = \nabla \phi_{2}(\mathbf{z}^*), \\
                    \mu_{i} \cdot g_i(\mathbf{z}^*) = 0, \mu_{i} \geq 0, \ g_i(\mathbf{z}^*) \leq 0, \ i = 1, \ldots 2 \cdot n, \\
                    \mathbf{1}_n^\top \mathbf{z}^* = h,
                \end{cases}
            \end{align}
            which shows that $\mathbf{z}^*$ is a \gls{KKT} point of \eqref{eq:LTS_as_DC}. 
            Now, it remains to show that $\mathbf{z}^*$ is a local minimiser of \eqref{eq:LTS_as_DC}.
            For this purpose, we first observe that 
            the characterisation of the orthogonal projection [cf. Proposition~\ref{prop:projection_theorem}] yields the inequality
            $$
            (\mathbf{z} - \mathbf{y}_{t})^\top (\mathbf{z}_{t} - \eta_{t} \cdot \nabla f(\mathbf{z}_{t}) - \mathbf{y}_{t}) \leq 0
            $$
            for all $\mathbf{z} \in \Delta_h$. 
            We can then rearrange the inequality to obtain the following result
            $$
            (\mathbf{z} - \mathbf{y}_{t})^\top \nabla f(\mathbf{z}_{t})\geq \frac{1}{\eta_{t}} \cdot (\mathbf{z} - \mathbf{y}_{t})^\top (\mathbf{z}_{t} - \mathbf{y}_{t})
            $$
            for all $\mathbf{z} \in \Delta_h$.
            Taking the limit along a subsequence converging to an accumulation point $\mathbf{z}^*$, we obtain the result 
            $$
            \nabla f(\mathbf{z}^*)^\top (\mathbf{z} - \mathbf{z}^*) \geq 0
            $$ for all $\mathbf{z} \in \Delta_h$. 
            This is equivalent to
            $$ 
            \nabla f(\mathbf{z}^*)^\top \mathbf{z}^* \leq \mathbf{z}^\top \nabla f(\mathbf{z}^*)
            $$ 
            for all $\mathbf{z} \in \Delta_h$
            and hence
            $$
            \mathbf{z}^* \in \underset{\mathbf{z} \in \Delta_h}{\text{arg min}} \ \mathbf{z}^\top \nabla f(\mathbf{z}^*).
            $$
            Now suppose that there are no ties between the $h$th and ($h+1$)th smallest squared residuals.
            Then, the $h$ smallest components of $\nabla f(\mathbf{z}^*)$ are uniquely determined, which implies that $\mathbf{z}^*$ is the unique solution of the above linear optimisation problem.
            As a result, we can obtain the strict inequality
            $$
            \nabla f(\mathbf{z}^*)^\top (\mathbf{z} - \mathbf{z}^*) > 0
            $$ 
            for all $\mathbf{z} \in \Delta_h \setminus \{\mathbf{z}^*\}$. 
            This shows that for every non-zero feasible direction $\mathbf{d}$ at $\mathbf{z}^*$, the inequality $\nabla f(\mathbf{z}^*)^\top \mathbf{d} >0$ holds, which implies that $\mathbf{z}^*$ is a strict local minimiser.

        \item [(v)]
        To establish the last statement, we first show that $f$ is a multivariate rational function. 
        For this purpose, we utilise the Cayley-Hamilton Theorem to obtain
        \begin{equation*}
        (\mathbf{X}^\top \mathbf{Z} \mathbf{X})^{-1} = \frac{(\mathbf{X}^\top \mathbf{Z} \mathbf{X})^{p - 1} + \ldots + b_2 (\mathbf{z}) \cdot (\mathbf{X}^\top \mathbf{Z} \mathbf{X}) + b_1 (\mathbf{z}) \cdot \mathbf{I}_p}{b_0 (\mathbf{z})} ,
        \end{equation*}
        where the coefficients $b_1(\mathbf{z}), \ldots,  b_{p-1}(\mathbf{z})$ are multivariate polynomials in~$\mathbf{z}$.
        The coefficient $b_0(\mathbf{z}) = (-1)^p \cdot \text{det}(\mathbf{X}^\top \mathbf{Z} \mathbf{X})$ appears in the denominator and is also multivariate polynomial in~$\mathbf{z}$.
        Thus, the objective function can be rewritten as a fraction of two polynomial functions as long as $\mathbf{X}^\top \mathbf{Z} \mathbf{X}$ is invertible, which shows that $f$ is a real-analytic function. 
        We can then obtain the result that the function fulfils the \L{}ojasiewicz property with $\vartheta_{f} \in [\frac{1}{2},  1)$ [cf. \citet[p. 26]{Yagi2021};  \citet[ch. 1]{SheilSmall2002}; \citet[p. 67]{Lojasiewicz1965}].
        With this result, we can follow the standard techniques developed in \cite{aragon2018} and \citet{aragon2020} to derive the linear and sublinear convergence of $\{\mathbf{z}_{t}\}$. 
        We skip the detailed proof for brevity. 
    \end{enumerate}
\end{prf}



\subsection{Comparison of Fixed and Successive DC Decompositions}\label{sec:LS_and_sDC}


In this subsection, we theoretically and practically analyse how the choice between a fixed $\rho$ and a varying $\rho_t$ affects the performance of \gls{sBDCA}.
From a theoretical perspective, the convexity parameter is inversely proportional to~$\eta_t$, which represents the learning rate in the gradient-descent step.
At each iteration, the update is performed as follows
\begin{equation}\label{eq:gradient_descent_step}
    \Tilde{\mathbf{z}}_t = \mathbf{z}_t - \eta_t \cdot \nabla f \left(\mathbf{z}_t \right) = \mathbf{z}_t - \frac{1}{2 \cdot \rho_t} \cdot \nabla f \left(\mathbf{z}_t \right),
\end{equation}
and subsequently serves as input to the projection procedure.
Larger $\rho_t$ values correspond to smaller gradient-descent steps, which yield smaller improvements in the corresponding \gls{DCA} step. 
From a computational point of view, $\rho_t$ should therefore not be chosen too large. 
However, choosing $\rho_t$ too small is also not advisable, as this could cause the algorithm to overshoot desirable local minima or even diverge. 
To overcome this trade-off, we propose a flexible step size in Algorithm~\ref{alg:sBDCA}. 
One way to implement this idea is through the introduction of 
\begin{equation}\label{eq:rho_t}
    \rho_t = \min \Bigl \{ \max  \bigl( \nu \cdot \rho_{t - 1}, \rho_{\text{min}} \bigr), \rho_{\text{max}} \Bigr \} \quad \text{with} \quad \rho_{\text{max}} < \infty,
\end{equation}
where $\nu \in \mathbb{R}_{>0}$ determines whether the learning rate increases or decreases over time, $\rho_{\min} \in \mathbb{R}_{>0}$ ensures that the procedure does not diverge, and $\rho_{\max} \in \mathbb{R}_{>0}$ prevents the update steps from becoming insignificant. 
Based on this adaptation, a different \gls{DC} decomposition is obtained per iteration, along with an iteration-specific learning rate $\eta_t = \frac{1}{2 \cdot \rho_t}$. 


After this brief theoretical discussion, we now inspect the practical perspective using ten synthetic toy examples.
These examples are created based on the \gls{OSCM} introduced by \citet[p. 3221]{NGUYEN2010}, where outliers are created in the $\mathbf{y}$-direction. 
The regression coefficients are randomly drawn according to $\beta_j \sim \mathcal{U}(0, 5)$ for all $j \in \{1, \ldots, p\}$, and the corresponding observations are constructed as
\begin{align*}
    \begin{split}
        y_i &= \beta_1 + \beta_2 \cdot x_{i2} + \ldots + \beta_p \cdot x_{ip} + \epsilon_i, \quad \text{for all} \quad i \in \{1, \ldots, n\},
    \end{split}
\end{align*}
where $x_{ij} \sim \mathcal{U}(0, 1)$ for all $j \in \{2, \ldots, p\}$, and $\epsilon_i \sim \mathcal{N}(0, 1)$. 
Then, outliers are introduced by replacing the error terms for a subset of instances with values drawn from a $\chi^2$-distribution with five degrees of freedom.
For each combination of $n \in \{250, 500\}$, $h \in \{\lceil (n + p + 1) \cdot 2^{-1} \rceil\}$, and $p \in \{\lceil n \cdot 0.1 \rfloor, \lceil n \cdot 0.2 \rfloor, \lceil n \cdot 0.3 \rfloor, \lceil n \cdot 0.4 \rfloor, \lceil n \cdot 0.5 \rfloor\}$, one dataset is generated containing $\lceil n \cdot 0.1 \rfloor$ outliers. 
For these datasets, we examine the performance of \gls{sBDCA} under six different parameter settings.
In three configurations, the algorithm runs with a fixed learning rate where $\rho \in \{1, 10, 100\}$ such that the procedure is in line with the \gls{BDCA}. 
In the other settings, different combinations of $\nu \in \{0.9, 0.95\}$, $\rho_0 \in \{10, 100, 1000\}$ and $\rho_{\min} \in \{2.5\}$ are selected.  
Since the \gls{sBDCA} produces non-deterministic results, each configuration is evaluated using 250 random starting points sampled from the interior of the feasible set.
Note that this numerical comparison was performed in Python using the \href{https://www.southampton.ac.uk/research/facilities/iridis-research-computing-facility}{IRIDIS~5 Research Computing Facility}, where each compute node is equipped with an Intel Xeon Gold 6138 processor (40 CPUs at 2.0 GHz) and 192 GB of DDR4 memory.
To ensure a reproducible and transparent comparison, the complete implementation is publicly available via the project’s GitHub repository: \href{https://github.com/mlthormann/LTS-With-DC-Programming/}{LTS-with-DC-Programming \faExternalLink}.


The results for the ten toy examples are summarised in Figure~\ref{fig:comp_fix_succ_dc}, where the performance of Fast-\gls{LTS} serves as a benchmark. 
As no publicly available Python implementation of the heuristic was available, we translated the MATLAB code from the \textit{LIBRA} library into Python [cf. \citet{Verboven2010}].
The two upper subplots in Figure~\ref{fig:comp_fix_succ_dc} compare the \gls{ATD} ($:= \sqrt{f(\mathbf{z}) \cdot h^{-1}}$, right y-axis, box plots) and the median computation time (left y-axis, bar plots) of all feasible outcomes of the \gls{sBDCA} configurations with a fixed learning rate. 
Overall, our empirical findings confirm the trade-off between computational efficiency and solution accuracy highlighted in the theoretical discussion.
Larger values of $\rho$ result in lower objective function values but at the expense of increased median computation time. 
This slowdown is primarily due to the higher number of iterations required for convergence. 
It is particularly pronounced in high-dimensional settings with many independent variables, where computing the inverse of $\mathbf{X}^\top \mathbf{Z} \mathbf{X}$ becomes more costly at each iteration.
In contrast, smaller values of $\rho$ substantially reduce runtime but also compromise solution quality.
In comparison to the Fast-\gls{LTS}, \gls{sBDCA} with a larger $\rho$ provides a similar performance in settings with fewer regression coefficients, but outperforms the state-of-the-art in terms of output quality in higher-dimensional settings.


\begin{figure}[H]%
    \centering
    \captionsetup{justification=centering}
    \includegraphics[width=1\textwidth]{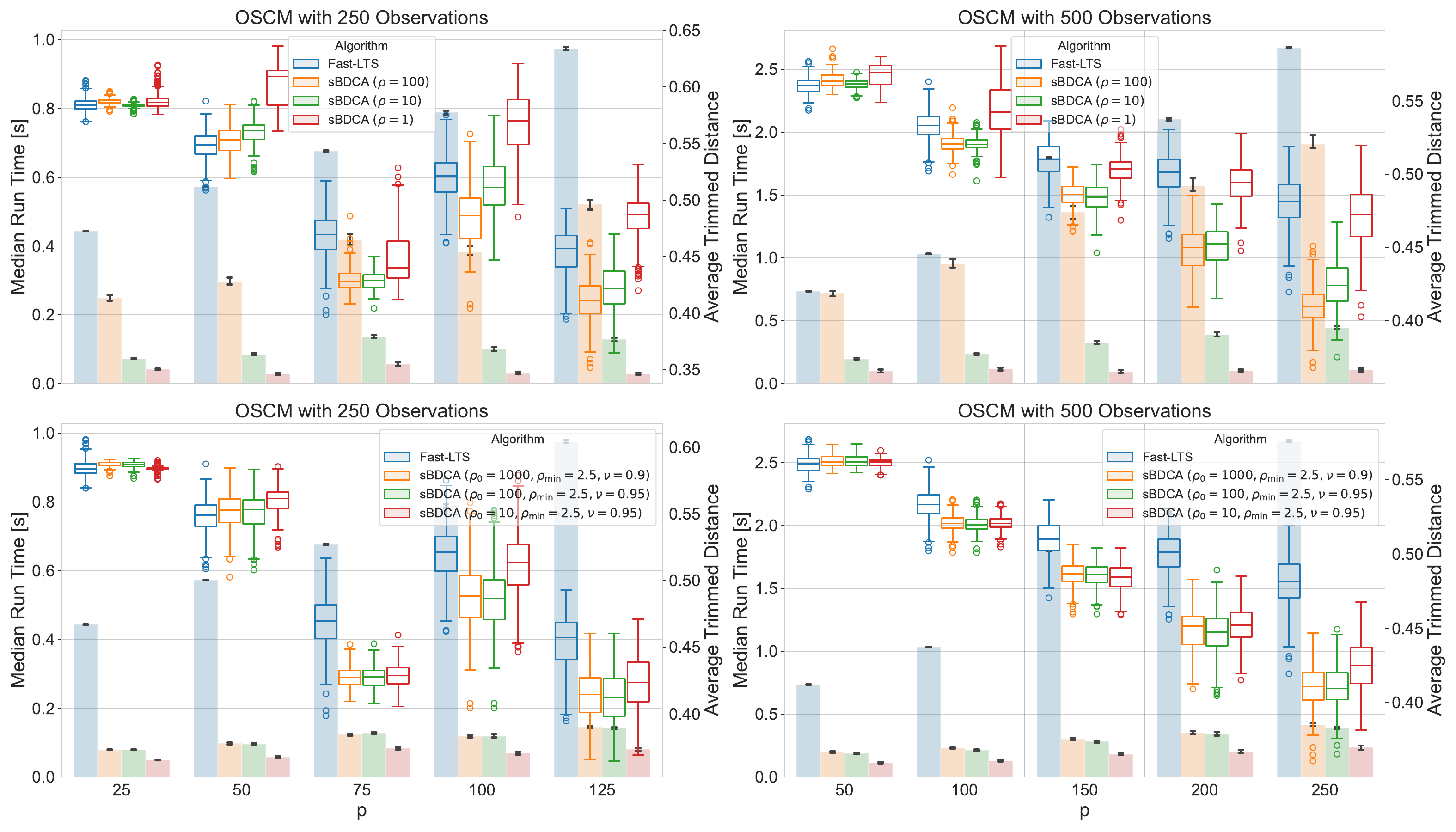}
    \caption{Influence of $\rho_t$ on \gls{ATD} and Computation Time.}%
    \label{fig:comp_fix_succ_dc}%
\end{figure}

\vspace{-0.2cm} 


The two lower subplots in Figure~\ref{fig:comp_fix_succ_dc} compare the \gls{ATD} (right y-axis, box plots) and the median computation time (left y-axis, bar plots) of all feasible outcomes of the \gls{sBDCA} configurations with an iteration-specific learning rate. 
Compared to the results shown in the upper subplots, the best median \gls{ATD} remains similar, while its variability is slightly reduced.
In high-dimensional settings, \gls{sBDCA} with an iteration-specific learning rate also consistently outperforms Fast-\gls{LTS} in terms of \gls{ATD}, while delivering comparable solution quality in lower-dimensional cases.
In terms of computation time, all \gls{sBDCA} variants shown in the lower subplots exhibit notable efficiency improvements, consistently outperforming Fast-\gls{LTS} across the ten toy examples.
Especially in settings with many independent variables, the computation time of Fast-\gls{LTS} increases sharply, whereas \gls{sBDCA} configurations with an iteration-specific learning rate maintain a similar runtime.
The observed efficiency gains can be explained by the fact that \gls{sBDCA} requires relatively few iterations, with the count remaining stable or even decreasing as $p$ increases. 
In contrast, Fast-\gls{LTS} is hindered by the need to update the regression coefficients more than 1,000 times during the concentration steps, regardless of $p$ [cf. \citet[p. 2504]{TORTI2012}].
As previously noted, these updates require computing the inverse of $\mathbf{X}^\top \mathbf{Z} \mathbf{X}$, which constitutes the primary source of computational cost and scales significantly with increasing $p$.



\subsection{Initialization Strategy \& Preconditioning the Gradient with Leverage}\label{sec:initialisation_and_preconditioner}


Previously, we outlined the derivation of \gls{sBDCA} and examined the impact of the convexity parameter on runtime.
From a computational perspective, this modification already made the proposed \gls{DC} programming algorithm highly competitive across the ten toy examples. 
In this subsection, we explore strategies to further enhance the solution quality, which appears to be strongly influenced by the chosen starting point [cf. Figure~\ref{fig:comp_fix_succ_dc}]. 
From a technical perspective, this dependence can be explained by examining the partial derivative of the objective function $f$ with respect to an individual weight $z_i$, for $i \in \{1, \ldots, n\}$, formally expressed as 
\begin{align}\label{eq:derivative}
    \begin{split}
        \frac{\partial f(\mathbf{z})}{\partial z_i} &= \frac{\partial (\mathbf{y}^\top \mathbf{Z} \mathbf{y} - \mathbf{y}^\top \mathbf{Z} \mathbf{X} (\mathbf{X}^\top \mathbf{Z} \mathbf{X})^{-1} \mathbf{X}^\top \mathbf{Z} \mathbf{y})}
        {\partial z_i} \\
        &= y_i^2 - 2 \cdot \frac{\partial (\mathbf{y}^\top \mathbf{Z} \mathbf{X})}{\partial z_i} (\mathbf{X}^\top \mathbf{Z} \mathbf{X})^{-1} \mathbf{X}^\top \mathbf{Z} \mathbf{y} - \mathbf{y}^\top \mathbf{Z} \mathbf{X} \frac{\partial \Bigl( (\mathbf{X}^\top \mathbf{Z} \mathbf{X})^{-1} \Bigr)}{\partial z_i} \mathbf{X}^\top \mathbf{Z} \mathbf{y}\\
        &= y_i^2 - 2 \cdot y_i \cdot \mathbf{x}_i^\top (\mathbf{X}^\top \mathbf{Z} \mathbf{X})^{-1} \mathbf{X}^\top \mathbf{Z} \mathbf{y} + \mathbf{y}^\top \mathbf{Z} \mathbf{X} (\mathbf{X}^\top \mathbf{Z} \mathbf{X})^{-1} \mathbf{x}_i \mathbf{x}_i^\top (\mathbf{X}^\top \mathbf{Z} \mathbf{X})^{-1} \mathbf{X}^\top \mathbf{Z} \mathbf{y} \\
        &= \Bigl(y_i - \mathbf{x}_i^\top (\mathbf{X}^\top \mathbf{Z} \mathbf{X})^{-1} \mathbf{X}^\top \mathbf{Z} \mathbf{y} \Bigr)^2 \\
        &= \Bigl( y_i - \mathbf{x}_i^\top \boldsymbol{\beta} (\mathbf{z}) \Bigr)^2,
    \end{split}
\end{align}
which corresponds to the squared residual of the $i$-th instance computed based on $\mathbf{z}$. 
It plays a critical role in the strongly convex subproblem, which is solved at each iteration by projecting $\Tilde{\mathbf{z}}_t = \mathbf{z}_t - \eta_t \cdot \nabla f(\mathbf{z}_t)$ onto the feasible set. 
In particular, the orthogonal projection onto the capped simplex tends to map the largest and smallest components of $\Tilde{\mathbf{z}}_t$ close to zero and one, respectively. 
Consequently, the $t$-th regression fit must already identify outliers among the instances with the largest residuals and smallest weights. 
Otherwise, differences between the current weight and the scaled residual may not fall among the $n-h$ smallest values, potentially increasing the influence of outliers erroneously. 
This highlights that the selection of an appropriate starting point is critical for finding lower objective function values, and is discussed in detail below.


\begin{figure}[H]%
    \centering
    \captionsetup{justification=centering}
    \includegraphics[width=1\textwidth]{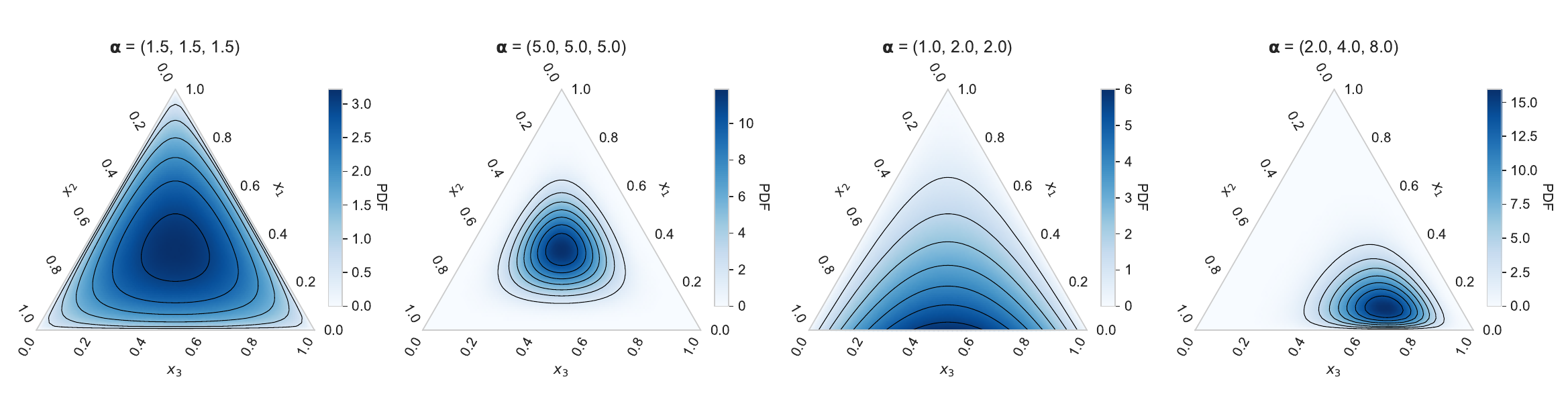}
    \caption{Influence of $\boldsymbol{\alpha}$ on the \gls{PDF} of the Dirichlet Distribution.}%
    \label{fig:dirichlet_pdf}%
\end{figure}

\vspace{-0.2cm}


Throughout this paper, we assume no prior knowledge regarding the number or location of outliers in the data. 
Consequently, \gls{sBDCA} has so far been initialised using random points drawn from the interior of the capped simplex to avoid the vertices that correspond to local minima with high \gls{ATD}. 
Such starting points can be generated by sampling from a Dirichlet distribution, whose \gls{PDF} is given by 
\begin{equation*}
    f_{\text{Dir}} \left(\mathbf{x}, \boldsymbol{\alpha} \right) := \frac{\Gamma \left( \sum_{i=1}^n \alpha_i \right)}{\prod_{i=1}^n \Gamma(\alpha_i)} \prod_{i=1}^n x_i^{\alpha_i - 1},
\end{equation*}
where $\mathbf{x} = [x_1, \ldots, x_n]^\top \in \Delta_1$ denotes the support, $\boldsymbol{\alpha} = [\alpha_1, \ldots, \alpha_n]^\top \in \mathbb{R}_{>0}^{n}$ are the concentration parameters, and $\Gamma(\cdot)$ denotes the gamma function, which generalises the factorial to non-integer values [cf. \citet[p. 309]{ibe2013}; \citet[p. 38]{ng2011}]. 
By setting $\alpha_1 = \ldots = \alpha_n$, the resulting Dirichlet distribution becomes symmetric, concentrating its probability mass evenly around the centre of the probability simplex. 
This behaviour is illustrated in Figure~\ref{fig:dirichlet_pdf}, where it can further be seen that larger and more uniform values of $\alpha_i$, for all $i \in \{1, \ldots, n\}$, cause random draws to cluster tightly around the centre point $x_1 = \ldots = x_n = \frac{1}{n}$. 
To obtain a feasible initialisation within the capped simplex, these values can subsequently be scaled by a factor $h$. 


Since the inspection of the gradient $\nabla f$ at the beginning of this subsection suggested that starting points should ideally not be too far from the global solution -- which corresponds to a vertex -- we next consider an initialisation strategy based on an extreme point of the feasible set. 
Specifically, we denote by $\mathbf{v}_0 \in \{0, 1\}^n$ a vertex satisfying $\mathbf{v}_0^\top \mathbf{1}_n = h$. 
The probability that such a feasible point excludes all outliers is given by 
\begin{equation*}
    \mathbb{P} \bigl( \mathbf{v}_{0} \text{ contains none of the $n_o$ outliers} \bigr) = \prod_{i=0}^{h-1} \frac{n_{c} - i}{n - i},
\end{equation*}
where $n_c \in \mathbb{N}$ and $n_o \in \mathbb{N}$ denote the number of core and outlying observations, respectively, with $n = n_c + n_o$. 
Even for moderate values such as $n_c = 95$, $n_o = 5$, and $h = \frac{n}{2}$, this probability remains relatively low, namely $\prod_{i=0}^{50-1} \frac{95 - i}{100 - i} = 0.0312$. 
Consequently, such vertex-based starting points will likely include outliers. 
This drawback was also noted by \citet[p. 33--34]{rousseeuw2006}, who instead proposed to initialise Fast-\gls{LTS} using random $p$-subsets, denoted by $\mathbf{p}_0 \in \{0, 1\}^n$ with $\mathbf{p}_0 \in \Delta_p$.
As such a starting point selects only $p$ observations, the probability that it excludes all outliers is given by 
\begin{equation*}
    \mathbb{P} \bigl( \mathbf{p}_{0} \text{ contains no outlier} \bigr) = \prod_{i=0}^{p-1} \frac{n_{c} - i}{n - i} > \mathbb{P} \bigl( \mathbf{v}_{0} \text{ contains no outlier} \bigr),
\end{equation*} 
which is larger than the corresponding probability for $\mathbf{v}_0$ whenever $p < h$. 
Although a $p$-subset does not constitute a feasible solution to the \gls{LTS} problem and thus cannot serve directly as an initialisation for \gls{sBDCA}, it can be used to estimate an initial regression fit. 
This initial fit provides estimates of the squared residuals, which in turn can be used to select the $h$ observations with the smallest errors, yielding a vertex initialisation. 
It appears to be the most logical approach to generate a feasible extreme point without prior outlier knowledge.


\begin{figure}[H]%
    \centering
    \captionsetup{justification=centering}
    \includegraphics[width=1\textwidth]{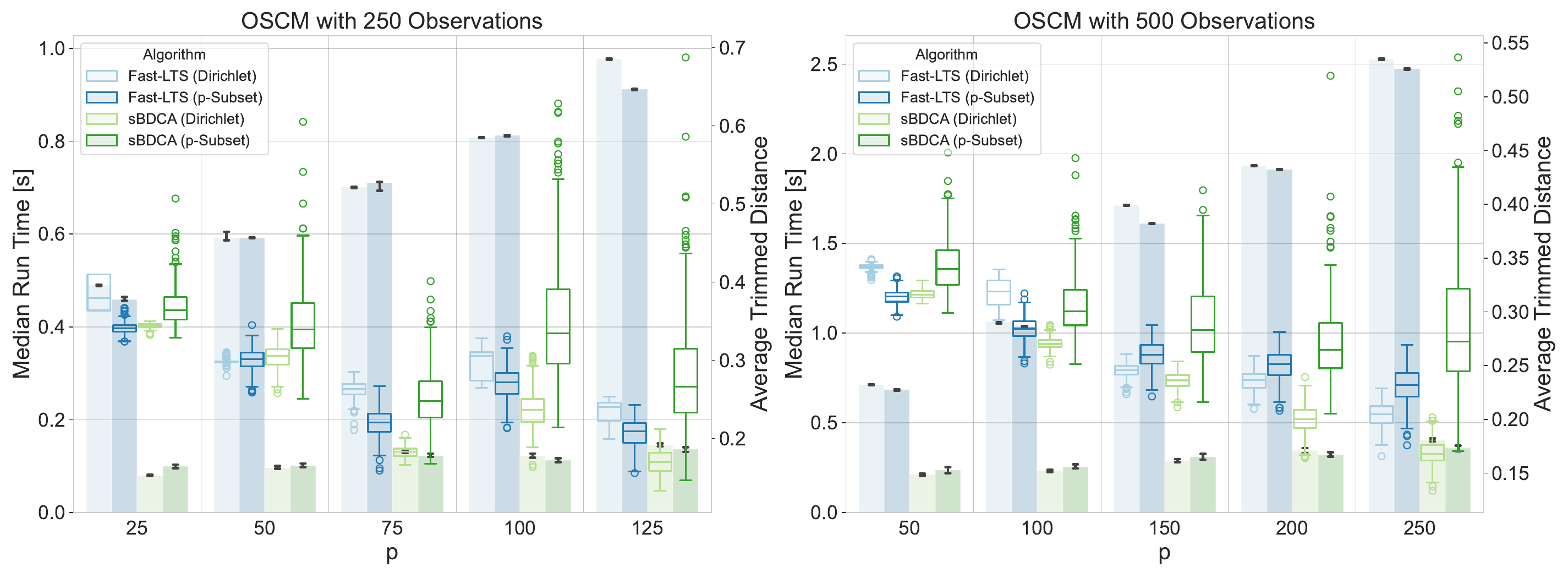}
    \caption{Comparison of different Initialisation Strategies for \gls{sBDCA}.}%
    \label{fig:toy_example3}%
\end{figure}

\vspace{-0.2cm}


To determine whether the $p$-subset initialisation strategy facilitates the identification of local solutions with lower \gls{ATD}, we repeat the ten toy examples from the previous subsection.
The results are presented in Figure~\ref{fig:toy_example3}, where the two  subplots compare the \gls{ATD} (right y-axis, box plots) and the median computation time (left y-axis, bar plots) of all feasible outcomes for the different combinations of $n$ and $p$.
When focusing on computation time, we observe that the two initialisation strategies have no significant impact on the runtime of \gls{sBDCA}.
However, the solution quality of the proposed \gls{DC} programming algorithm noticeably deteriorates under the $p$-subset initialisation.
As a result, \gls{sBDCA} is no longer competitive with Fast-\gls{LTS} in terms of \gls{ATD}. 
This underscores a major drawback of the $p$-subset strategy, whose effect on the variability of the \gls{ATD} for the given datasets could be compared with the return on an investment in randomly selected stocks. 
Specifically, putting all resources into a few $p$ shares (i.e., instances) can yield either exceptionally high returns (i.e., very low \gls{ATD}) or result in total losses (i.e., a regression breakdown with very high \gls{ATD}). 
In contrast, spreading investments more evenly over all $n$ stocks leads to better diversification, which offers stable but not extraordinary returns (i.e., we obtain more similar solutions, which may not correspond to the global minimum).


Notably, Fast-\gls{LTS} does not suffer from the drawback of the $p$-subset initialisation. 
As explained in Subsection~\ref{sec:LTS_Challenges_Heuristic}, the heuristic internally generates 500 random $p$-subsets during each initialisation, which are processed separately through the internal sorting procedure. 
After a few iterations, only the most promising results are retained and run until convergence.  In the end, the final results are compared once more, and only the best one serves as output. 
This design increases the likelihood of obtaining at least one clean subset without outliers and ensures a robust end result. 
In general, such a ``The More, The Merrier''-approach was also suggested by other researches, e.g., \citet[p. 104]{Critchley2010}, \citet[p. 206]{Rousseeuw1997} and \citet[p. 198]{hawkins1984}, and could improve the solution quality of \gls{sBDCA} under the $p$-subset initialisation. 
However, it is less practical from a computational perspective, as the underlying optimisation procedure is more costly than the sorting strategy of Fast-\gls{LTS}. 
For example, in Figure~\ref{fig:toy_example3} it already can be seen that if the results of 250 starting points were combined to one by selecting the best solution, the computation time of all attempts also needed to be summed up. 
This would significantly increase the overall runtime of \gls{sBDCA}, and make it slower than Fast-\gls{LTS}.


The previous toy examples demonstrated that, when limited to a single starting point, an alternative initialisation strategy does not improve the objective function value of the local solution found by \gls{sBDCA}. 
Therefore, we now address how the potentially misleading direction provided by the gradient of $f$ could be corrected. 
For this purpose, we revisit the different types of outliers introduced in Subsection~\ref{sec:outliers_in_regression}. 
Generally, an instance is classified as an outlier if it deviates from the true underlying regression fit \textit{or} exhibits an unusual combination of independent variables. 
Observations that meet both criteria are known as \textit{bad leverage points}, which are particularly problematic as they can significantly distort the regression fit [cf. \citet[p. 1]{wilcox2023}]. 
This can lead to \textit{swamping effects}, where non-outlying observations have larger residuals than actual outliers [cf. \citet[p. 508]{Habshah2009}]. 
However, while the residual of an observation indicates its distance from the regression fit (outliers in the $\mathbf{y}$-direction), it does not provide direct information about unusual combinations of independent variables (outliers in the $\mathbf{x}$-direction). 
To address this limitation, we could combine the squared residual, given by $\nabla f$,  with a measure that quantifies the influence of the covariates. 
In classical linear regression, this information can be available in the \textit{hat matrix}, which is formally defined as
\begin{equation*}
    \mathbf{H} := \mathbf{X} (\mathbf{X}^\top \mathbf{X})^{-1} \mathbf{X}^\top  \in \mathcal{M}_{n \times n}(\mathbb{R}),
\end{equation*}
and possesses a number of noteworthy properties [cf. \citet[p. 15]{morris2011}]. 
For example, $\mathbf{H}$ is symmetric ($\mathbf{H} = \mathbf{H}^\top$), idempotent ($\mathbf{H}^2 = \mathbf{H}$) and positive semi-definite ($\mathbf{H} \succeq \mathbf{0}_{n \times n}$) [cf. \citet[p. 120]{Fahrmeir2021}; \citet[p. 196]{Hocking2003}]. 
Originally, the hat matrix was introduced by John W. Tukey who derived its name based on the fact that the matrix maps $\mathbf{y}$ into~$\hat{\mathbf{y}}  := \mathbf{H}\mathbf{y} \in \mathbb{R}^n$ [cf. \citet[pp. 180--181]{Fahrmeir2021}; \citet[p. 17]{Hoaglin1978}]. 
The diagonal elements of this matrix are referred to as \textit{leverage statistics/scores} or \textit{hat values} [cf. \citet[p. 99]{James2021}; \citet[pp. 398--399]{fox2019}]. 
In outlier diagnostics, they are of particular interest as they can be interpreted as outlier score in $\mathbf{x}$-direction and reveal how strongly individual instances have influenced their own predictions [cf. \citet[p. 13]{DEMAESSCHALCK20001}]. 
For more details about the hat values, the interested reader is referred to Appendix \ref{sec:hat_values}.


In our \gls{DC} programming algorithm, the hat values can be used to deduce a Quasi-Newton-type correction of the direction given by the gradient $\nabla f$. 
For this purpose, we first observe that the hat matrix of the \gls{LTS} problem slightly deviates from $\mathbf{H}$ as it has the form 
\begin{equation*}
    \Tilde{\mathbf{H}} := \mathbf{X} (\mathbf{X}^\top \mathbf{Z} \mathbf{X})^{-1} \mathbf{X}^\top \mathbf{Z} \in \mathcal{M}_{n \times n}(\mathbb{R}),
\end{equation*}
where we denote the individual elements by $\Tilde{\mathbf{H}} = ( \tilde{h}_{ij} )_{1 \leq i, j \leq n}$. 
Moreover, the hat values $0 \leq \tilde{h}_{ii} \leq 1$ are only bounded from below by zero as the weights of some observations can be equal to zero. 
Based on this information, we can then construct the following diagonal scaling matrix
\begin{equation*}
    \mathbf{B}_t = \text{diag} \left (\frac{1}{(1 - \tilde{h}_{11, t})^\theta}, \ldots, \frac{1}{(1 - \tilde{h}_{nn, t})^\theta} \right) \in \mathcal{M}_{n \times n} (\mathbb{R}),
\end{equation*}
where $\theta \in \{ a \in \mathbb{R}_+ \ | \ \forall i:  (1 - \tilde{h}_{ii, t})^{a} > 0 \ \text{and} \ \frac{r^2_i}{(1 - \tilde{h}_{ii, t})^{a}} < \infty\}$, and $\tilde{h}_{ii, t}$ represents the $i$-th diagonal element of $\Tilde{\mathbf{H}}_t$ calculated based on $\mathbf{z}_t$. 
As this matrix and its inverse are positive definite, i.e., 
\begin{equation*}
\mathbf{B}_t \succeq \underset{i}{\min} \Biggl \{ \frac{1}{(1 - \tilde{h}_{ii, t})^\theta} \biggr \} \cdot \mathbf{I}_n \quad \text{and} \quad \mathbf{B}_t^{-1} \succeq \underset{i}{\min} \biggl \{  (1 - \tilde{h}_{ii, t})^\theta \biggr \}  \cdot \mathbf{I}_n, 
\end{equation*}
it can be applied as a preconditioner to the gradient of $f$.
This procedure can be interpreted as a standardisation of the squared residuals to the same variance [cf. Appendix \ref{sec:hat_values}].
The negative hat values can also be seen as an approximation of the diagonal elements of the Hessian  
\begin{equation*}
    \frac{\partial^2 f \left( \mathbf{z} \right)}{\partial z_i \partial z_i} = - 2 \cdot r_i^2 \cdot \mathbf{x}_i^\top\left( \mathbf{X}^\top \mathbf{Z} \mathbf{X}\right)^{-1} \mathbf{x}_i .
\end{equation*}
Then, instead of exactly solving the strongly convex subproblem at the beginning of the optimisation procedure, we can obtain $\Tilde{\mathbf{y}}_t$ based on $P_{\Delta_h} \bigl ( \mathbf{z}_t - \eta_t \cdot \mathbf{B}_t \nabla f (\mathbf{z}_t) \bigr)$, where $\eta_t$ again must be reasonably small,  and $\underset{i}{\min} \ (1 - \tilde{h}_{ii, t})^\theta$ sufficiently greater than zero.
Furthermore, $-\mathbf{B}_t \nabla f (\mathbf{z}_t)$ must be in the cone $\text{co} \bigl\{ -\nabla f (\mathbf{z}_t)$, $P_{\Delta_h} \bigl (- \nabla f (\mathbf{z}_t) \bigr) \bigr\}$ to theoretically guarantee that we reduce the objective function value, i.e., $f(\Tilde{\mathbf{y}}_t) < f(\mathbf{z}_t)$ [cf. \citet[Proposition~1]{Gafni1984}]. 
Even though $\Tilde{\mathbf{y}}_t$ is no longer the exact solution of \eqref{eq:subproblem_lts}, we can afterwards still perform the line search along the direction $\mathbf{d}_t$ due to Proposition~\ref{prop:descent_direction_dk}. 


\begin{algorithm}[H]
\vspace{0.05cm}
  \KwInput{Initial $\mathbf{z}_0 \in \Delta_h$, $\varepsilon > 0$, $T \in \mathbb{N}$, $\rho_0, \rho_{\text{min}}, \rho_{\text{max}}, \nu > 0$, $\theta \in \mathbb{N}$}  
  \vspace{0.1cm}
  \For{$t = 0,\ldots,T-1$}
   {
   \vspace{0.1cm}
   $\eta_t \gets \frac{1}{2 \cdot \rho_t}$ \\
   \vspace{0.1cm}
   $\Tilde{\mathbf{y}}_t \gets P_{\Delta_h} \Bigl( \mathbf{z}_t - \eta_t  \cdot \mathbf{B}_t \nabla f(\mathbf{z}_t) \Bigr)$ \\
   \vspace{0.1cm}
    $\mathbf{d}_t \gets \Tilde{\mathbf{y}}_t - \mathbf{z}_t$ \\
    \vspace{0.1cm}
    \If{$\lVert \mathbf{d}_t \rVert \leq \varepsilon$}
    {\vspace{0.1cm}
    \If{$\theta = 0$}{
    \vspace{0.1cm}
    \textbf{break}
    }
    \vspace{0.1cm}
    $\bar{\lambda}_t \gets 0$
    }\Else{   
    \vspace{0.1cm}
    $\bar{\lambda}_t \gets  \underset{( \Tilde{\mathbf{y}}_t + \bar{\lambda} \cdot \mathbf{d}_t  ) \in \Delta_h}{\text{arg max}} \  \bar{\lambda}$ \\
    \vspace{0.1cm}
    }
    $\mathbf{z}_{t + 1} \gets \Tilde{\mathbf{y}}_t + \bar{\lambda}_t \cdot \mathbf{d}_t$ \\
    \vspace{0.1cm}
    \If{$\Bigl( f(\mathbf{z}_{t + 1}) > f(\mathbf{z}_t) \Bigr)$ \textbf{\upshape or} $\Bigl( \theta > 0 \ \textbf{\upshape and} \ \lVert \mathbf{d}_t \rVert \leq \varepsilon \Bigr)$}{
    $\mathbf{z}_{t + 1} \gets \Tilde{\mathbf{y}}_t$ \\
    \vspace{0.1cm}
    $\theta \gets \theta - 1$
    }
    $\rho_{t + 1} \gets \min \Bigl \{ \max  \bigl( \nu \cdot \rho_t, \rho_{\text{min}} \bigr), \rho_{\text{max}} \Bigr \} $ \\
    \vspace{0.1cm}
   }
   \KwOutput{$\mathbf{z}_t$}
   \vspace{0.05cm}
\caption{\gls{sBDCA} with Preconditioning}\label{alg:sBDCA_with_precond}
\end{algorithm}  



\vspace{0.2cm}

\begin{prop}{prop:descent_direction_dk}
    If $f(\Tilde{\mathbf{y}}_t) < f(\mathbf{z}_t)$, then $\mathbf{d}_t = \Tilde{\mathbf{y}}_t - \mathbf{z}_t$ is a descent direction at $f(\Tilde{\mathbf{y}}_t)$, i.e., 
    \begin{equation*}
        \nabla_{\mathbf{d}_t} f(\Tilde{\mathbf{y}}_t) = \mathbf{d}_t^\top \nabla f(\Tilde{\mathbf{y}}_t) < 0,
    \end{equation*}
    and a line search can be performed along this direction.
\end{prop}
\begin{prf}{prf:descent_direction_dk}
    As $f$ is a concave and continuously differentiable function on $\Delta_h$, we can derive the inequality
    \begin{equation*}
        f(\Tilde{\mathbf{y}}_t) < f(\mathbf{z}_t) \leq f(\Tilde{\mathbf{y}}_t) - \langle \nabla f(\Tilde{\mathbf{y}}_t), \mathbf{d}_t \rangle,
    \end{equation*}
    which can only hold if $\langle \nabla f(\Tilde{\mathbf{y}}_t), \mathbf{d}_t \rangle < 0$.
\end{prf}

\vspace{-0.2cm}


In the new update scheme, the parameter $\theta$ plays a special role as it determines the importance of the hat value compared to the squared residual. 
The bigger we choose the parameter, the more the residuals of instances with higher hat values are inflated, and accordingly these observations are targeted for weight reduction. 
Especially at the beginning, it appears to be reasonable to select a bigger $\theta$ to underestimate the variance and to aim for a regression fit that prioritises low leverage scores over small residuals. 
Afterwards, we can keep $\theta$ constant as long as the projection step further reduces the objective function value. 
Otherwise, we have to reduce the parameter to fulfil the conditions that ensure an improvement of $f$ through the projection procedure. 
Furthermore, if~$\theta$ is still bigger than zero in the last iterations, we have to enforce that the parameter is set to zero. 
This ensures that we just perform the normal \gls{sBDCA} in the end, which guarantees that we find a local solution of the \gls{LTS} problem with a (sub)linear rate. 
Algorithm~\ref{alg:sBDCA_with_precond} illustrates how this idea could be implemented in practice.


To assess the impact of the proposed correction technique, we repeat the ten toy examples once more. 
The corresponding results are shown in Figure \ref{fig:toy_example3},
where the two subplots again compare the \gls{ATD} (right y-axis, box plots) and the median computation time (left y-axis, bar plots) of all feasible outcomes for the different combinations of $n$ and $p$.
When considering computational efficiency, we observe that preconditioning increases the computation time in all settings, which can be attributed to a higher number of iterations required for convergence.
Nevertheless, both \gls{sBDCA} versions still outperform the runtime of Fast-\gls{LTS}, particularly in high-dimensional settings with many independent variables. 
In terms of solution quality, the correction technique significantly reduces the variability, which allows \gls{sBDCA} to remain highly competitive even when initialised with a single $p$-subset.
Furthermore, \gls{sBDCA} with preconditioning consistently outperforms Fast-\gls{LTS} and \gls{sBDCA} without preconditioning in terms of median \gls{ATD} independent of the initialisation strategy. 
In high-dimensional settings, this difference becomes even more pronounced, where Algorithm~\ref{alg:sBDCA_with_precond} achieves up to 50\% smaller objective function values than Fast-\gls{LTS}.
As mentioned in Subsection~\ref{sec:LTS_Challenges_Heuristic}, the core procedure of Fast-\gls{LTS} is not inherently robust and highly depends on a good initialisation. 
In settings with many independent variables, the $\mathbb{P} \bigl( \mathbf{p}_{0} \text{ contains no outlier} \bigr)$ decreases, such that the 500 starting points are no longer sufficient to output a solution that is close to the best known \gls{ATD}.
This supports the results of  \citet[p. 2505]{TORTI2012}, who highlight that 500 starting points may not be enough to explore all local minima.


\begin{figure}[H]%
    \centering
    \captionsetup{justification=centering}
    \includegraphics[width=1\textwidth]{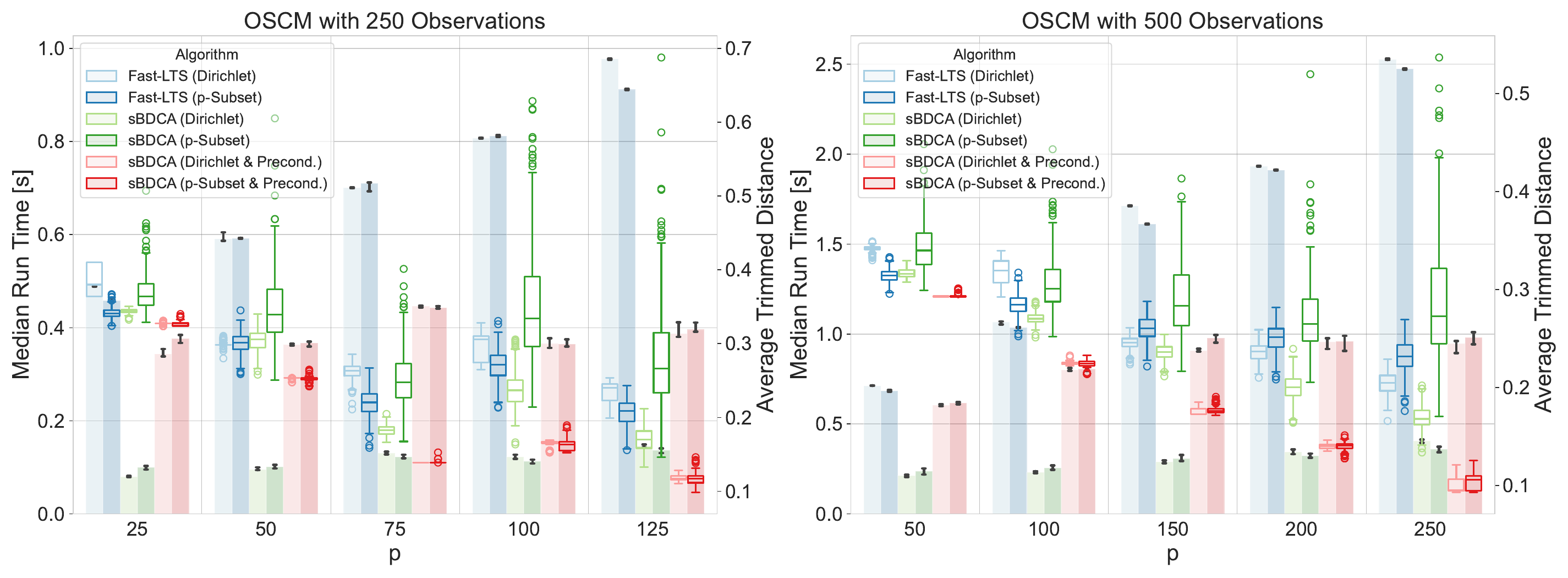}
    \caption{Influence of Preconditioner on \gls{ATD} and Computation Time.}%
    \label{fig:toy_example4}%
\end{figure}

\vspace{-0.2cm}


\section{Applications to Robust Regression}\label{sec:Application_to_Robust_Regression}


In the theoretical part of this paper, preliminary numerical results were already presented in Subsection~\ref{sec:LS_and_sDC} and Subsection~\ref{sec:initialisation_and_preconditioner} to showcase the influence of the successive \gls{DC} decompositions, the choice of starting points, and the proposed preconditioning matrix.
To provide a more comprehensive evaluation of the algorithmic performance, this section extends the empirical analysis by considering multiple synthetic and real-world datasets with leverage points and regression outliers under varying combinations of $n$ and $p$. 
Note that most experiments are again performed in Python using the \href{https://www.southampton.ac.uk/research/facilities/iridis-research-computing-facility}{IRIDIS 5 Research Computing Facility}  where each compute node is equipped with an Intel Xeon Gold 6138 processor (40 CPUs at 2.0 GHz) and 192 GB of DDR4 memory. 
To ensure a transparent and reproducible experimental set-up, the complete implementation is also publicly available via the project’s GitHub repository: \href{https://github.com/mlthormann/LTS-With-DC-Programming/}{LTS-with-DC-Programming \faExternalLink}.



\subsection{Smaller Regression Examples}


In this subsection, the focus is on smaller regression examples that were collected by \citet{rousseeuw1987} and \citet{Hadi1993}. 
In the academic literature, these datasets are often used to obtain an initial judgment about the algorithmic performance, as the global solutions can still be computed with a manageable computational effort. 
For our numerical study, we initially only select the Fire Claims, First Word, International Calls, and Star Cluster datasets [cf.~\citet[pp. 26, 29, 47, 50]{rousseeuw1987}] to visually compare the outputs of \gls{OLS}, Fast-\gls{LTS}, and \gls{sBDCA} with preconditioning. 
Since the latter two algorithms do not produce deterministic results, they are initialised 500 times and only the median results are illustrated in Figure~\ref{fig:Simple_LR}. 
Similarly, the regression fits for \gls{RANSAC}, Huber regression, Quantile regression and the \gls{TS} estimator are shown as these approaches are widely used in the literature [cf. Subsection~\ref{sec:robust_estimators_in_reg_analysis}] and are implemented in the \texttt{sklearn} Python library [cf. \citet{scikit-learn}].


\begin{figure}[H]%
    \centering
    \captionsetup{justification=centering}
    \includegraphics[width=1\textwidth]{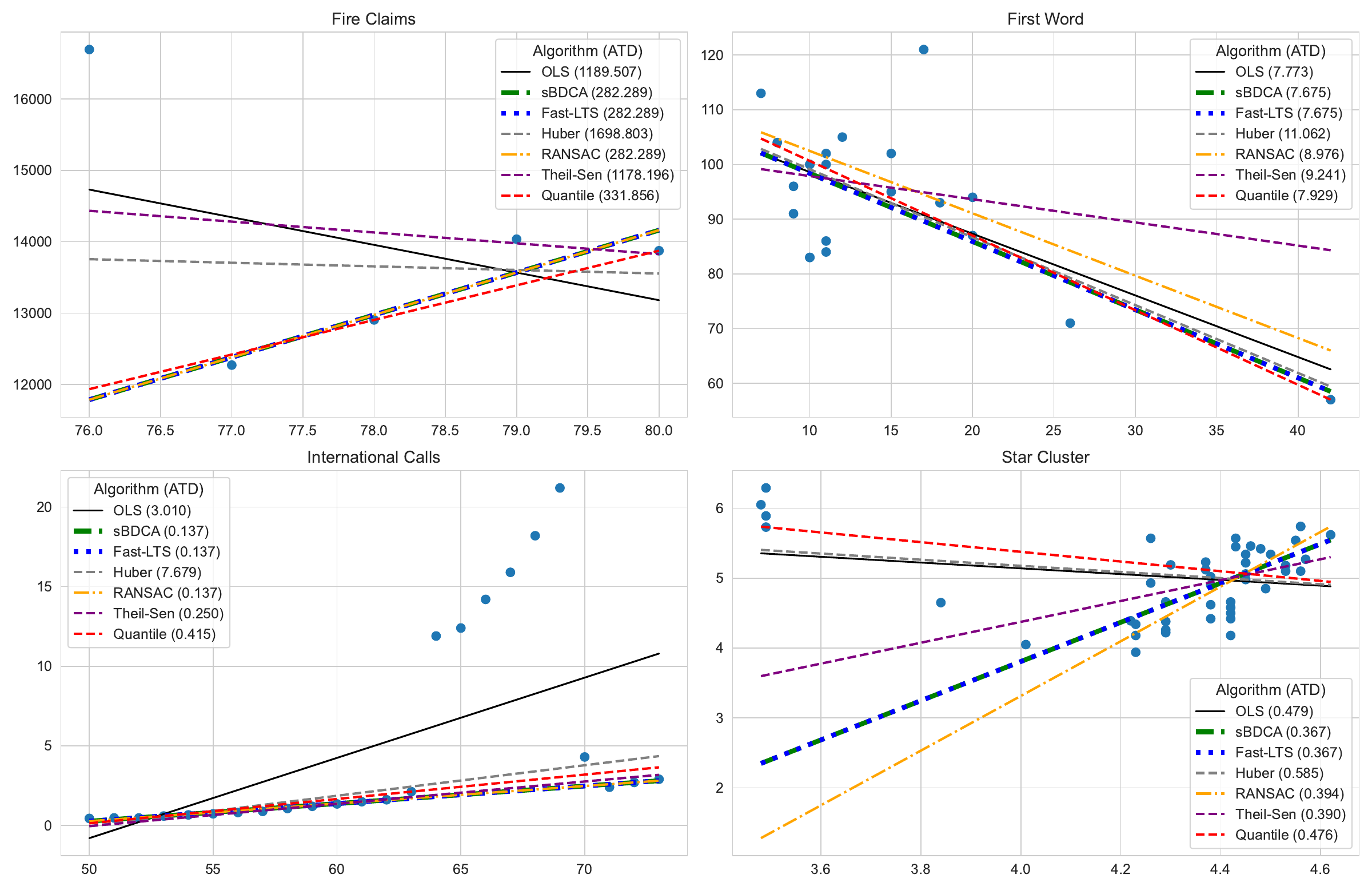}
    \caption{Simple Linear Regression Examples.}%
    \label{fig:Simple_LR}%
\end{figure}

\vspace{-0.2cm}


The regression fits in Figure~\ref{fig:Simple_LR} indicate that Fast-\gls{LTS} and \gls{sBDCA} with preconditioning yield identical median results for the chosen datasets.
Overall, both approaches do not encounter difficulties in separating outliers from the remaining observations and outperform the other estimators in terms of \gls{ATD}. 
Beyond this conclusion, \gls{RANSAC} emerges as another approach that delivers consistently robust results across the datasets, but exhibits minor deficiencies in the two subplots on the right side of Figure~\ref{fig:Simple_LR}.
A comparable performance is also achieved by the Quantile regression, but the estimator completely breaks down for the Star Cluster dataset, where the four isolated points in the top right corner cause the distortion.
An even worse performance is provided by the three remaining estimators that produce a distorted regression fit for at least two out of the four selected datasets. 
While \gls{OLS} and Huber regression have problems to handle the bad leverage points present in the Fire Claims and Star Cluster datasets, the \gls{TS} estimator breaks down under the First Word and Fire Claims examples.


\begin{table}[H]
    \scriptsize
    \centering
    \captionsetup{justification=centering}
    \caption{Median \acrfull{ATD} per Dataset and Algorithm.}
    \label{tab:Small_Regression_Examples}
    \begin{tabular}{>{\raggedright}p{2.7475cm} >{\raggedleft}p{0.25cm} >{\raggedleft}p{0.05cm} >{\raggedleft}p{0.25cm} | ccc >{\centering}p{0.95cm} >{\centering}p{0.95cm} cc}
    \toprule
    Data Set Name & n & p & h & {\tiny Fast-\gls{LTS}} & Huber & \gls{OLS} & {\tiny Quantile} & {\tiny \gls{RANSAC}} & \gls{TS} & \gls{sBDCA} \\
    \midrule
    Barnett \& Lewis & 7 & 1 & 5 & \textbf{2.577} & 3.426 & 269.6 & 3.033 & \textbf{2.577} & 3.033 & \textbf{2.577} \\
    Blood Pressure & 10 & 1 & 9 & \textbf{6.190} & 6.215 & 7.766 & 6.254 & 7.252 & 6.191 & \textbf{6.190} \\
    Cushny \& Peebles & 10 & 1 & 8 & \textbf{2.784} & 2.813 & 4.630 & 2.828 & 2.828 & 2.828 & \textbf{2.784} \\
    Kootenay River & 13 & 1 & 12 & \textbf{1.818} & 1.900 & 6.705 & 1.865 & \textbf{1.818} & 1.834 & \textbf{1.818} \\
    Lactic Acid Concen. & 20 & 1 & 16 & \textbf{0.625} & 0.666 & 0.701 & 0.688 & 0.701 & 0.668 & \textbf{0.625} \\
    Plastic Material & 10 & 1 & 8 & \textbf{0.387} & 0.418 & 0.503 & 0.410 & 0.434 & 0.393 & \textbf{0.387} \\
    Ranges of Projectiles & 8 & 1 & 6 & \textbf{37.84} & 44.56 & 85.65 & 39.10 & 39.81 & 42.05 & \textbf{37.84} \\
    Semidiameters Venus & 15 & 1 & 13 & \textbf{0.309} & 0.309 & 0.311 & 0.309 & 0.323 & 0.309 & \textbf{0.309} \\
    \midrule 
    Accidents & 7 & 2 & 6 & \textbf{493.5} & 513.5 & 1419 & 538.0 & \textbf{493.5} & 623.0 & \textbf{493.5} \\
    Body \& Brain Weight & 28 & 2 & 25 & \textbf{0.696} & 0.760 & 1.134 & 0.762 & 0.713 & 0.709 & \textbf{0.696} \\
    Fire Claims & 5 & 2 & 4 & \textbf{282.3} & 852.8 & 1190 & 331.9 & \textbf{282.3} & 1178 & \textbf{282.3} \\
    First Word & 21 & 2 & 19 & \textbf{7.675} & 7.714 & 7.773 & 7.929 & 8.976 & 9.241 & \textbf{7.675} \\
    Inflation in China & 9 & 2 & 8 & \textbf{3.577} & 4.025 & 60.96 & 3.972 & 4.718 & 4.162 & \textbf{3.577} \\
    International Calls & 24 & 2 & 17 & \textbf{0.137} & 0.627 & 3.010 & 0.415 & \textbf{0.137} & 0.250 & \textbf{0.137} \\
    Monthly Payments & 12 & 2 & 9 & \textbf{0.581} & 0.690 & 4.572 & 0.675 & 0.670 & 0.801 & \textbf{0.581} \\
    Pension Funds & 18 & 2 & 16 & \textbf{0.390} & 0.390 & 0.390 & 0.397 & 0.394 & 0.392 & \textbf{0.390} \\
    Pilot Plant & 20 & 2 & 19 & \textbf{1.187} & 1.317 & 11.95 & 1.407 & \textbf{1.187} & 1.222 & \textbf{1.187} \\
    Siegel's Example & 9 & 2 & 6 & \textbf{0.000} & 0.094 & 0.133 & 0.092 & \textbf{0.000} & 0.001 & \textbf{0.000} \\
    Star Cluster & 47 & 2 & 42 & \textbf{0.367} & 0.477 & 0.479 & 0.476 & 0.394 & 0.390 & \textbf{0.367} \\
    \midrule 
    Cloud Point & 19 & 3 & 16 & \textbf{0.207} & 0.245 & 0.292 & 0.223 & 0.292 & 0.238 & \textbf{0.207} \\
    Delivery Time & 25 & 3 & 24 & \textbf{2.273} & 2.457 & 2.729 & 2.393 & 2.637 & 2.319 & \textbf{2.273} \\
    Hadi and Simonoff & 25 & 3 & 22 & \textbf{0.786} & 0.918 & 0.988 & 0.976 & 0.988 & 0.928 & \textbf{0.786} \\
    Heart Catheterisation & 12 & 3 & 7 & \textbf{0.312} & 0.945 & 0.999 & 0.911 & 0.940 & 0.549 & 0.369 \\
    Phosphorus Content & 18 & 3 & 11 & \textbf{3.543} & 3.967 & 6.270 & 4.567 & 4.486 & 5.517 & \textbf{3.543} \\
    \midrule 
    Air Quality & 31 & 4 & 23 & \textbf{9.103} & 11.23 & 155.8 & 11.54 & 11.45 & 11.67 & \textbf{9.103} \\
    Education Expend. & 50 & 4 & 49 & \textbf{37.48} & 37.84 & 40.66 & 39.85 & 43.99 & 41.59 & \textbf{37.48} \\
    Hawkins-Bradu-Kass & 75 & 4 & 65 & \textbf{0.540} & 0.616 & 1.073 & 0.586 & 0.628 & 1.229 & 0.561 \\
    Salinity & 28 & 4 & 19 & \textbf{0.396} & 0.430 & 0.587 & 0.414 & 0.482 & 0.448 & \textbf{0.396} \\
    Stackloss & 21 & 4 & 16 & \textbf{0.892} & 1.140 & 1.822 & 1.003 & 1.002 & 1.017 & \textbf{0.892} \\
    \midrule 
    Aircraft & 23 & 5 & 14 & \textbf{1.604} & 2.482 & 3.275 & 2.299 & 2.172 & 2.281 & 1.625 \\
    Exact Fit & 25 & 5 & 20 & \textbf{0.000} & \textbf{0.000} & 0.948 & \textbf{0.000} & 0.471 & 0.005 & \textbf{0.000} \\
    \midrule 
    Coleman & 20 & 6 & 17 & \textbf{0.555} & 0.604 & 0.931 & 0.649 & 0.628 & 0.809 & \textbf{0.555} \\
    Wood Specific Gravity & 20 & 6 & 16 & \textbf{0.006} & 0.012 & 0.014 & 0.013 & \textbf{0.006} & 0.018 & 0.009 \\
    \bottomrule
    \end{tabular}
\end{table}


After obtaining a visual impression of the performance, we now evaluate the solution quality in terms of \gls{ATD} for all datasets collected by \citet{rousseeuw1987} and \citet{Hadi1993}. 
For this purpose, the estimators with non-deterministic outputs are again initialised 500 times per dataset, and the median \gls{ATD} is computed. 
The corresponding results are presented in Table \ref{tab:Small_Regression_Examples}, where the best outcome for each dataset is highlighted in bold. 
Overall, the performance of Fast-\gls{LTS} and \gls{sBDCA} with preconditioning are again very similar, with both approaches yielding the lowest \gls{ATD} for the majority of datasets. 
Occasionally, the state-of-the-art heuristic achieves a slightly lower \gls{ATD}, as illustrated, for instance, in the last row of Table \ref{tab:Small_Regression_Examples}. 
This can be attributed to the fact that Fast-\gls{LTS} generates 500 random $p$-subsets per initialisation ($500 \cdot 500 = 250,000$ starting points in total), which are sufficient to explore all local optima for very small datasets. 
In contrast, \gls{sBDCA} with preconditioning creates only a single starting point per initialisation.
As a result, in four of the 34 cases, it converges to a local solution that is slightly inferior to the global minimum, but it remains the second-best option for these datasets.
Beyond this observation, the performance of the remaining estimators depends on the specific dataset. 
As already discussed in relation to Figure~\ref{fig:Simple_LR}, \gls{RANSAC} often delivers robust results but achieves the lowest \gls{ATD} only for a small number of datasets.
Moreover, \gls{RANSAC} does not consistently outperform the Huber regression, \gls{TS} estimator, and Quantile regression in terms of \gls{ATD}.
The overall worst performance is provided by \gls{OLS}, which often produces results with very high \gls{ATD}.



\subsection{Synthetic Data Experiments in High Dimensions}\label{sec:synthetic_data_exp}


After comparing the algorithms on smaller regression examples with $n < 100$ and $p \leq 6$, this subsection examines medium- and large-scale synthetic datasets to evaluate performance relative to the number of instances and input variables. 
To ensure a fair and unbiased comparison, the numerical experiments follow the recommendations of \citet{beiranvand2017}.
Consistent with these best practices, performance is examined across three dimensions: efficiency, reliability, and output quality.
Specifically, efficiency is measured in terms of median computation time, while solution quality is assessed based on the distance to the best-known solution. 
To evaluate reliability, the median objective function values and the fraction of infeasible solutions are analysed across 500 initialisations per dataset.
Uncertainty in the point estimates is represented either by box plots or by 95\% \glspl{CI}, which are computed based on 100,000 bootstrap samples and adjusted using the Bonferroni correction [cf. \citet[pp. 655--656]{fox2015}; \citet{Armstrong2014};  \citet{Becher1993}; \citet{efron1979}; \citet{bonferroni1936}].


The evaluation of the algorithmic reliability requires generating multiple starting points, which must be selected carefully.
As \citet[p. 821]{beiranvand2017} note, biased starting points are among the main shortcomings in performance comparisons.
In particular, Subsection~\ref{sec:initialisation_and_preconditioner} has already shown that the algorithmic outputs are sensitive to the initialisation strategy. 
To avoid granting any algorithm an unfair advantage, we evaluate performance using both the $p$-subset and Dirichlet initialisation strategies. 
Based on the same reasoning, we also considered including the core procedure of Fast-\gls{LTS} without the benefit of 500 starting points in the comparison.
However, this addition provided limited insight, as the resulting \gls{ATD} is far from competitive when the heuristic is restricted to a single starting point.  


\begin{figure}[H]%
    \centering
    \includegraphics[width=1\textwidth]{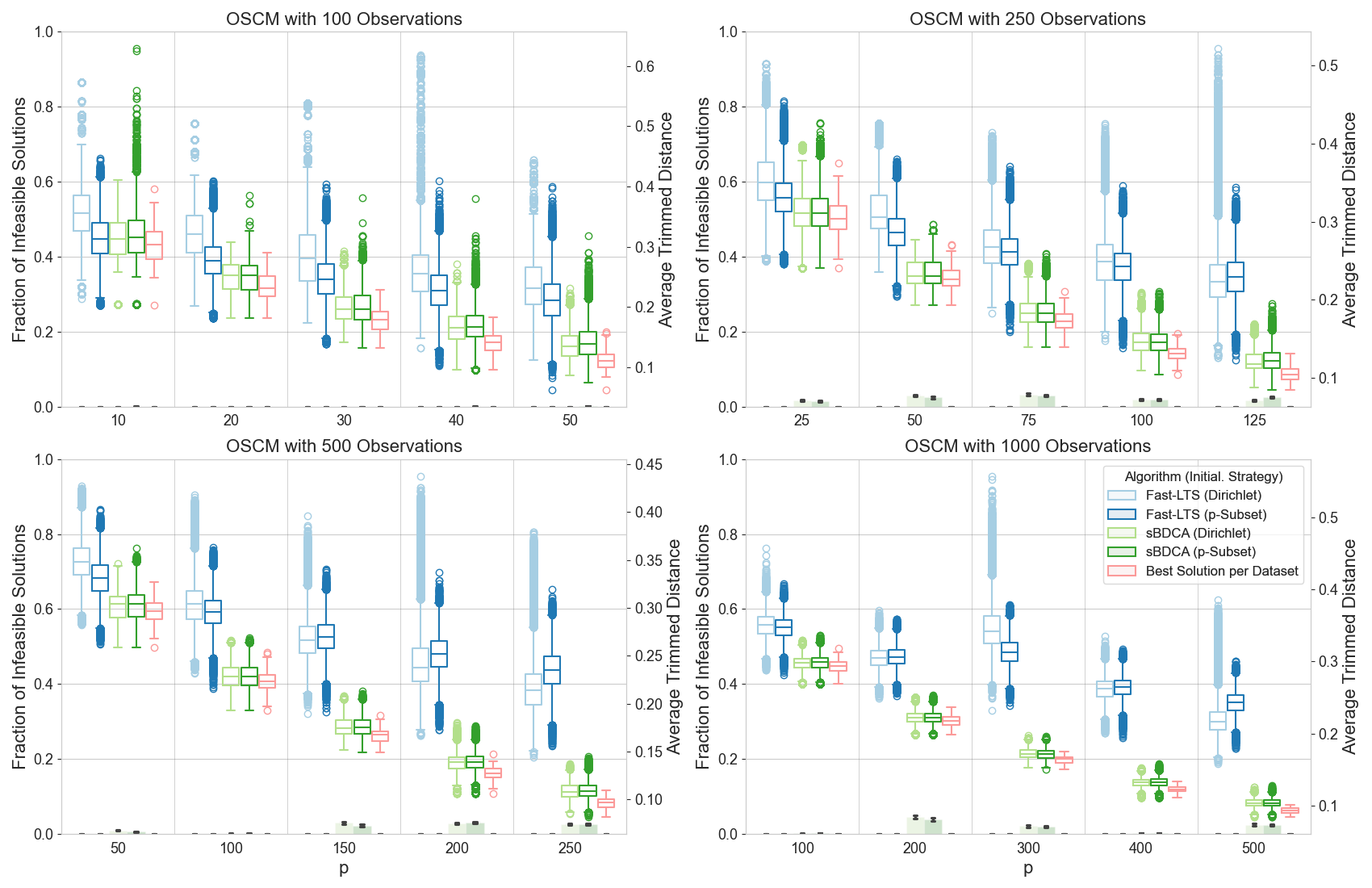}
    \caption[Performance comparison using the \gls{OSCM} and focusing on algorithmic reliability and output quality.]{Performance comparison between \gls{sBDCA} with preconditioning and Fast-\gls{LTS} using the \gls{OSCM} and focusing on algorithmic reliability and output quality. 
    Each subplot displays the fraction of infeasible solutions (left y-axis, bar plots) and the \gls{ATD} of all feasible outcomes (right y-axis, box plots) across different numbers of input variables on the x-axis. 
    Results are grouped by solver type and initialisation strategy (Dirichlet, $p$-subset).
    The box plots in red display the lowest \gls{ATD} among all feasible solutions per dataset independent of the solver.}%
    \label{fig:atd_oscm}%
\end{figure}

\vspace{-0.4cm}


Building on this assessment framework, the first numerical experiment of this subsection revisits the \gls{OSCM} introduced in Subsection~\ref{sec:LS_and_sDC}.
In particular, we extend the previous comparison by generating 100 random datasets for each combination of  $n \in \{100, 250, 500, 1000\}$ and $p \in \{\lceil n \cdot 0.1 \rceil, \lceil n \cdot 0.2 \rceil, \lceil n \cdot 0.3 \rceil, \lceil n \cdot 0.4 \rceil, \lceil n \cdot 0.5 \rceil \}$. 
In total, this yields 2,000 datasets.
The proportion of outliers is set to $n_{o} = \lceil n \cdot 0.1 \rceil$ to be in line with the estimate for routine datasets provided by \citet[p. 88]{hampel1973} and \citet[p. 26]{hampel1986}.
Likewise, we fix $h$ at $\lceil \frac{n + p + 1}{2} \rceil$ close to its lower bound, as the true number of outliers is typically unknown in practice.
The latter choice is based on the recommendation of \citet[p. 38]{rousseeuw2006}.


Figure~\ref{fig:atd_oscm} compares the fraction of infeasible solutions (left y-axis, bar plots) and the \gls{ATD} of all feasible outcomes (right y-axis, box plots) across the selected settings. 
Overall, the subplots confirm that \gls{sBDCA} with preconditioning achieves superior performance in terms of both solution quality and algorithmic reliability.
Compared to Fast-\gls{LTS}, the results of the proposed \gls{DC} programming algorithm are less sensitive to the initialisation strategy and are closer to the best solution per dataset for each $(n,p)$ combination. 
Especially in settings with many input variables, this performance gap increases such that the median \gls{ATD} of our proposed algorithm is up to 50\% lower compared to the value of the state-of-the-art. 
In these cases, the box plots also show no overlap, underscoring the substantial performance gap. 
As discussed in Subsection~\ref{sec:LTS_Challenges_Heuristic} and Subsection~\ref{sec:initialisation_and_preconditioner}, this discrepancy can be attributed to the $p$-subset initialisation strategy, which leads to initial regression fits of worse quality in high-dimensional settings. 
The 500 starting points generated within each Fast-\gls{LTS} initialisation are then insufficient to explore the solution landscape. 
Another limitation of the heuristic becomes evident in the lower-right subplot, where the \gls{ATD} does not  decreases monotonously as the number of input variables increases. 
This can be explained by the internal sub-sampling procedure applied to all datasets with $n > 600$, whose execution becomes more challenging in settings with many input variables.


\begin{figure}[H]%
    \centering
    \includegraphics[width=1\textwidth]{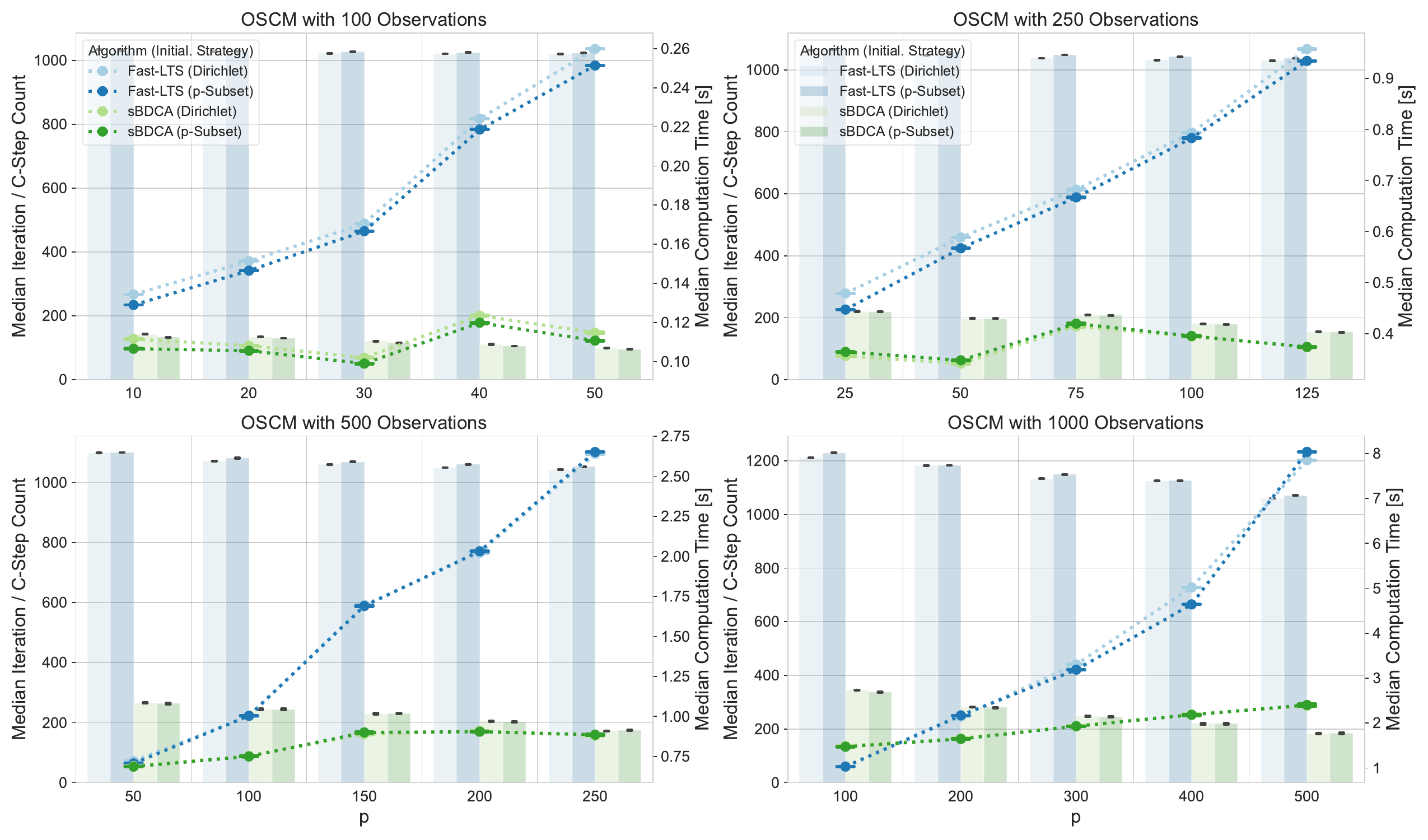}
    \caption[Performance comparison using the \gls{OSCM} and focusing on algorithmic efficiency.]{Performance comparison between \gls{sBDCA} with preconditioning and Fast-\gls{LTS} using the \gls{OSCM} and focusing on algorithmic efficiency. 
    Each subplot displays the median number of iterations / c-steps (left y-axis, bar plots) and the computation time of all feasible solutions (right y-axis, box plots) across different numbers of input variables on the x-axis. 
    Results are grouped by solver type and initialisation strategy (Dirichlet, $p$-subset).}%
    \label{fig:cpu_oscm}%
\end{figure}

\vspace{-0.4cm}


To complete the comparative analysis for the \gls{OSCM},
Figure~\ref{fig:cpu_oscm}  compares the median iteration~/~c-step count  (left y-axis, bar plots) and the computation time of all feasible outcomes (right y-axis, box plots) across the different settings.
Overall, the results indicate that our proposed \gls{DC} programming algorithm achieves superior computational performance across a range of settings. 
In particular, for datasets with a high number of independent variables, \gls{sBDCA} with preconditioning significantly outperforms Fast-\gls{LTS}, with median computation times up to 3.25 times faster.
This performance gap can primarily be attributed to the more expensive inversion of the $p \times p$ matrix $\mathbf{X}^\top \mathbf{Z} \mathbf{X}$ during a similar number of concentration steps in the Fast-\gls{LTS} heuristic, as previously discussed in Subsection~\ref{sec:LS_and_sDC}. 
In contrast, under the specified parameter setting, the iteration count of \gls{sBDCA} shows a slight decline as the dimensionality increases, which helps to mitigate the computational burden associated with the matrix inversion.  


The first numerical experiment provided several examples involving regression outliers.
To follow common practice in the academic literature, where algorithms are typically compared using data examples containing outliers in the $\mathbf{x}$- and $\mathbf{y}$-direction [cf. \citet[p.~21]{Chen2018}; \citet[p. 5]{XIE2022}; \citet[p. 2191]{Dogru2018}; \citet[p. 1329]{Nurunnabi2014}; \citet[p. 32]{AlNoor2013}; \citet[p. 114]{ROUSSEEUW1992}], the second numerical experiment is based on synthetic datasets with leverage points.
Specifically, the data examples are generated using the mechanism introduced by \citet[p. 209]{rousseeuw1987}, which has subsequently been applied by several other researchers, including \citet[p. 104]{satman2013}, \citet[p. 111]{Critchley2010}, \citet[p. 196]{Hofmann2010}, \citet[pp. 40--41]{rousseeuw2006}, and \citet[p. 114]{ROUSSEEUW1992}.
Within this framework, the core observations are generated according to
\begin{equation}\label{eq:sdm}
    y_i = x_{i1} + x_{i2} + \ldots + x_{ip} + \epsilon_i \quad \text{for all} \quad i \in \{1, \ldots, n\},
\end{equation}
where $\epsilon_i \sim \mathcal{N}(0, 1)$, $x_{i1} = 1$, and $x_{ij} \sim \mathcal{N}(0, 10^2)$ for $j \in \{2, \ldots, p\}$. 
Subsequently, outliers are introduced in the $\mathbf{x}$-direction by replacing the original $x_{i2}$ values for $n_o$ instances with random values drawn from $\mathcal{N}(100, 10^2)$. 
In our experiments, this setup is used to again generate 100 random datasets for each combination of $n \in \{100, 250, 500, 1000\}$ and $p \in \{\lceil n \cdot 0.1 \rceil, \lceil n \cdot 0.2 \rceil, \lceil n \cdot 0.3 \rceil, \lceil n \cdot 0.4 \rceil, \lceil n \cdot 0.5 \rceil\}$, with fixed values for $n_c =\lceil n \cdot 0.95 \rceil$ and $h = \lceil \frac{n + p + 1}{2} \rceil$, and results in a total of 2,000 datasets with (bad) leverage points. 
In the following paragraphs, we refer to this data-generating process as the \gls{DMLP}.


\begin{figure}[H]%
    \centering
    \includegraphics[width=1\textwidth]{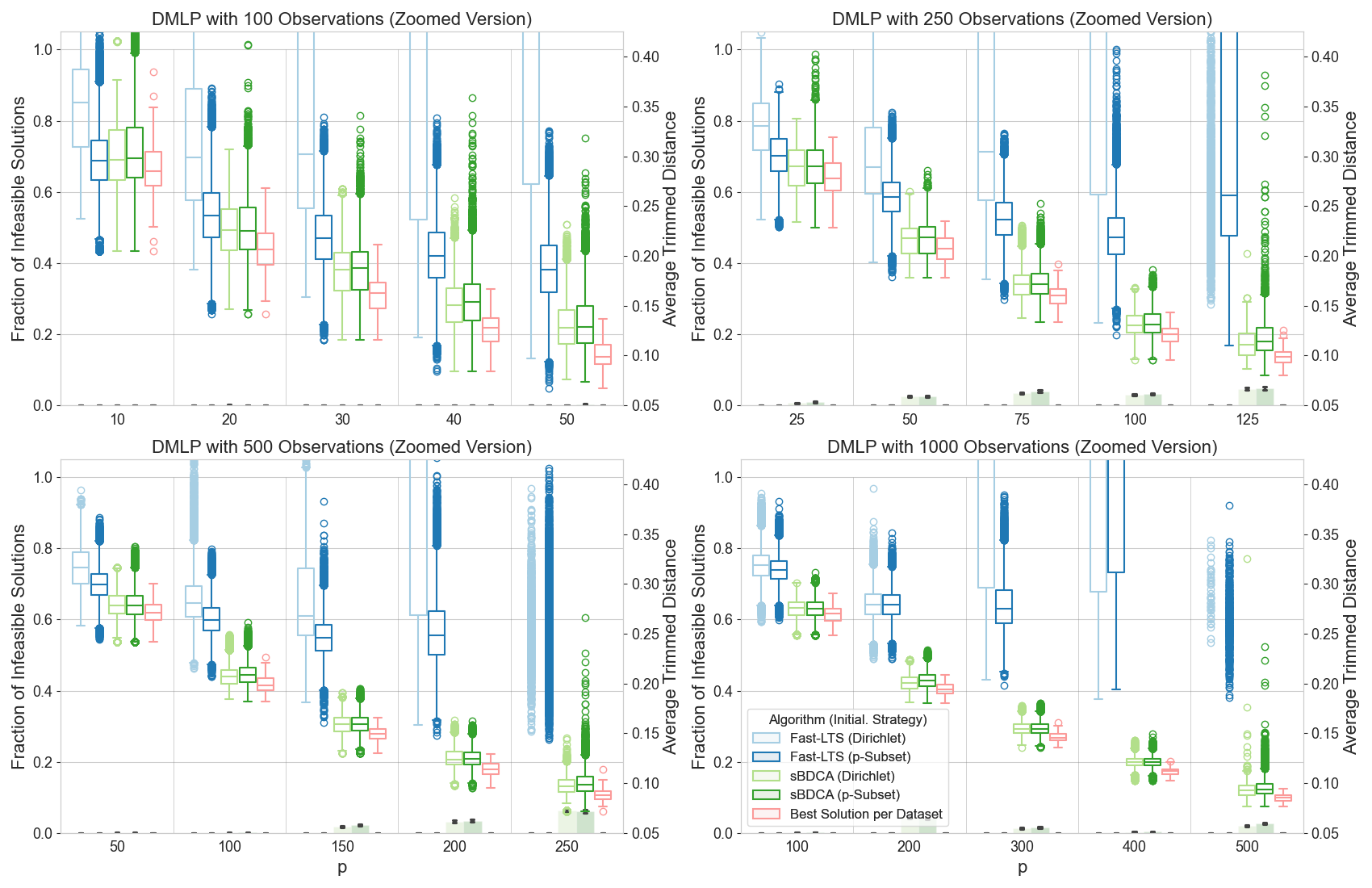}
    \caption[Performance comparison using the \gls{DMLP} and focusing on algorithmic reliability and output quality.]{Performance comparison between \gls{sBDCA} with preconditioning and Fast-\gls{LTS} using the \gls{DMLP} and focusing on algorithmic reliability and output quality. 
    Each subplot displays the fraction of infeasible solutions (left y-axis, bar plots) and the \gls{ATD} of all feasible outcomes (right y-axis, box plots) across different numbers of input variables on the x-axis. 
    Results are grouped by solver type and initialisation strategy (Dirichlet, $p$-subset).
    The box plots in red display the lowest \gls{ATD} among all feasible solutions per dataset independent of the solver.}%
    \label{fig:atd_s2}%
\end{figure}

\vspace{-0.4cm}


Figure~\ref{fig:atd_s2} compares the fraction of infeasible solutions (left y-axis, bar plots) and the \gls{ATD} of all feasible outcomes (right y-axis, box plots) across the selected settings for the \gls{DMLP}.  
To better observe the performance differences, the \gls{ATD}-axis is restricted to the interval $[0.05, 0.45]$. 
A version without this restriction is provided in Appendix~\ref{sec:further_num_results}.
Overall, the results again indicate that \gls{sBDCA} with preconditioning delivers superior performance in terms of both solution quality and algorithmic reliability.
Our proposed \gls{DC} programming algorithm not only provides the lowest median \gls{ATD} for most datasets, but also delivers robust results in settings with many independent variables where the Fast-\gls{LTS} estimator breaks down. 
This performance gap is most pronounced in the subplots for datasets with $n \in \{500, 1000\}$ and $p \in \{\lceil n \cdot 0.4 \rceil, \lceil n \cdot 0.5 \rceil\}$.
In these cases, the median \gls{ATD} of \gls{sBDCA} with preconditioning remains near 0.1, while the median results of Fast-\gls{LTS} start above 0.25 and go up to 1.5.  
This discrepancy can again be explained by the worse quality of the initial regression fits in high-dimensional settings, from which the non-robust core procedure of the heuristic has limited capacity to recover from. 
Due to the same reasoning, the performance of the state-of-the-art is also significantly more sensitive to the chosen initialisation across all settings regardless of the specific value of $p$ and $n$.
In contrast, \gls{sBDCA} with preconditioning has the capacity to escape from a suboptimal initialisation due to the in Subsection~\ref{sec:initialisation_and_preconditioner} introduced hat-value-based scaling that inflates the residuals of potential leverage points. 


\begin{figure}[H]%
    \centering
    \captionsetup{justification=centering}
    \includegraphics[width=1\textwidth]{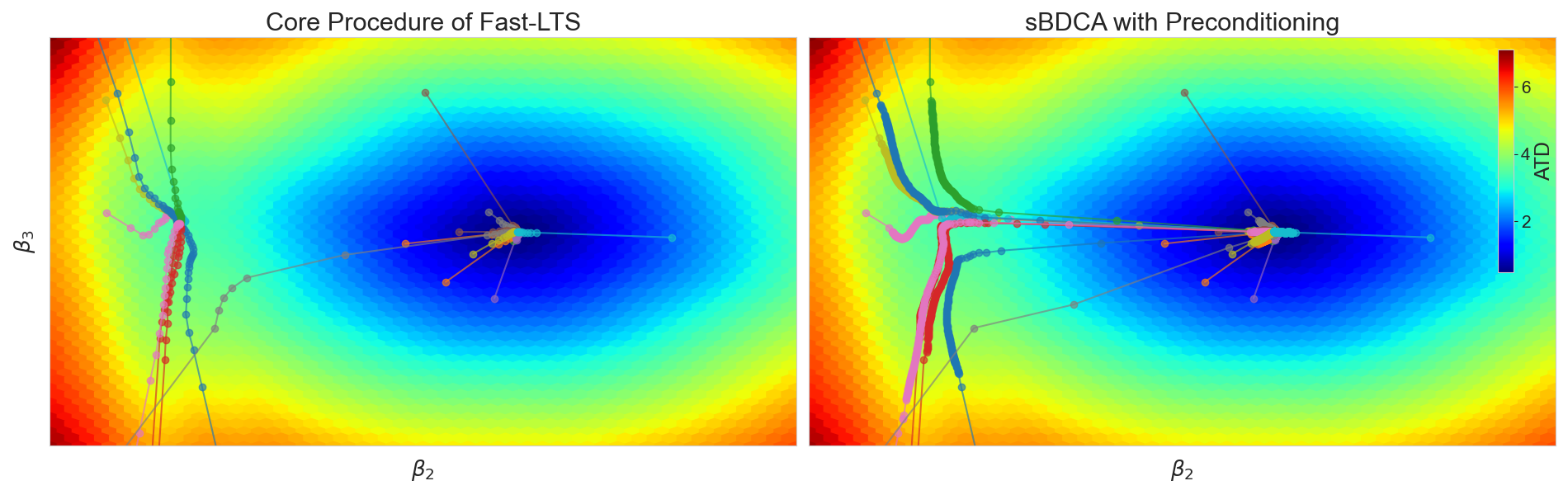}
    \caption{Optimisation Paths of Fast-\gls{LTS} and \gls{sBDCA} with Preconditioning.}%
    \label{fig:heat_map}%
\end{figure}

\vspace{-0.4cm}


\begin{figure}[H]%
    \centering
    \includegraphics[width=1\textwidth]{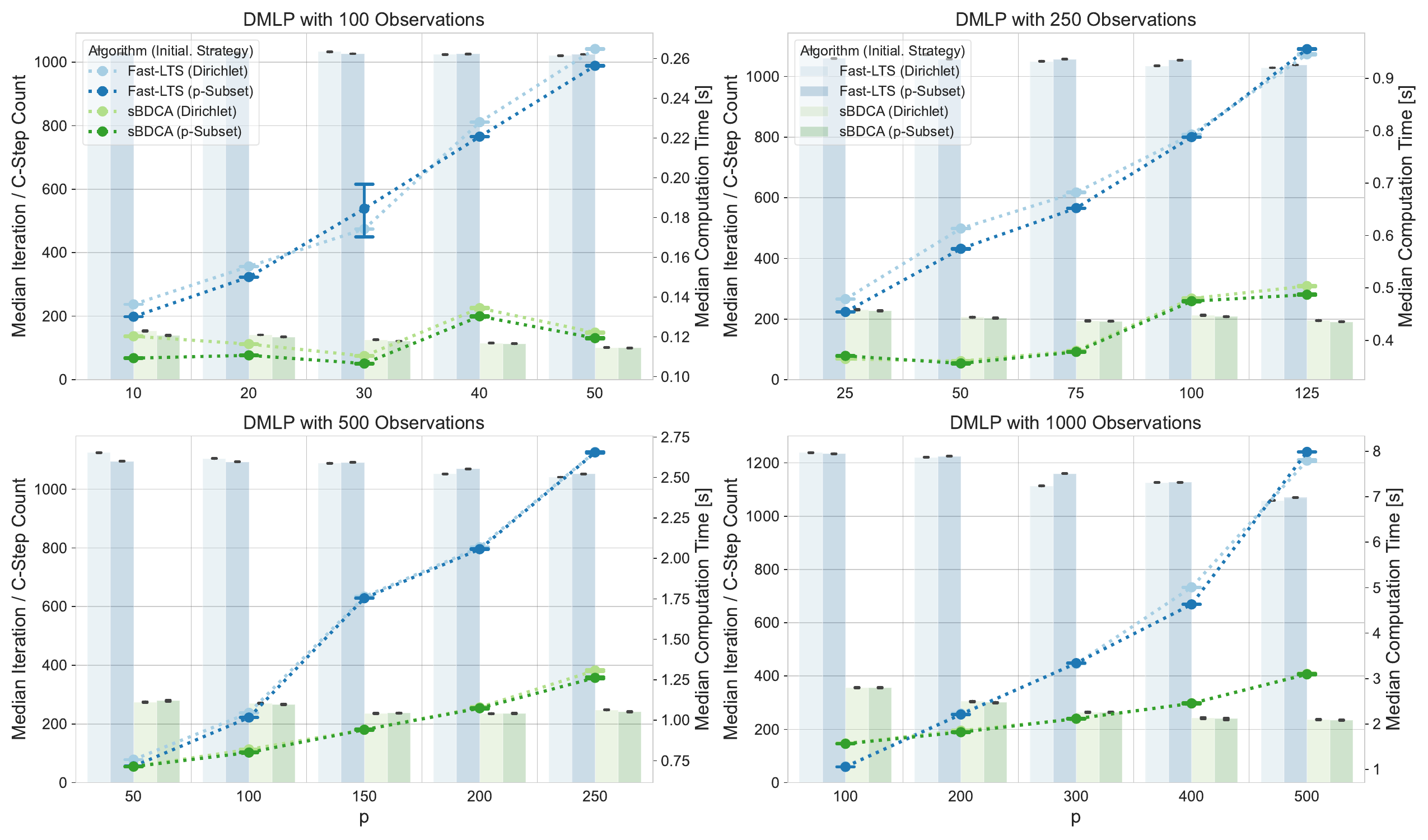}
    \caption[Performance comparison using the \gls{DMLP} and focusing on algorithmic efficieny.]{Performance comparison between \gls{sBDCA} with preconditioning and Fast-\gls{LTS} using the \gls{DMLP} and focusing on algorithmic efficiency. 
    Each subplot displays the median number of iterations / c-steps (left y-axis, bar plots) and the computation time of all feasible solutions (right y-axis, box plots) across different numbers of input variables on the x-axis. 
    Results are grouped by solver type and initialisation strategy (Dirichlet, $p$-subset).}%
    \label{fig:cpu_s2}%
\end{figure}

\vspace{-0.2cm}


To further investigate the initialisation sensitivity of Fast-\gls{LTS}, Figure~\ref{fig:heat_map} presents a heatmap of the objective landscape, overlaid with the optimisation paths of both algorithms for 20 randomly selected $p$-subset starting points.
In this example, the underlying dataset was generated using the \gls{DMLP} with $n \equiv 1000$, $p \equiv 3$, and $n_c \equiv n \cdot 0.85$. 
In the left-hand area of the objective landscape, multiple starting points are located close to a suboptimal local solution. 
The right-hand subplot demonstrates how \gls{sBDCA} with preconditioning effectively escapes this valley, while the left-hand subplot reveals how the core procedure of Fast-\gls{LTS} repeatedly becomes trapped. 
Even though, in this example, the heuristic still yields a robust end result due to a sufficient number of starting points located in less vulnerable areas, the trajectories illustrate the non-robustness of its core procedure. 
In settings with a high number of independent variables, many instances must be selected to create a $p$-subset.
Thus, the probability of obtaining at least one clean subset without outliers approaches zero, as discussed in Subsection \ref{sec:initialisation_and_preconditioner}. 
As a result, an increasing number of starting points are located closer to suboptimal local solutions, from which the heuristic cannot escape.
This then causes non-robust end results, as shown in Figure~\ref{fig:atd_s2}.


To provide the final part of the performance comparison using the \gls{DMLP}, Figure~\ref{fig:cpu_oscm} compares the median iteration / c-step count  (left y-axis, bar plots) and the computation time of all feasible outcomes (right y-axis, box plots) across the different settings. 
Overall, the results indicate similar performance trends as observed for the \gls{OSCM} in Figure~\ref{fig:cpu_oscm}. 
However, to ensure robust end results, more conservative input parameters had to be chosen for the \gls{sBDCA} with preconditioning, which cause higher computation times across the selected settings. 
Nevertheless, our proposed algorithm still outperforms Fast-\gls{LTS} in terms of runtime in nearly all instances.



\subsection{Synthetic Data Experiments with High Contamination}\label{sec:data_exp_contamination}


After evaluating the algorithmic performance relative
to the number of instances and input variables, we next focus on a higher contamination level. 
For this purpose, we revisit the \gls{OSCM} and the \gls{DMLP} and set $n_c =\lceil n \cdot 0.55 \rceil$, which corresponds to an outlier proportion of 45\% in each dataset.
We consider this more extreme setting because, in theory, the \gls{LTS} estimator can tolerate contamination levels of up to 50\%. 
Given this outlier share, we then generate 100 random datasets for each combination of  $n \in \{250, 500\}$ and $p \in \{10, 25, 50\}$, where we fix $h =\lceil \frac{n}{2} \rceil$ at its lower bound.
In total, this yields 800 datasets.
For all of these datasets, both algorithms are again initialised 500 times based on the two previously used strategies as well as once with a deterministic strategy using the data mean $\widehat{\boldsymbol{\beta}}_0 = \bar{\mathbf{x}} \coloneqq n^{-1} \sum_{i=1}^n \mathbf{x}_i$.


Let us now focus on the results for the \gls{OSCM}, which are illustrated in Figure~\ref{fig:oscm_contamination}.
Overall, both algorithms do not break down, regardless of the initialisation strategy, which indicates that a high share of regression outliers does not challenge the procedures. 
Moreover, in settings with more input variables, the \gls{sBDCA} with preconditioning again outperforms Fast-\gls{LTS}.


\begin{figure}[H]%
    \centering
    \includegraphics[width=1\textwidth]{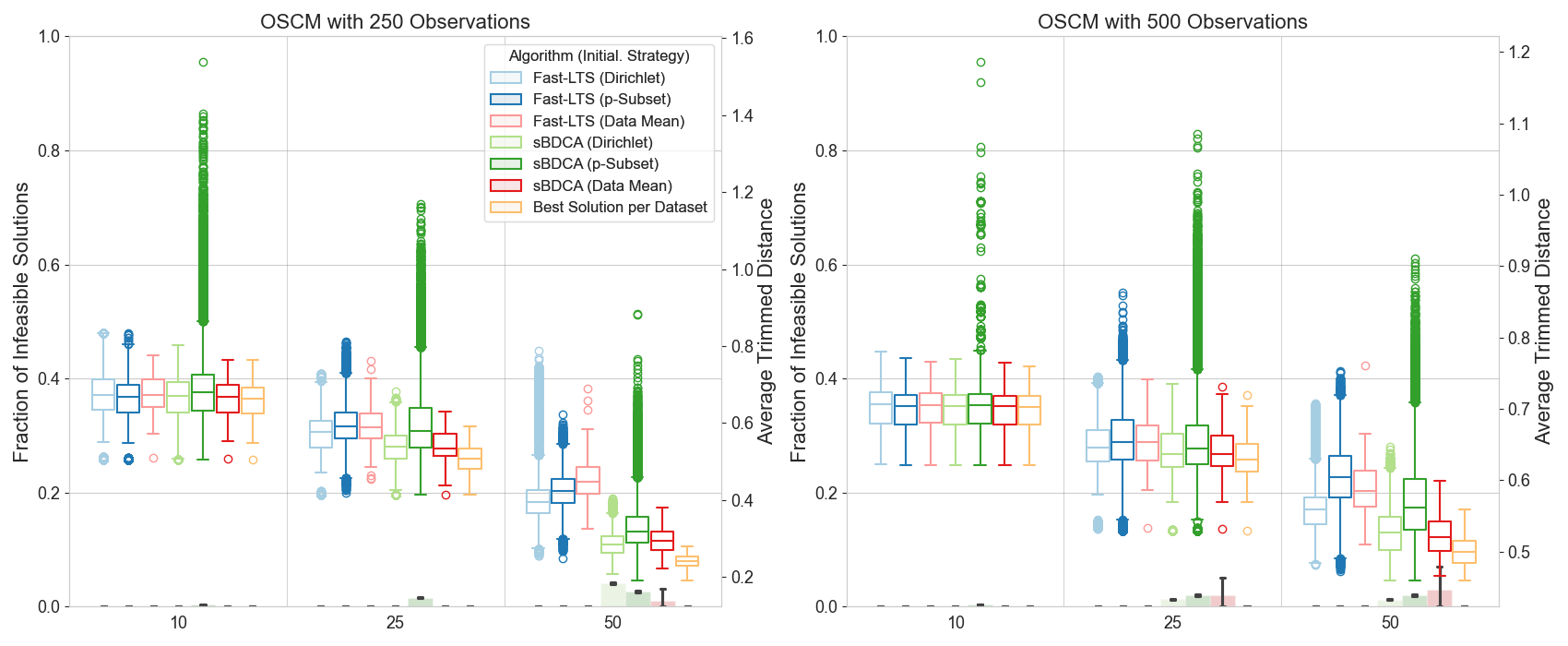}
    \caption[Performance comparison using the \gls{OSCM} with a high contamination level.]{Performance comparison between \gls{sBDCA} with preconditioning and Fast-\gls{LTS} using the \gls{OSCM} with an outlier proportion of 45\%. 
    Each subplot displays the fraction of infeasible solutions (left y-axis, bar plots) and the \gls{ATD} of all feasible outcomes (right y-axis, box plots) across different numbers of input variables on the x-axis. 
    Results are grouped by solver type and initialisation strategy (Dirichlet, $p$-subset, Data Mean).
    The box plots in orange display the lowest \gls{ATD} among all feasible solutions per dataset independent of the solver.}%
    \label{fig:oscm_contamination}%
\end{figure}


Having evaluated the results for the \gls{OSCM}, we next consider the \gls{DMLP}.
The corresponding results are shown in Figure~\ref{fig:dmlp_contamination} and differ notably from those observed for the \gls{OSCM}.
Both algorithms break down under the $p$-subset and Dirichlet initialisations, suggesting that their performance is sensitive to the presence of many leverage points.
In contrast, the \gls{ATD} with data mean initialisation remains robust for both algorithms, indicating that this initialisation strategy may be particularly beneficial in settings with high contamination levels.


\begin{figure}[H]%
    \centering
    \includegraphics[width=1\textwidth]{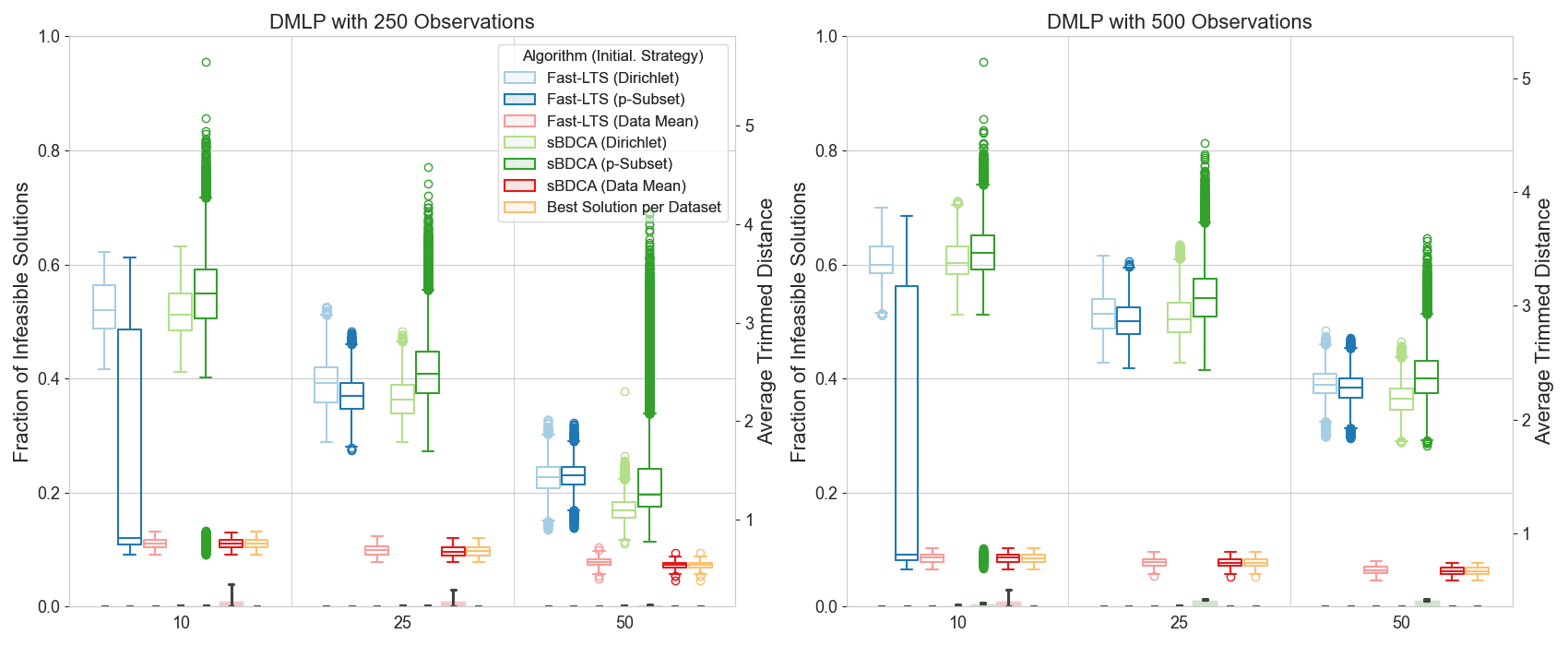}
    \caption[Performance comparison using the \gls{DMLP} with a high contamination level.]{Performance comparison between \gls{sBDCA} with preconditioning and Fast-\gls{LTS} using the \gls{DMLP} with an outlier proportion of 45\%. 
    Each subplot displays the fraction of infeasible solutions (left y-axis, bar plots) and the \gls{ATD} of all feasible outcomes (right y-axis, box plots) across different numbers of input variables on the x-axis. 
    Results are grouped by solver type and initialisation strategy (Dirichlet, $p$-subset, Data Mean).
    The box plots in orange display the lowest \gls{ATD} among all feasible solutions per dataset independent of the solver.}%
    \label{fig:dmlp_contamination}%
\end{figure}

\vspace{-0.4cm}



\subsection{Real-World Data Experiments}


Building on the synthetic-data comparison, this subsection evaluates algorithmic performance using two high-dimensional real-world datasets from the \href{https://archive.ics.uci.edu/}{UCI Machine Learning Repository}. 
Specifically, we selected the \gls{MSD} and the \gls{SCD}, originally introduced by \citet{bertin2011} and \citet{HAMIDIEH2018}, respectively.
The \gls{MSD} consists of metadata and audio features for one million predominantly Western songs released between 1922 and 2011. 
The dataset supports various tasks, but our focus is on predicting the year of a song's release across more than 450,000 training instances. 
Each song is represented by 90 features, including key acoustic attributes such as pitch, timbre, and loudness. 
The \gls{SCD} comprises data on over 21,000 superconductors. 
The main predictive task is to estimate the log-transformed superconducting critical temperature using 81 features derived from the materials’ chemical compositions. 
These features encompass descriptors related to valence, electron affinity, atomic mass, and other physicochemical characteristics. 
For further details on both datasets, we refer the reader to \citet{bertin2011} and \citet{HAMIDIEH2018}.


In line with the experiments in Subsection~\ref{sec:synthetic_data_exp}, we evaluate performance based on algorithmic efficiency, reliability, and output quality.
Given that both datasets have a fixed number of variables, we stratify observations and sample accordingly to analyse performance as a function of $p$.
Following this procedure, we construct 100 stratified subsets for each sample size $n \in \{\lfloor \frac{p}{0.1} \rfloor, \lfloor \frac{p}{0.2} \rfloor, \lfloor \frac{p}{0.3} \rfloor, \lfloor \frac{p}{0.4} \rfloor, \lfloor \frac{p}{0.5} \rfloor\}$, resulting in a total of 1,000 subsamples $(= 5 \cdot 100 \cdot 2)$. 
Both algorithms are then initialised 500 times per setting using both initialisation strategies.


Building on this evaluation framework, Figure~\ref{fig:MSD_SCD_ATD_Inf} compares the fraction of infeasible solutions (left y-axis, bar plots) and the \gls{ATD} of all feasible outcomes (right y-axis, box plots) across the selected settings. 
Overall, \gls{sBDCA} with preconditioning and Dirichlet initialisation provides the best performance in terms of algorithmic reliability and output quality.
However, in contrast to the synthetic data experiments, the performance of our proposed \gls{DC} programming algorithm with the $p$-subset initialisation is more dataset-dependent. 
In particular, for settings with fewer observations, the \gls{ATD} exhibits greater variability and predominantly fails to significantly outperform the results of Fast-\gls{LTS}.
This observation may be attributed to a more complicated data structure characterised by a non-normal dependent variable and correlated input variables.
The latter then makes it more challenging to mitigate the effects of an unfavorable $p$-subset initialisation, even when the proposed preconditioner based on the hat values is applied.


\begin{figure}[H]%
    \centering
    \includegraphics[width=1\textwidth]{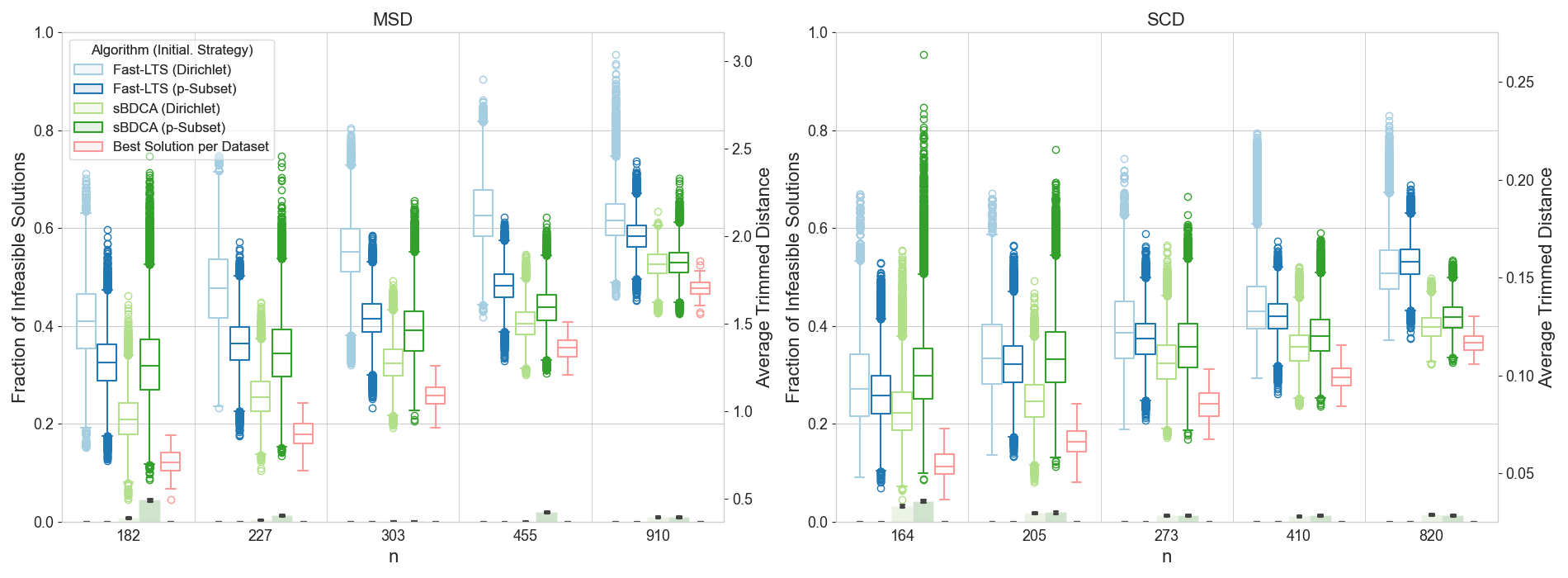}
    \caption[Performance comparison using the \gls{MSD} / \gls{SCD} and focusing on algorithmic reliability and output quality.]{Performance comparison between \gls{sBDCA} with preconditioning and Fast-\gls{LTS} using the \gls{MSD} / \gls{SCD} and focusing on algorithmic reliability and output quality. 
    Each subplot displays the fraction of infeasible solutions (left y-axis, bar plots) and the \gls{ATD} of all feasible outcomes (right y-axis, box plots) across different numbers of input variables on the x-axis. 
    Results are grouped by solver type and initialisation strategy (Dirichlet, $p$-subset).
    The box plots in red display the lowest \gls{ATD} among all feasible solutions per dataset independent of the solver.}%
    \label{fig:MSD_SCD_ATD_Inf}%
\end{figure}

\vspace{-0.3cm}


To complete the comparative analysis,
Figure~\ref{fig:MSD_SCD_CT_Iter}   compares the median iteration / c-step count  (left y-axis, bar plots) and the computation time of all feasible outcomes (right y-axis, box plots) across the different settings.
In comparison to the experiments in Subsection~\ref{sec:synthetic_data_exp}, the runtime of \gls{sBDCA} with preconditioning changes depending on the initialisation strategy.
This can be explained by the input parameters, which had to be chosen more conservatively for the $p$-subset strategy to ensure an acceptable output quality.
Nevertheless, the overall conclusion remains the same, as our proposed algorithm outperforms Fast-\gls{LTS} in terms of computation time. 


\vspace{-0.2cm}

\begin{figure}[H]%
    \centering
    \includegraphics[width=1\textwidth]{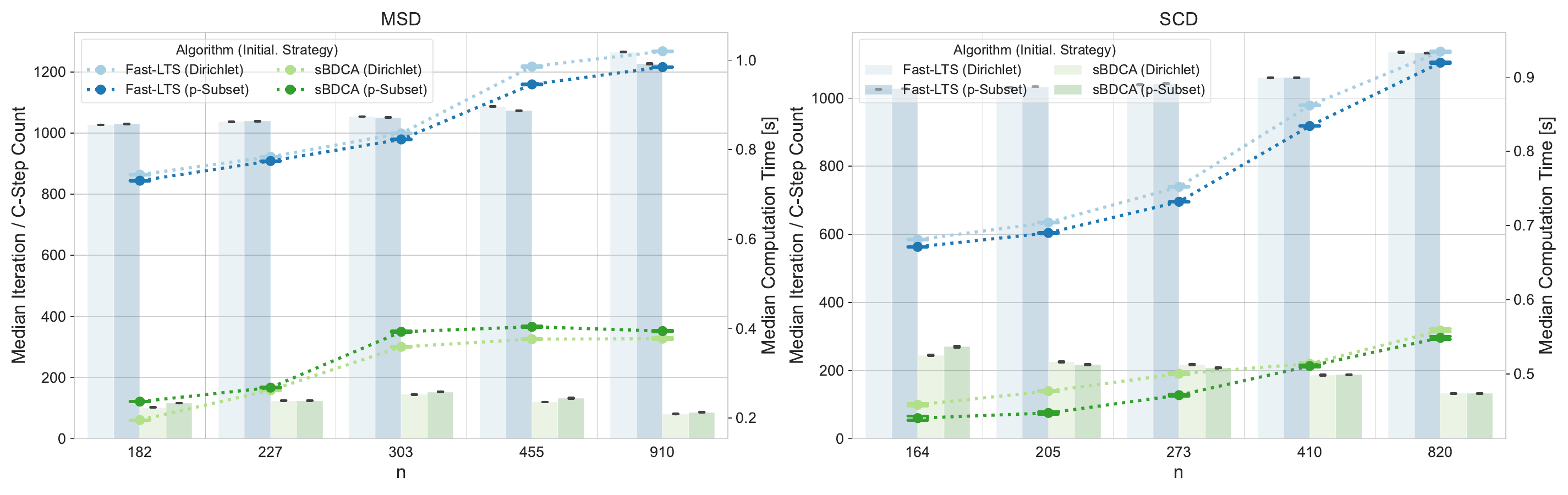}
    \caption[Performance comparison using the \gls{MSD} / \gls{SCD} and focusing on algorithmic efficiency.]{Performance comparison between \gls{sBDCA} with preconditioning and Fast-\gls{LTS} using the \gls{MSD} / \gls{SCD} and focusing on algorithmic efficiency. 
    Each subplot displays the median number of iterations / c-steps (left y-axis, bar plots) and the computation time of all feasible solutions (right y-axis, box plots) across different numbers of input variables on the x-axis. 
    Results are grouped by solver type and initialisation strategy (Dirichlet, $p$-subset).}%
    \label{fig:MSD_SCD_CT_Iter}%
\end{figure}

\vspace{-0.5cm}


\section{Application to Outlier Detection}\label{sec:Application_to_Outlier_Detection}



In this section, we present a novel application of robust regression for outlier detection in sports analytics.
Specifically, our work focuses on empirical football player valuation and builds on insights from the systematic literature review by \citet{Franceschi2024}.
In their survey, the authors outline several challenges and opportunities within this field, such as better data availability and dissemination to enhance future research. 
\citet{Franceschi2024} also identify certain independent variables, such as a player's age and minutes played, as key drivers of the \gls{MaVa}, but highlight the potential to explore additional attributes. 
Moreover, while \gls{OLS} has been the most commonly used method, the review advocates for the exploration of more sophisticated approaches, such as robust regression. 
The results of \citet{Kirschstein2019} are mentioned as a first example due to the use of an MM-estimator for \gls{MaVa} predictions and outlier analysis. 
However, their work did not account for (bad) leverage points, nor did it provide detailed insights into regression outliers, as the input variables had been merged using principal component analysis.
In this section, we propose an alternative methodology for outlier detection that explicitly addresses these limitations identified in the literature on empirical football player valuation.


As an initial part of our empirical analysis, we constructed a novel dataset designed to predict the \gls{MaVa} updates of football players from Germany’s \href{https://www.bundesliga.com/en/bundesliga}{1.~Bundesliga} at the end of the 2022/2023 season. 
The data was collected through publicly available APIs for two widely used football platforms: \href{https://www.transfermarkt.co.uk/}{Transfermarkt} and \href{https://www.sofascore.com/}{Sofascore}. 
The former provides historical market valuations for players and is considered one of the most prominent data sources for empirical research in this area [cf. \citet[p. 612]{MULLER2017}]. 
In contrast, \href{https://www.sofascore.com/}{Sofascore} offers publicly available player ratings and detailed match statistics, including metrics such as expected goals, minutes played, and positional information [cf. \citet[p. 676]{Ball2025}]. 


Overall, the considered sample comprises 449 observations and six independent variables, each describing various aspects of the player, their performance, or the performance of their team during the season. 
Table~\ref{tab:Independent_Vars} provides a summary of the available features, with the initial two characters of each variable name indicating the respective data source (i.e., \texttt{tm\_xxx} for \href{https://www.transfermarkt.co.uk/}{Transfermarkt} and \texttt{sc\_xxx} for \href{https://www.sofascore.com/}{Sofascore}). 
It is important to highlight that we restricted the data to players who participated in at least one rated match during the respective season on \href{https://www.sofascore.com/}{Sofascore} to ensure that all features could be computed. 
The final two columns of Table~\ref{tab:Independent_Vars} display the \gls{PCC} and \gls{SCC} for each independent variable with respect to the outcome. 
Based on these measures, the feature showing the strongest positive correlation with the response is the average rating per game, where the \gls{SCC} exceeds the \gls{PCC}, indicating potential non-linear effects. 
Beyond that, both the total number of minutes played and the average expected goals per game show a strong positive correlation with the outcome, where the \gls{SCC} again exceeds the \gls{PCC}. 
The strongest negative correlation with the \gls{MaVa} at the end of the season is observed for the team's final league position.
This aligns with expectations, as teams that can afford players with higher \glspl{MaVa} typically perform better over the season.
Moreover, players from teams that perform well throughout the season usually witness a rise in their \glspl{MaVa} by the end of the considered time period.
The player's age demonstrates a smaller negative \gls{PCC} and \gls{SCC} with the response. 
This variable is expected to exhibit a parabolic relationship with the outcome, where \gls{MaVa} tends to be higher for players in their prime (typically around 25 years old) and decreases for both younger and older players.
While \gls{SCC} can capture certain non-linear effects, provided the relationship remains monotonic, the \gls{PCC} is limited to capturing only linear relationships.
As a result, neither metric is likely to capture the age effect adequately. 


\begin{table}[H]
    \scriptsize
    \centering
    \captionsetup{justification=centering}
    \caption{Description of Independent Variables and Outcome.}
    \label{tab:Independent_Vars}
    \begin{tabular}{c p{2.85cm} p{0.7cm} lrr}
    \toprule
    \# & Variable & Short & Description & \gls{PCC} & \gls{SCC} \\
    \midrule
    0 & \texttt{tm\_player\_mvE} & mvE & \gls{MaVa} at the end of the season. & 1.000 & 1.000 \\
    1 & \texttt{tm\_player\_age} & Age & Age at the beginning of the season. & -0.198 & -0.198 \\
    2 & \texttt{sc\_player\_ratingAvg} & Rating & Average rating per rated game. & 0.422 & 0.546 \\
    3 & \texttt{sc\_player\_xGoalsAvg} & xG & Average expected goals per rated game. & 0.367 & 0.460 \\
    4 & \texttt{sc\_player\_timePlayedSum} & TP & Minutes played in the whole season. & 0.355 & 0.571 \\
    5 & \texttt{sc\_team\_rank} & Pos. & The end position of player's team in the league table. & -0.439 & -0.451 \\
    6 & \texttt{sc\_team\_rank\_best7} & Pos7 & Flag for the top 7 teams at season's end. & 0.402 & 0.383 \\
    \bottomrule
    \end{tabular}
\end{table}

\vspace{-0.2cm}


After inspecting the different variables, Table \ref{tab:MV_per_Team} presents a summary of  the \glspl{MaVa} for each team.
The first two columns display each team's final league position and name, respectively. 
The third column reports the number of players per team, denoted as $n_{\text{T}} \in \mathbb{N}$. 
Note that players who participated in matches for two different teams during the considered time period are assigned to the club with which they finished the season. 
For instance, \href{https://www.transfermarkt.co.uk/yann-sommer/profil/spieler/42205}{Yann Sommer} initially played for \href{https://www.transfermarkt.co.uk/borussia-monchengladbach/startseite/verein/18}{Borussia M'gladbach} but transferred to \href{https://www.transfermarkt.co.uk/fc-bayern-munchen/startseite/verein/27}{Bayern München} during the winter break.
Therefore, he was assigned to the latter team. 
The final two columns of Table \ref{tab:MV_per_Team} show the median \gls{MaVa} per team at the beginning (mvB) and end (mvE) of the season in millions of Euros, respectively. 
Both columns indicate that the \gls{MaVa} tends to decrease as a team's league position declines, consistent with the findings presented in Table~\ref{tab:Independent_Vars}. 
Furthermore, the median \gls{MaVa} of players from teams that secured a position ensuring participation in international competitions the following season is generally higher. 
In the lower half of the table, the median \gls{MaVa} per team is very similar, but for teams at the bottom it decreased compared to the beginning of the season.


\begin{table}[H]
    \scriptsize
    \centering
    \captionsetup{justification=centering}
    \caption{Median \gls{MaVa} before and after the Season per Team.}
    \label{tab:MV_per_Team}
    \begin{tabular}{>{\centering}p{0.35cm} lcrr | >{\centering}p{0.35cm} lcrr} 
    \toprule 
    Pos. & Team Name & $n_{\text{T}}$ & mvB & mvE & Pos. & Team Name & $n_{\text{T}}$ & mvB & mvE \\
    \midrule 
    1 & FC Bayern München & 26 & 32.50 & 29.00 & 10 & Borussia M'gladbach & 25 & 5.00 & 5.00 \\
    2 & Borussia Dortmund & 29 & 11.00 & 13.00 & 11 & 1. FC Köln & 25 & 3.00 & 3.50 \\
    3 & RB Leipzig & 22 & 19.00 & 14.50 & 12 & TSG Hoffenheim & 27 & 6.00 & 3.50 \\
    4 & 1. FC Union Berlin & 24 & 3.75 & 5.00 & 13 & SV Werder Bremen & 22 & 3.00 & 3.00 \\
    5 & SC Freiburg & 24 & 3.50 & 4.50 & 14 & VfL Bochum 1848 & 24 & 1.60 & 1.45 \\
    6 & Bayer 04 Leverkusen & 22 & 16.00 & 16.50 & 15 & FC Augsburg & 27 & 3.50 & 3.50 \\
    7 & Eintracht Frankfurt & 24 & 7.50 & 7.50 & 16 & VfB Stuttgart & 26 & 3.00 & 2.50 \\
    8 & VfL Wolfsburg & 23 & 6.00 & 8.00 & 17 & FC Schalke 04 & 28 & 2.00 & 1.20 \\
    9 & 1. FSV Mainz 05 & 26 & 2.75 & 3.50 & 18 & Hertha BSC & 25 & 4.00 & 2.00 \\
    \bottomrule 
    \end{tabular}   
\end{table}

\vspace{-0.2cm}


To obtain more detailed insights into the data, we conduct a univariate analysis of the individual variables. 
In Figure~\ref{fig:histo_mv_before_and_after_season}, the histogram in blue color displays the \gls{MaVa} at the end of the season. 
The distribution is highly skewed and exhibits substantial variance, driven by a small number of players with exceptionally high \glspl{MaVa}. 
To address this, we applied a logarithmic transformation to stabilise the variance, to obtain an approximate normal distribution, and to permit negative outcome predictions. 
The green histogram shows the log-transformed variable, with the probability density function of a normal distribution superimposed.
In the plot on the right-hand side of the first row, the log-transformed \gls{MaVa} is plotted against each team's final league position using box- and strip-plots.
These visualisations further confirm the negative correlation reported in Table~\ref{tab:Independent_Vars}. 
The second row of Figure~\ref{fig:histo_mv_before_and_after_season} presents three scatter-plots illustrating the relationships between the dependent variable and several covariates. 
The left plot shows the square-root transformed average expected goals per game (\texttt{sc\_player\_xGoalsAvg\_sqrt}) against the log-transformed \gls{MaVa}, where the data reveals that multiple players (e.g., goalkeepers) have an average value of zero. 
The right plot presents the square-root transformed total minutes played (\texttt{sc\_player\_timePlayedSum\_sqrt}) against the log-transformed \gls{MaVa}, showing increased variance in \gls{MaVa} among players with greater total playing time. 
The center plot displays player age against the dependent variable, which confirms the anticipated parabolic relationship.


\begin{figure}[H]%
    \centering
    \captionsetup{justification=centering}
    \includegraphics[width=1\textwidth]{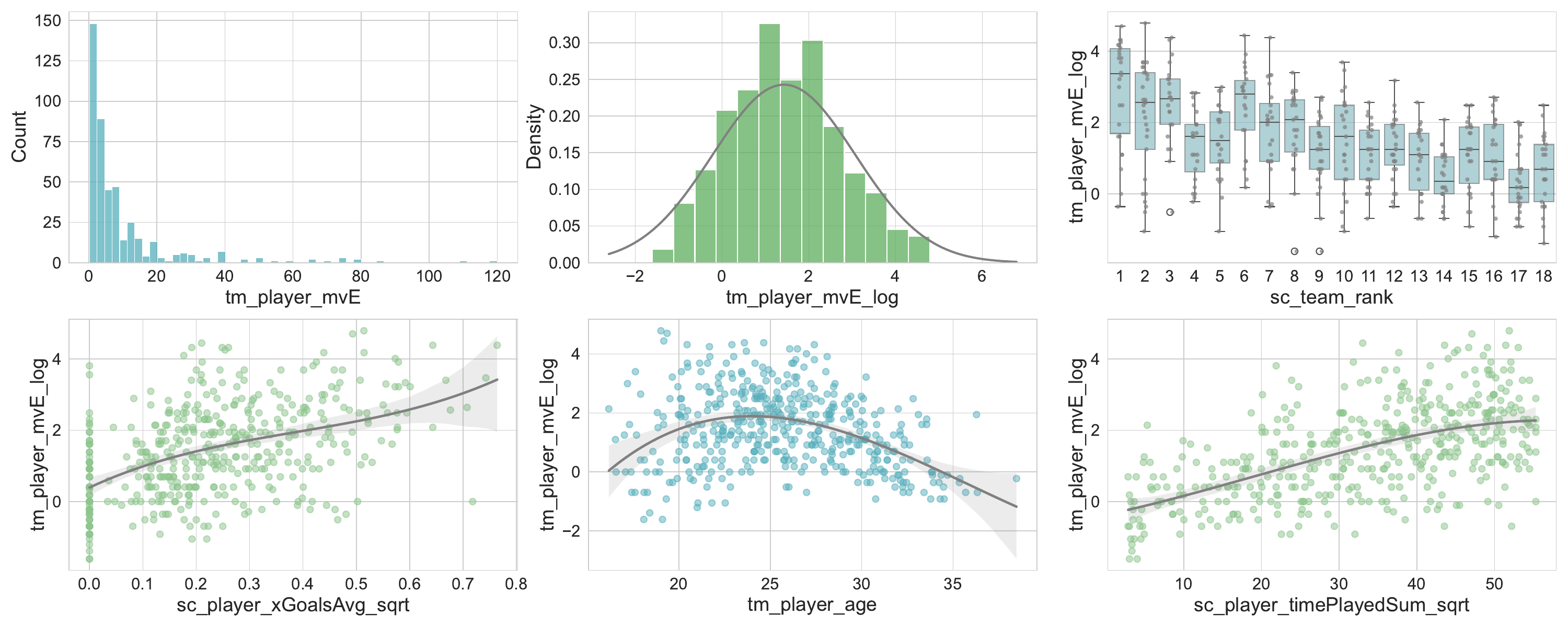}
    \caption{Variables Transformations and Univariate Analyses.}%
    \label{fig:histo_mv_before_and_after_season}%
\end{figure}

\vspace{-0.2cm}


To complete the univariate analysis, Figure~\ref{fig:scatterplot} presents several scatterplots illustrating the relationship between the log-transformed \gls{MaVa} at the end of the season and the average rating per game. 
The left plot shows the full range of the independent variable, where a single observation with an exceptionally low player rating appears markedly distant from the remaining data points. 
In the centre plot, this extreme outlier is removed, revealing another outlier characterised by a very high average rating but a relatively low \gls{MaVa}. 
These observations correspond to players who participated in only a single game, in which their performance was exceptionally below or above their typical level, respectively. 
Finally, the right plot excludes both very low and very high average player ratings, confirming the positive \gls{PCC} and \gls{SCC} reported in Table~\ref{tab:Independent_Vars}. 
In this plot, the relationship also appears largely linear, suggesting that the higher \gls{SCC} is likely attributable to outlying observations, which obscure the underlying relationship.


\begin{figure}[H]%
    \centering
    \captionsetup{justification=centering}
    \includegraphics[width=1\textwidth]{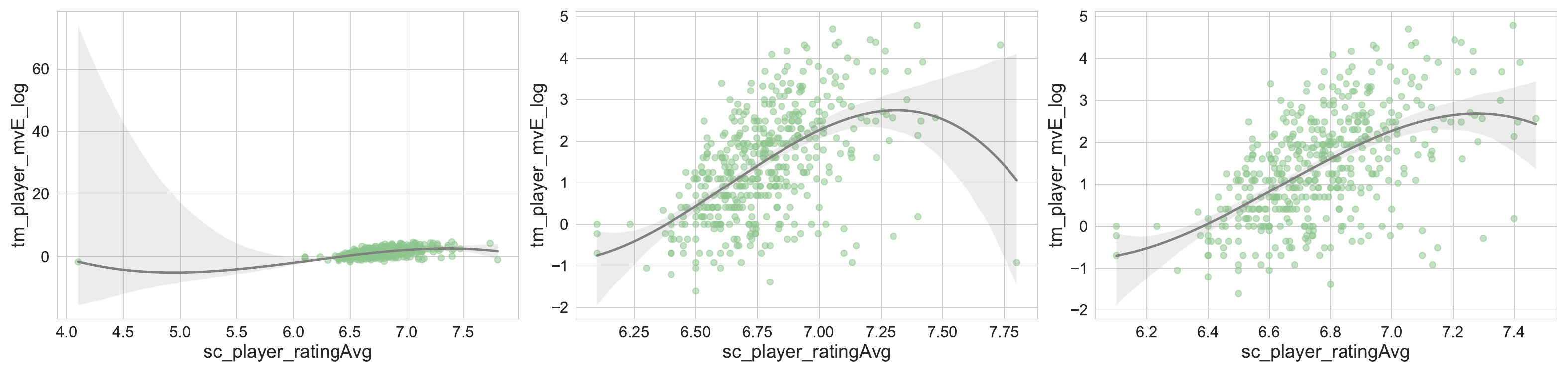}
    \caption{Scatter-Plots between the \gls{MaVa} and the Average Rating.}%
    \label{fig:scatterplot}%
\end{figure}

\vspace{-0.2cm}


After the univariate analyses, we proceed with the estimation of the regression coefficients. Table \ref{tab:model_summaries} presents the results obtained from each model. 
Overall, our proposed \gls{DC} programming algorithm (with $h = \lceil \frac{n + p + 1}{2} \rceil$) identifies a regression fit with an almost perfect (adjusted) $\text{R}^2$, resulting in a highly significant model. 
This high level of significance extends to the individual independent variables, all of which have p-values below 0.01. 
Examining the estimated regression coefficients, all features influence the log-transformed \gls{MaVa} in a manner consistent with the \gls{PCC} and \gls{SCC} shown in Table \ref{tab:Independent_Vars}. 
In contrast, the \gls{OLS} model summary reveals that the (adjusted) $\text{R}^2$ value decreased substantially, though it remains at a high level. 
The overall model is still highly significant, but the p-value of the F-statistic shows a slight increase. 
A more notable effect is observed when examining the coefficients of the individual independent variables and their significance levels. 
Most regression coefficients have decreased in absolute magnitude, with the largest reduction observed for the average rating and expected goals per game, likely due to the leverage points identified in Figure~\ref{fig:histo_mv_before_and_after_season} and Figure~\ref{fig:scatterplot}. 
Furthermore, the indicator variable for the top seven teams in the league table is no longer statistically significant.
This result illustrates the sensitivity of significance levels to changes in sample composition.
Consequently, conventional thresholds, as discussed by \citet[pp. 588--589]{Franceschi2024}, should be interpreted with caution, and future research should account for their instability.


\begin{table}[H]%
    \centering
    \scriptsize
    \caption{Model Summary of \gls{OLS} (left) and \gls{sBDCA} with Preconditioning (right).}
    \begin{varwidth}{\linewidth}
\begin{verbatim}
=======================================================  =======================================================
 Dep. Variable:  mvEnd_log   R-squared:          0.772    Dep. Variable:  mvEnd_log   R-squared:          0.977
 Model:                OLS   Adj. R-squared:     0.769    Model:              sBDCA   Adj. R-squared:     0.976
 No. Observations:     449   F-statistic:        213.5    No. Observations:     229   F-statistic:        1323.
 Df Residuals:         441   Prob (F-stat.): 2.60e-137    Df Residuals:         221   Prob (F-stat.): 1.49e-176
 Df Model:               7   Log-Likelihood:   -416.51    Df Model:               7   Log-Likelihood:    69.406
 BIC:             881.8709   AIC:             849.0147    BIC:             -95.3427   AIC:            -122.8125
=======================================================  =======================================================
                                  coef  std err p-value                                    coef  std err p-value
-------------------------------------------------------  -------------------------------------------------------
 Intercept                      -7.258   1.251   0.000    Intercept                     -11.437   0.655   0.000
 tm_player_age                   0.309   0.073   0.000    tm_player_age                   0.333   0.033   0.000
 tm_player_age_squared          -0.008   0.001   0.000    tm_player_age_squared          -0.009   0.001   0.000
 sc_team_position_best7          0.141   0.114   0.216    sc_team_position_best7          0.372   0.047   0.000
 sc_player_ratingAvg             0.776   0.135   0.000    sc_player_ratingAvg             1.315   0.075   0.000
 sc_team_position               -0.086   0.011   0.000    sc_team_position               -0.056   0.005   0.000
 sc_player_xGoalsAvg_sqrt        1.355   0.193   0.000    sc_player_xGoalsAvg_sqrt        2.540   0.083   0.000
 sc_player_timePlayedSum_sqrt    0.044   0.002   0.000    sc_player_timePlayedSum_sqrt    0.038   0.001   0.000
=======================================================  =======================================================
\end{verbatim}
    \end{varwidth}
    \label{tab:model_summaries}
\end{table}

\vspace{-0.2cm}


To extend the outlier analysis performed by \citet{Kirschstein2019}, we propose leveraging the regression results to generate the outlier map introduced by \citet{Rousseeuw1990}.
The main idea behind the map is to plot standardised residuals against robust distances $\text{RD}(\mathbf{x}_i)$ to detect different types of outliers. 
The robust distances can be calculated using the \gls{MCD} method introduced by \citet{rousseeuw1984, rousseeuw1985}, which was implemented in the \texttt{sklearn} Python library based on the fast algorithm by \citet{Rousseeuw1999}. 
Figure~\ref{fig:outlier_map} shows the map for each model.
The black horizontal lines in the map highlight the interval $[-2.5, 2.5]$, which represents the tolerance for regression outliers, while the black vertical line marks the interval $[0,  \sqrt{\chi^2_{p-1, .975}}]$, with its right-hand side indicating the boundary for leverage points. 
Based on these markers, it is straightforward to classify the outliers into the different categories introduced in Subsection~\ref{sec:outliers_in_regression}. 
When comparing the maps of both models, it is particularly striking that almost no regression outliers are found using \gls{OLS}, which is likely due to the inflated variance of the residuals caused by the outliers. 
In contrast, on the map of \gls{sBDCA} with preconditioning, the observations are much more scattered, indicating numerous bad leverage points as well as multiple outliers in the other categories. 


\vspace{-0.2cm}

\begin{figure}[H]%
    \centering
    \captionsetup{justification=centering}
    \includegraphics[width=1\textwidth]{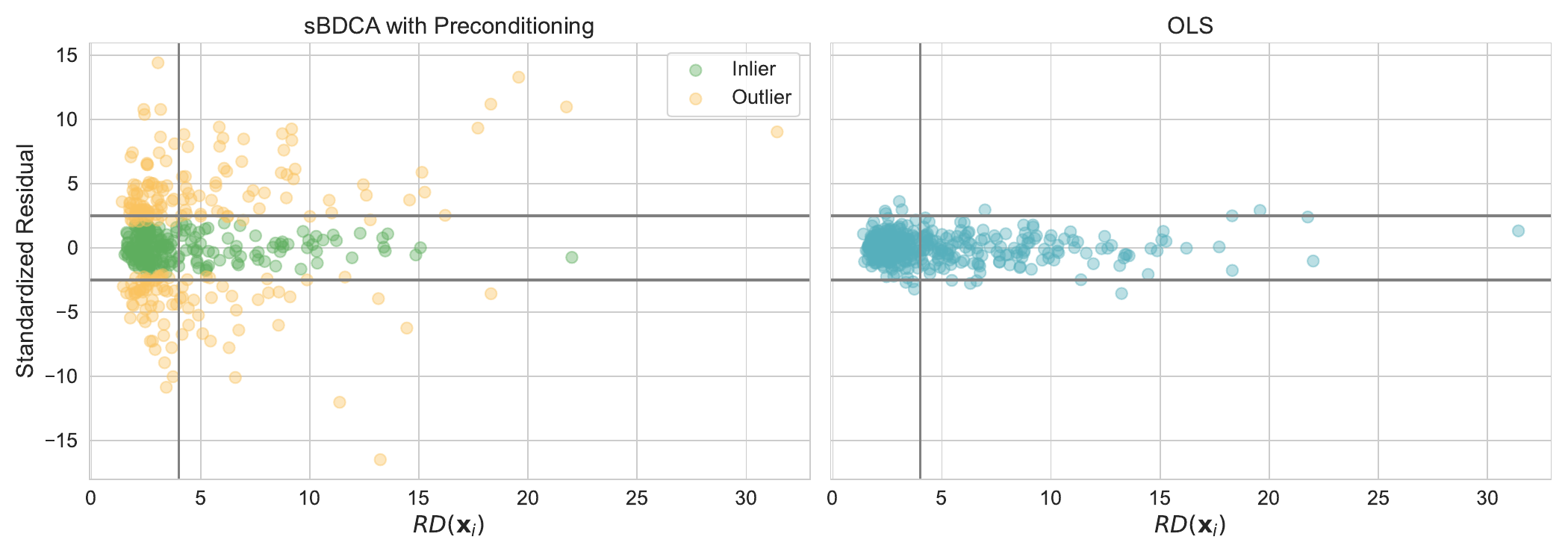}
    \caption{Outlier Map for each Model.}%
    \label{fig:outlier_map}%
\end{figure}

\vspace{-0.4cm}


To further investigate model differences, we examine the largest prediction discrepancies between \gls{OLS} and \gls{sBDCA} with preconditioning for the inliers identified in Figure~\ref{fig:outlier_map}. 
For this purpose, Table \ref{tab:largest_deviations_inliers} presents the 15 players with the largest differences between the predicted \glspl{MaVa}. 
The first column lists the player names, while the second through seventh columns display the realisations of the independent variables and the RD($\mathbf{x}_i$). 
Subsequently, the eighth and ninth columns show the predicted \glspl{MaVa} from each model. 
This is followed by the true \gls{MaVa} at the end of the season, along with the next and second-next available valuations, denoted as NV and SNV, respectively. 
Overall, the results demonstrate that the predictions of our proposed \gls{DC} programming algorithm are closer to the true \glspl{MaVa} at the end of the season. 
Furthermore, the NV and SNV show that our model more accurately captures most players’ future \gls{MaVa} trends.
Thus, the model outputs could also help to identify athletes who may be currently over- or undervalued.


\begin{table}[H]
    \scriptsize
    \centering
    \caption{Prediction Deviations between \gls{OLS} and \gls{sBDCA} with Preconditioning.}
    \begin{tabular}{p{2cm} | cccccc | cc ccc}
    \toprule
    Player Name & Age & Rating & xG & TP & Pos. & RD & \gls{sBDCA} & \gls{OLS} & mvE & NV & SNV \\
    \midrule
    J. Kimmich & 27.41 & 7.73 & 0.07 & 2,815 & 1 & 5.21 & 57.68 & 47.09 & 75.00 & 75.00 & 60.00 \\
    S. Gnabry & 26.98 & 7.11 & 0.33 & 1,939 & 1 & 2.64 & 42.93 & 31.70 & 55.00 & 45.00 & 40.00 \\
    J. Bellingham & 19.02 & 7.40 & 0.26 & 2,693 & 2 & 7.51 & 133.42 & 78.56 & 120.00 & 150.00 & 180.00 \\
    J. Brandt & 26.18 & 7.22 & 0.16 & 2,410 & 2 & 2.38 & 40.96 & 34.31 & 40.00 & 40.00 & 40.00 \\
    S. Haller & 28.04 & 7.06 & 0.45 & 1,232 & 2 & 3.28 & 28.81 & 18.57 & 30.00 & 25.00 & 18.00 \\
    Y. Moukoko & 17.62 & 6.72 & 0.26 & 1,134 & 2 & 9.95 & 26.99 & 20.43 & 30.00 & 30.00 & 25.00 \\
    C. Nkunku & 24.64 & 7.23 & 0.58 & 1,898 & 3 & 3.75 & 97.57 & 47.79 & 80.00 & 75.00 & 70.00 \\
    T. Werner & 26.34 & 6.88 & 0.34 & 1,954 & 3 & 2.68 & 31.69 & 24.43 & 25.00 & 20.00 & 17.00 \\
    J. Frimpong & 21.57 & 6.94 & 0.22 & 2,717 & 6 & 3.26 & 48.48 & 35.35 & 40.00 & 45.00 & 50.00 \\
    D. Kamada & 25.92 & 6.93 & 0.31 & 2,270 & 7 & 2.65 & 30.83 & 21.25 & 27.00 & 20.00 & 15.00 \\
    J. Lindstrøm & 22.35 & 6.77 & 0.22 & 1,693 & 7 & 2.67 & 22.81 & 16.89 & 28.00 & 25.00 & 22.00 \\
    R. Kolo Muani & 23.59 & 7.08 & 0.41 & 2,646 & 7 & 3.20 & 71.05 & 39.41 & 80.00 & 70.00 & 60.00 \\
    M. Thuram & 24.92 & 7.02 & 0.55 & 2,516 & 10 & 3.59 & 40.64 & 24.36 & 32.00 & 40.00 & 55.00 \\
    M. Berisha & 24.16 & 6.92 & 0.35 & 1,689 & 15 & 2.72 & 14.28 & 8.63 & 12.00 & 14.00 & 12.00 \\
    S. Guirassy & 26.32 & 6.92 & 0.50 & 1,656 & 16 & 3.40 & 13.55 & 7.30 & 14.00 & 40.00 & 40.00  \\
    \bottomrule
    \end{tabular}
    \label{tab:largest_deviations_inliers}
\end{table}


Next, we inspect the bad leverage points identified in Figure~\ref{fig:outlier_map}. 
For this purpose, Table~\ref{tab:bad_leverage_points} presents the 15 players with absolute standardised residuals above 2.5 and the largest robust distances.
Overall, the observations can be grouped by age, with older players at the top and very young players at the bottom.
This then also offers an initial explanation for their classification as leverage points, as the corresponding scatterplot in Figure \ref{fig:histo_mv_before_and_after_season} only shows few instances with similar ages. 
For the young players, it can additionally be observed that their total game time was very limited during the season, adversely affecting their performance metrics.
In the case of Lukas Ambros, the average rating likely does not reflect his true abilities, resulting in an underestimation of his \gls{MaVa}.
Similarly, Arijon Ibrahimovic stands out due to his high expected goals and his team position, which causes his actual \gls{MaVa} to deviate from the general pattern found in the data. 
For the older players, many of the \gls{MaVa} underestimations can also be explained by the limited total game time and low average expected goals. 
In contrast, Anthony Modeste's and Lars Stindl's \glspl{MaVa} are flagged as leverage points due to their relatively high average expected goals and the table positions of their teams, for which the low \glspl{MaVa} are atypical.


\begin{table}[H]
    \scriptsize
    \centering
    \caption{Identified Bad Leverage Points in the Dataset.}
    \begin{tabular}{p{2.4cm} | cccccc | cc ccc}
    \toprule
    Player Name & Age & Rating & xG & TP & Pos. & RD & \gls{sBDCA} & \gls{OLS} & mvE & NV & SNV \\
    \midrule
    M. Hasebe & 38.48 & 6.72 & 0.00 & 1,309 & 7 & 31.41 & 0.16 & 0.36 & 0.80 & 0.50 & 0.50 \\
    M. Neuer & 36.29 & 7.11 & 0.02 & 1,080 & 1 & 21.76 & 0.98 & 1.61 & 7.00 & 5.00 & 4.00 \\
    P. Pekarík & 35.69 & 6.45 & 0.00 & 206 & 18 & 19.57 & 0.05 & 0.08 & 0.50 & 0.30 & 0.30 \\
    C. Trimmel & 35.37 & 7.00 & 0.03 & 1,818 & 4 & 18.30 & 1.51 & 2.34 & 0.80 & 0.80 & 0.60 \\
    K.-P. Boateng & 35.35 & 6.40 & 0.02 & 487 & 18 & 18.30 & 0.09 & 0.15 & 0.70 & 0.00 & 0.00 \\
    R. Gikiewicz & 34.70 & 7.10 & 0.00 & 2,070 & 15 & 16.20 & 0.57 & 0.91 & 0.90 & 0.65 & 0.40 \\
    J. Baumgartlinger & 34.52 & 6.57 & 0.02 & 561 & 15 & 15.27 & 0.18 & 0.29 & 0.40 & 0.00 & 0.00 \\
    D. Caligiuri & 34.48 & 6.72 & 0.04 & 296 & 15 & 15.14 & 0.21 & 0.27 & 0.60 & 0.40 & 0.00 \\
    T. Sippel & 34.30 & 6.68 & 0.00 & 509 & 10 & 14.57 & 0.21 & 0.41 & 0.40 & 0.30 & 0.30 \\
    A. Modeste & 34.24 & 6.73 & 0.23 & 815 & 2 & 14.45 & 2.14 & 2.48 & 0.70 & 0.40 & 0.30 \\
    L. Stindl & 33.87 & 6.97 & 0.17 & 1,743 & 10 & 13.15 & 2.03 & 2.33 & 1.00 & 1.00 & 1.00 \\
    N. Petersen & 33.59 & 6.59 & 0.05 & 287 & 5 & 12.46 & 0.62 & 0.86 & 1.50 & 0.00 & 0.00 \\
    Y. Sommer & 33.56 & 6.90 & 0.00 & 2,610 & 1 & 12.60 & 2.40 & 5.10 & 5.00 & 5.00 & 5.00 \\
    \midrule
    L. Ambros & 18.08 & 4.10 & 0.00 & 16 & 8 & 17.70 & 0.04 & 0.19 & 0.20 & 0.23 & 0.23 \\
    A. Ibrahimovic & 16.56 & 6.90 & 0.52 & 13 & 1 & 13.23 & 19.07 & 8.79 & 1.00 & 4.00 & 4.00 \\
    \bottomrule
    \end{tabular}
    \label{tab:bad_leverage_points}
\end{table}

\vspace{-0.2cm}


Lastly, we examine younger players with at least five full matches of game time, for whom our model produced the highest overestimations.
The corresponding results are shown in Table~\ref{tab:young_and_underrated}. 
In six out of ten cases, the \gls{MaVa} increased during the subsequent two valuations, approaching or even exceeding the values predicted by our model.
These discrepancies may point to market mispricing, and could be leveraged to inform strategies for identifying undervalued talent.
Future research could investigate whether such patterns persist across larger datasets or different leagues to determine whether they truly reflect market inefficiencies or instead arise from other factors.


\begin{table}[H]
    \scriptsize
    \centering
    \caption{\acrfull{MaVa} Overestimations among young Players.}
    \begin{tabular}{p{1.9cm} | cccccc | cc ccc}
    \toprule
    Player Name & Age & Rating & xG & TP & Pos. & RD & \gls{sBDCA} & \gls{OLS} & mvE & NV & SNV \\
    \midrule
    D.-A. Zagadou & 23.09 & 6.84 & 0.10 & 1,272 & 16 & 1.87 & 5.49 & 4.39 & 2.50 & 5.50 & 10.00 \\
    D. Szoboszlai & 21.70 & 7.42 & 0.18 & 2,454 & 3 & 4.16 & 86.08 & 55.41 & 50.00 & 70.00 & 75.00 \\
    D. Malen & 23.46 & 7.09 & 0.35 & 1,725 & 2 & 2.79 & 57.96 & 37.27 & 28.00 & 35.00 & 35.00 \\
    Jr. D. Ebimbe & 21.62 & 6.76 & 0.19 & 1,065 & 7 & 2.88 & 15.75 & 11.44 & 10.00 & 14.00 & 12.00 \\
    L. Maina & 23.04 & 6.70 & 0.14 & 2,259 & 11 & 2.22 & 10.83 & 10.95 & 6.00 & 6.00 & 5.00 \\
    N. Soldo & 21.44 & 6.66 & 0.09 & 592 & 11 & 2.35 & 3.98 & 3.80 & 1.50 & 1.50 & 1.50 \\
    O. Marmoush & 23.41 & 6.63 & 0.19 & 1,475 & 8 & 2.95 & 9.46 & 9.62 & 6.00 & 8.00 & 15.00 \\
    R. Zalazar & 22.90 & 6.85 & 0.14 & 1,327 & 17 & 2.27 & 6.32 & 4.58 & 4.00 & 6.00 & 8.00 \\
    S. Tigges & 23.93 & 6.79 & 0.17 & 1,598 & 11 & 1.78 & 9.29 & 8.34 & 3.50 & 2.50 & 2.00 \\    
    W.-Y. Jeong & 22.79 & 6.52 & 0.08 & 646 & 5 & 2.37 & 6.24 & 6.38 & 3.50 & 3.50 & 3.50 \\
    \bottomrule
    \end{tabular}
\label{tab:young_and_underrated}
\end{table}


\section{Conclusions}\label{sec:conclusions}



In this paper, we proposed a \gls{DC} programming algorithm to solve the classical \gls{LTS} problem  more effectively (from a single starting point) than existing approaches. 
For this purpose, we first recalled how the problem can be reformulated as a concave minimisation defined over a polytope. 
Based on this representation, we then derived a novel \gls{DC} decomposition for the problem.
This, in turn, enabled us to construct strongly convex approximations of the \gls{LTS} problem, which can be efficiently solved via inexpensive projections onto the feasible set.
Afterwards, we proposed the \gls{sBDCA} as a solution method that combines successive \gls{DC} decompositions with a simplified line search procedure.
Specifically for the \gls{LTS} problem, we proved that the algorithm converges to a local solution with a linear rate in the fastest case, and with a sublinear rate in the slowest case.
To obtain this result, we deduced, based on the Cayley-Hamilton Theorem, that the objective function is a multivariate rational function and a real-analytic function [cf. Definition~\ref{defi:real_analytic}], which fulfils the \L ojasiewic property with exponent $\vartheta_f \in [\frac{1}{2}, 1)$.
To ensure robust results based on a single starting point, we further derived a problem-specific preconditioning matrix that utilises the hat values of instances to correct the gradient direction of the objective function.
Ultimately, this preconditioner was incorporated into the projection step of the \gls{sBDCA} framework to better balance the influence of regression outliers and leverage points on the optimisation result.


In the practical part of this paper, we conducted extensive numerical experiments to compare the performance of our proposed \gls{DC} programming algorithm with the Fast-\gls{LTS} heuristic based on several synthetic and real-world datasets.
In these comparisons, we demonstrated that, especially in settings with many independent variables, \gls{sBDCA} with preconditioning is up to 3.25 times faster than the state-of-the-art, and finds up to 90\% lower objective function values. 
We also analysed the \gls{ATD} depending on different initialisation strategies, which showed that our proposed \gls{DC} programming algorithm is more robust than the core procedure of Fast-\gls{LTS}. 
As a result, the \gls{sBDCA} with preconditioning needed only one starting point to match or surpass the solution quality of the heuristic that depends on 500 starting points.
In addition to this performance comparison, we introduced a novel application of robust regression for outlier detection in sports analytics.
For this purpose, we collected a new dataset and constructed a prediction model for the \gls{MaVa} of football players, which was subsequently used to perform an outlier analysis. 
Even though a lot of noise was present in the data, such as players who only played a single match in the season, we could precisely detect the underlying data trends.
Moreover, we identified a set of athletes who may be misvalued, which may inform strategies for acquiring undervalued talent.


Ultimately, several directions for future research emerge from this work.
From a theoretical perspective, alternative preconditioning matrices could be derived and analysed to further enhance the algorithm’s effectiveness.
Moreover, other \gls{DC} decompositions could be constructed and investigated for the \gls{LTS} problem.
From a practical standpoint, future work could examine how to automatically tune the input parameters of \gls{sBDCA} to improve user-friendliness. 
In this context, exploring alternative initialisation strategies that remain reliable in high-dimensional settings would also be valuable.
Furthermore, future numerical comparisons could benefit from an enriched collection of synthetic and real-world regression datasets with outliers to enable a more systematic benchmarking of robust regression estimators.
To enhance the results on empirical football player valuations, the provided Python code could be leveraged to collect datasets from other seasons and leagues to examine whether similar conclusions hold. 
Similarly, it may be worth developing standardised criteria for selecting the player base more carefully to prevent misleading conclusions about the significance of input variables based on p-values.


\section*{Acknowledgments}

The authors acknowledge the use of the IRIDIS High Performance Computing Facility, and associated support services at the University of Southampton, in the completion of this work.

Further, the authors would like to thank Selin Damla Ahipa{\c{s}}ao{\u{g}}lu, Stefano Cipolla, Tim Marshall-Cox, and Samuel Ward for helpful remarks and insightful discussions throughout the development of this work.


\addcontentsline{toc}{section}{References}

\setcitestyle{numbers}
\renewcommand{\bibnumfmt}[1]{#1.}

\bibliographystyle{dcu}
\setlength{\bibsep}{5.5pt}
\bibliography{02_Chapters/08_References.bib}

@book{aggarwal2015,
      title={{Data Mining: The Textbook}},
      author={Aggarwal, Charu C},
      edition={1st},
      year={2015},
      publisher={Springer},
      address={Cham},
      note={\url{https://doi.org/10.1007/978-3-319-14142-8}}
}

@article{AGULLO2001,
         title = {New algorithms for computing the least trimmed squares regression estimator},
         journal = {Computational Statistics \& Data Analysis},
         volume = {36},
         number = {4},
         pages = {425--439},
         year = {2001},
         note = {\url{https://doi.org/10.1016/S0167-9473(00)00056-6}},
         author = {José Agulló}
}

@article{AlNoor2013,
         title={{Model of Robust Regression with Parametric and Nonparametric Methods}},
         author={Al-Noor, Nadia H. and Mohammad, Asmaa A},
         journal={Mathematical Theory and Modeling},
         volume={3},
         number={5},
         pages={27--39},
         year={2013}
}

@article{Alfons2013,
         note = {\url{https://doi.org/10.1214/12-AOAS575}},
         author = {Andreas Alfons and Christophe Croux and Sarah Gelper},
         journal = {The Annals of Applied Statistics},
         number = {1},
         pages = {226--248},
         publisher = {Institute of Mathematical Statistics},
         title = {SPARSE LEAST TRIMMED SQUARES REGRESSION FOR ANALYZING HIGH-DIMENSIONAL LARGE DATA SETS},
         volume = {7},
         year = {2013}
}

@article{alma2011,
         title={Comparison of robust regression methods in linear regression},
         author={Alma, {\"O}zlem G{\"u}r{\"u}nl{\"u}},
         journal={International Journal of Contemporary Mathematical Sciences},
         volume={6},
         number={9},
         pages={409--421},
         year={2011}
}

@misc{ang2021,
      title={{Fast Projection onto the Capped Simplex with Applications to Sparse Regression in Bioinformatics, arXiv preprint, arXiv:2110.08471 [math.OC]}}, 
      author={Ang, Andersen and Ma, Jianzhu and Liu, Nianjun and Huang, Kun and Wang, Yijie},
      year={2021},
      note={\url{https://doi.org/10.48550/arXiv.2110.08471}}, 
}

@book{anjos2011,
      title={{Handbook on Semidefinite, Conic and Polynomial Optimization}},
      author={Anjos, Miguel F and Lasserre, Jean B},
      edition={1st},
      series={International Series in Operations Research \& Management Science},
      year={2011},
      publisher={Springer},
      address={New York},
      note={\url{https://doi.org/10.1007/978-1-4614-0769-0}}
}

@book{Antoniou2021,
      title={{Practical Optimization: Algorithms and Engineering Applications}},
      author={Antoniou, Andreas and Lu, Wu-Sheng},
      edition={2nd},
      year={2021},
      series={Texts in Computer Science},
      publisher={Springer},
      address={New York},
      note={\url{https://doi.org/10.1007/978-1-0716-0843-2}}
}

@article{aragon2018,
         title={{Accelerating the DC algorithm for smooth functions}},
         author={Arag{\'o}n-Artacho, Francisco J and Fleming, Ronan MT and Vuong, Phan T},
         journal={Mathematical Programming},
         volume={169},
         pages={95--118},
         year={2018},
         publisher={Springer},
         note = {\url{https://doi.org/10.1007/s10107-017-1180-1}}
}

@article{aragon2020,
         author = {Arag\'{o}n-Artacho, Francisco J. and Vuong, Phan T.},
         title = {{The Boosted Difference of Convex Functions Algorithm For Nonsmooth Functions}},
         journal = {SIAM Journal on Optimization},
         volume = {30},
         number = {1},
         pages = {980--1006},
         year = {2020},
         note = { \url{https://doi.org/10.1137/18M123339X}}
}

@article{aragon2022,
         title={{The Boosted DC Algorithm for Linearly Constrained DC Programming}},
         author={Arag{\'o}n-Artacho, Francisco Javier and Campoy, Rub{\'e}n and Vuong, Phan T},
         journal={Set-Valued and Variational Analysis},
         pages={1265--1289},
         volume={30},
         year={2022},
         note={\url{https://doi.org/10.1007/s11228-022-00656-x}}
}

@article{Armstrong2014,
         author = {Armstrong, Richard A.},
         title = {{When to use the Bonferroni correction}},
         journal = {Ophthalmic and Physiological Optics},
         volume = {34},
         number = {5},
         pages = {502--508},
         note = {\url{https://doi.org/10.1111/opo.12131}},
         year = {2014}
}

@article{atkinson1999,
         title={Computing least trimmed squares regression with the forward search},
         author={Atkinson, Anthony C and Cheng, T.-C.},
         journal={Statistics and Computing},
         volume={9},
         pages={251--263},
         year={1999},
         note={\url{https://doi.org/10.1023/A:1008942604045}}
}

@article{Baesens2009,
        author = {Baesens, B and Mues, C and Martens, D and Vanthienen, J},
        title = {{50 years of data mining and OR: upcoming trends and challenges}},
        journal = {Journal of the Operational Research Society},
        volume = {60},
        number = {sup1},
        pages = {S16--S23},
        year = {2009},
        note = {\url{https://doi.org/10.1057/jors.2008.171}},
}

@article{Ball2025,
         author = {James Ball and Minh Huynh and Matthew C. Varley},
         title = {Comparing player rating systems as a metric for assessing individual performance in soccer},
         journal = {Journal of Sports Sciences},
         volume = {43},
         number = {7},
         pages = {676--686},
         year = {2025},
         note = {\url{https://doi.org/10.1080/02640414.2025.2471208}},
}

@article{BARBATO2024,
         title = {Mathematical programming for simultaneous feature selection and outlier detection under l1 norm},
         journal = {European Journal of Operational Research},
         volume = {316},
         number = {3},
         pages = {1070--1084},
         year = {2024},
         note= {\url{https://doi.org/10.1016/j.ejor.2024.03.035}},
         author = {Barbato, Michele and Ceselli, Alberto},
}

@article{Becher1993,
         note = {\url{https://doi.org/10.2307/2532271}},
         author = {Heiko Becher and Peter Hall and Susan R. Wilson},
         journal = {Biometrics},
         number = {4},
         pages = {1268--1272},
         title = {Bootstrap Hypothesis Testing Procedures},
         volume = {49},
         year = {1993}
}

@article{beiranvand2017,
         title={Best practices for comparing optimization algorithms},
         author={Beiranvand, Vahid and Hare, Warren and Lucet, Yves},
         journal={Optimization and Engineering},
         volume={18},
         pages={815--848},
         year={2017},
         note={\url{https://doi.org/10.1007/s11081-017-9366-1}}
}

@article{beliakov2012,
         title={Computing of high breakdown regression estimators without sorting on graphics processing units},
         author={Beliakov, Gleb and Johnstone, Michael and Nahavandi, Saeid},
         journal={Computing},
         volume={94},
         pages={433--447},
         year={2012},
         note={\url{https://doi.org/10.1007/s00607-011-0183-7}}
}

@misc{bernholt2006,
      title={{Robust Estimators are Hard to Compute, Technical Report, No. 2005/52,  Universität Dortmund, Sonderforschungsbereich 475 - Komplexitätsreduktion in Multivariaten Datenstrukturen, Dortmund}},
      author={Bernholt, Thorsten},
      year={2006},
      note={\url{https://www.econstor.eu/bitstream/10419/22645/1/tr52-05.pdf}}
}

@misc{bertin2011,
      title={{The Million Song Dataset}},
      author={Bertin-Mahieux, Thierry and Ellis, Daniel PW and Whitman, Brian and Lamere, Paul},
      year={2011},
      note={UCI Machine Learning Repository. Available at \url{https://archive.ics.uci.edu/dataset/203/yearpredictionmsd}}
}

@book{bertsekas1999,
      title={{Nonlinear Programming}},
      author={Bertsekas, Dimitri P},
      edition={2nd},
      year={1999},
      publisher={Athena Scientific},
      address={Belmont (Massachusetts)},
      note={ISBN 1-886529-00-0}
}

@article{Bertsimas2016,
         author = {Bertsimas, Dimitris and King, Angela},
         title = {{OR Forum—An Algorithmic Approach to Linear Regression}},
         journal = {Operations Research},
         volume = {64},
         number = {1},
         pages = {2--16},
         year = {2016},
         note = {\url{https://doi.org/10.1287/opre.2015.1436}}
}

@article{bonferroni1936,
         title={Teoria statistica delle classi e calcolo delle probabilita},
         author={Bonferroni, Carlo},
         journal={Pubblicazioni del R istituto superiore di scienze economiche e commericiali di firenze},
         volume={8},
         pages={3--62},
         year={1936}
}

@book{boyd2009,
      title={{Convex Optimization}},
      author={Boyd, Stephen and Vandenberghe, Lieven},
      year={2009},
      publisher={Cambridge University Press},
      address={Cambridge},
      note={\url{https://doi.org/10.1017/CBO9780511804441}}
}

@book{Brualdi1991, 
      address={Cambridge}, 
      series={Encyclopedia of Mathematics and its Applications}, 
      title={Combinatorial Matrix Theory}, 
      publisher={Cambridge University Press}, 
      author={Brualdi, Richard A. and Ryser, Herbert J.}, 
      year={1991}, 
      edition={1st},
      collection={Encyclopedia of Mathematics and its Applications},
      note={\url{https://doi.org/10.1017/CBO9781107325708}}
}

@ARTICLE{Celik1992,
         author={{\c{C}}elik, M.K. and Abur, A.},
         journal={IEEE Transactions on Power Systems}, 
         title={{A robust WLAV state estimator using transformations}}, 
         year={1992},
         volume={7},
         number={1},
         pages={106--113},
         note={\url{https://doi.org/10.1109/59.141693}}
}

@book{Chatterjee2020,
      title={{Handbook of Regression Analysis With Applications in R}},
      author={Chatterjee, Samprit and Simonoff, Jeffrey S.},
      year={2020},
      edition={2nd},
      publisher={John Wiley \& Sons},
      address={Hoboken (NJ)},
      series={Wiley Series in Probability and Statistics},
      note = {\url{https://doi.org/10.1002/9781119392491}}
}

@article{Chave2003,
         author = {Chave, Alan D. and Thomson, David J.},
         title = {{A Bounded Influence Regression Estimator Based on the Statistics of the Hat Matrix}},
         journal = {Journal of the Royal Statistical Society Series C: Applied Statistics},
         volume = {52},
         number = {3},
         pages = {307--322},
         year = {2003},
         doi = {10.1111/1467-9876.00406},
         note = {\url{https://doi.org/10.1111/1467-9876.00406}}
}

@article{Chen2018,
         author  = {Ruidi Chen and Ioannis Ch. Paschalidis},
         title   = {{A Robust Learning Approach for Regression Models Based on Distributionally Robust Optimization}},
         journal = {Journal of Machine Learning Research},
         year    = {2018},
         volume  = {19},
         number  = {13},
         pages   = {1--48},
         note     = {\url{http://jmlr.org/papers/v19/17-295.html}}
}

@article{cipolla2024,
         title={Proximal-stabilized semidefinite programming},
         author={Cipolla, Stefano and Gondzio, Jacek},
         journal={Computational Optimization and Applications,},
         pages={1--44},
         year={2024},
         volume={},
         note={\url{https://doi.org/10.1007/s10589-024-00614-3}}
}

@article{Critchley2010,
         title={A relaxed approach to combinatorial problems in robustness and diagnostics},
         author={Critchley, F. and Schyns, M. and Haesbroeck, G. and Fauconnier, C. and Lu, G. and Atkinson, R. A. and Wang, D. Q.},
         journal={Statistics and Computing},
         volume={20},
         pages={99--115},
         year={2010},
         note={\url{https://doi.org/10.1007/s11222-009-9119-x}}
}

@article{DEMAESSCHALCK20001,
         title = {{The Mahalanobis distance}},
         journal = {Chemometrics and Intelligent Laboratory Systems},
         volume = {50},
         number = {1},
         pages = {1--18},
         year = {2000},
         note = {\url{https://doi.org/10.1016/S0169-7439(99)00047-7}},
         author = {R. {De Maesschalck} and D. Jouan-Rimbaud and D. L. Massart}
}

@article{deOliveira2020,
         title={{The ABC of DC programming}},
         author={de Oliveira, Welington},
         journal={Set-Valued and Variational Analysis},
         volume={28},
         pages={679--706},
         year={2020},
         note={\url{https://doi.org/10.1007/s11228-020-00566-w}}
}

@article{Dogru2018,
         author = {Doğru, Fatma Zehra and Arslan, Olcay},
         title = {{Robust mixture regression modeling using the least trimmed squares (LTS)-estimation method}},
         journal = {Communications in Statistics - Simulation and Computation},
         volume = {47},
         number = {7},
         pages = {2184--2196},
         year = {2018},
         note = {\url{https://doi.org/10.1080/03610918.2017.1341528}}
}

@article{efron1979,
         title={{Bootstrap Methods: Another look at the Jackknife}},
         author={Efron, Bradley},
         journal={The Annals of Statistics},
         pages={1--26},
         year={1979},
         volume={7},
         number={1},
         note={\url{https://www.jstor.org/stable/2958830}}
}

@book{Fahrmeir2021,
      title={{Regression: Models, Methods and Applications}},
      author={Fahrmeir, Ludwig and Kneib, Thomas and Lang, Stefan and Marx, Brian D.},
      edition={2nd},
      year={2021},
      publisher={Springer},
      address={Berlin},
      note={\url{https://doi.org/10.1007/978-3-662-63882-8}}
}

@article{fernandes2023,
         author = {Fernandes, Alvaro A. A. and Koehler, Martin and Konstantinou, Nikolaos and Pankin, Pavel and Paton, Norman P.},
         title = {Data Preparation: A Technological Perspective and Review},
         journal = {SN Computer Science},
         volume = {4},
         number = {425},
         pages = {1--20},
         year  = {2023},
         note = {\url{https://doi.org/10.1007/s42979-023-01828-8}}
}

@article{FLORES2015,
         title = {{SOCP relaxation bounds for the optimal subset selection problem applied to robust linear regression}},
         journal = {European Journal of Operational Research},
         volume = {246},
         number = {1},
         pages = {44--50},
         year = {2015},
         note = {\url{https://doi.org/10.1016/j.ejor.2015.04.024}},
         author = {Flores, Salvador},
}

@book{fox2015,
      title={{Applied Regression Analysis and Generalized Linear Models}},
      author={Fox, John},
      year={2015},
      edition= {3rd},
      publisher={SAGE Publications},
      address={Thousand Oaks (California)}
}

@book{fox2019,
      title={{An R companion to applied regression}},
      author={Fox, John and Weisberg, Sanford},
      year={2019},
      edition= {3rd},
      publisher={SAGE Publications},
      address={Thousand Oaks (California)}
}

@article{Franceschi2024,
         author = {Franceschi, Maxence and Brocard, Jean-François and Follert, Florian and Gouguet, Jean-Jacques},
         title = {{Determinants of football players’ valuation: A systematic review}},
         journal = {Journal of Economic Surveys},
         volume = {38},
         number = {3},
         pages = {577--600},
         note = {\url{https://doi.org/10.1111/joes.12552}},
         year = {2024}
}

@article{Gafni1984,
         author = {Gafni, Eli M. and Bertsekas, Dimitri P.},
         title = {{Two-Metric Projection Methods for Constrained Optimization}},
         journal = {SIAM Journal on Control and Optimization},
         volume = {22},
         number = {6},
         pages = {936--964},
         year = {1984},
         note = {\url{https://doi.org/10.1137/0322061}}
}

@book{garcia2015,
      title={{Data Preprocessing in Data Mining}},
      author={Garc{\'\i}a, Salvador and Luengo, Juli{\'a}n and Herrera, Francisco},
      edition={1st},
      year={2015},
      publisher={Springer},
      address={Cham},
      series={Intelligent Systems Reference Library},
      note={\url{https://doi.org/10.1007/978-3-319-10247-4}}
}

@article{GILONI2002,
         title = {{Least Trimmed Squares Regression, Least Median Squares Regression, and Mathematical Programming}},
         journal = {Mathematical and Computer Modelling},
         volume = {35},
         number = {9--10},
         pages = {1043--1060},
         year = {2002},
         note = {\url{https://doi.org/10.1016/S0895-7177(02)00069-9}},
         author = {A. Giloni and M. Padberg}
}

@book{Giorgi2023, 
      address={Cham}, 
      series={International Series in Operations Research \& Management Science}, 
      title={Basic Mathematical Programming Theory}, 
      publisher={Springer}, 
      author={Giorgi, Giorgio and Jimen{\'e}z, Bienvenido and Novo, Vicente }, 
      year={2023}, 
      edition ={1st},
      note={\url{https://doi.org/10.1007/978-3-031-30324-1}}
}

@article{Habshah2009,
         author = {Habshah, M. and  Norazan,  M. R. and  Rahmatullah Imon, A. H.M.},
         title = {The performance of diagnostic-robust generalized potentials for the identification of multiple high leverage points in linear regression},
         journal = {Journal of Applied Statistics},
         volume = {36},
         number = {5},
         pages = {507--520},
         year = {2009},
         note = {\url{https://doi.org/10.1080/02664760802553463}}
}

@article{Hadi1993,
         author = {Ali S. Hadi and Jeffrey S. Simonoff},
         title = {{Procedures for the Identification of Multiple Outliers in Linear Models}},
         journal = {Journal of the American Statistical Association},
         volume = {88},
         number = {424},
         pages = {1264--1272},
         year = {1993},
         note = {\url{https://doi.org/10.1080/01621459.1993.10476407}}
}

@article{HAMIDIEH2018,
         title = {A data-driven statistical model for predicting the critical temperature of a superconductor},
         journal = {Computational Materials Science},
         volume = {154},
         pages = {346--354},
         year = {2018},
         note = {\url{https://doi.org/10.1016/j.commatsci.2018.07.052}},
         author = {Hamidieh, Kam},
}

@article{hampel1973,
         title={{Robust estimation: A condensed partial survey}},
         author={Hampel, Frank R},
         year={1973},
         journal={Zeitschrift für Wahrscheinlichkeitstheorie und Verwandte Gebiete},
         volume={27},
         pages={87--104},
         note = {\url{https://doi.org/10.1007/BF00536619}}
}

@book{hampel1986,
      title={{Robust Statistics: The Approach Based on Influence Functions}},
      author={Hampel, Frank R and Ronchetti, Elvezio M and Rousseeuw, Peter J and Stahel, Werner A},
      year={1986},
      edition={1st},
      series={Wiley Series in Probability and Statistics},
      publisher={John Wiley \& Sons},
      address={New York},
      note = {\url{https://doi.org/10.1002/9781118186435}}
}

@article{HARRINGTON2010,
         title = {{Finding approximate solutions to combinatorial problems with very large data sets using BIRCH}},
         journal = {Computational Statistics \& Data Analysis},
         volume = {54},
         number = {3},
         pages = {655--667},
         year = {2010},
         note = {\url{https://doi.org/10.1016/j.csda.2008.08.001}},
         author = { Harrington, Justin and Salibián-Barrera, Matias},
}

@article{hartman1959,
         title={On functions representable as a difference of convex functions},
         author={Hartman, Philip},
         year={1959},
         journal={Pacific Journal of Mathematics},
         volume={9},
         number={3},
         pages={707--713},
         note = {\url{https://doi.org/10.2140/pjm.1959.9.707}}
}

@book{hastie2009,
      title={{The Elements of Statistical Learning: Data Mining, Inference, and Prediction}},
      author={Hastie, Trevor and Tibshirani, Robert and Friedman, Jerome H},
      edition={2nd},
      year={2009},
      note={\url{https://doi.org/10.1007/b94608}},
      address ={New York},
      publisher={Springer}
}

@book{hawkins1980,
      title={{Identification of Outliers}},
      author={Hawkins, Douglas M},
      edition={1st},
      year={1980},
      series={Monographs on Statistics and Applied Probability},
      publisher={Springer},
      note={\url{https://doi.org/10.1007/978-94-015-3994-4}},
      address={Dordrecht}
}

@article{hawkins1984,
         title={{Location of Several Outliers in Multiple-Regression Data Using Elemental Sets}},
         author={Hawkins, Douglas M and Bradu, Dan and Kass, Gordon V},
         journal={Technometrics},
         volume={26},
         number={3},
         pages={197--208},
         year={1984},
         note={\url{https://doi.org/10.1080/00401706.1984.10487956}}
}

@article{HAWKINS1994,
         title = {The feasible solution algorithm for least trimmed squares regression},
         journal = {Computational Statistics \& Data Analysis},
         volume = {17},
         number = {2},
         pages = {185--196},
         year = {1994},
         note = {\url{https://doi.org/10.1016/0167-9473(92)00070-8}},
         author = {Douglas M. Hawkins}
}

@article{HAWKINS1999,
         title = {Improved feasible solution algorithms for high breakdown estimation},
         journal = {Computational Statistics \& Data Analysis},
         volume = {30},
         number = {1},
         pages = {1--11},
         year = {1999},
         note = {\url{https://doi.org/10.1016/S0167-9473(98)00082-6}},
         author = {Douglas M. Hawkins and David J. Olive}
}

@article{Hawkins2002,
         author = {Hawkins, Douglas M and Olive, David J },
         title = {{Inconsistency of Resampling Algorithms for High-Breakdown Regression Estimators and a New Algorithm}},
         journal = {Journal of the American Statistical Association},
         volume = {97},
         number = {457},
         pages = {136--159},
         year = {2002},
         note = {\url{https://doi.org/10.1198/016214502753479293}}
}

@article{Heng2025,
         title={Bootstrap estimation of the proportion of outliers in robust regression},
         author={Heng, Q. and Lange, K.},
         journal={Statistics and Computing},
         volume={35},
         number ={3},
         pages={1--14},
         year={2025},
         note={\url{https://doi.org/10.1007/s11222-024-10526-1}}
}

@InProceedings{Ho2020,
               author="Ho, Vinh Thanh and Le Thi, Hoai An and Pham Dinh, Tao",
               editor="Le Thi, Hoai An and Le, Hoai Minh and Pham Dinh, Tao and Nguyen, Ngoc Thanh",
               title={{DCA with Successive DC Decomposition for Convex Piecewise-Linear Fitting}},
               booktitle="Advanced Computational Methods for Knowledge Engineering",
               year="2020",
               publisher="Springer International Publishing",
               address="Cham",
               pages="39--51",
               note={\url{https://doi.org/10.1007/978-3-030-38364-0_4}}
}

@article{HO2021,
         title = {{DCA-based algorithms for DC fitting}},
         journal = {Journal of Computational and Applied Mathematics},
         volume = {389},
         pages = {113353},
         year = {2021},
         note = {\url{https://doi.org/10.1016/j.cam.2020.113353}},
         author = {Vinh Thanh Ho and Hoai An {Le Thi} and Tao {Pham Dinh}},
}

@article{Hoaglin1978,
         author = {Hoaglin, David C. and Welsch, Roy E.},
         title = {{The Hat Matrix in Regression and ANOVA}},
         journal = {The American Statistician},
         volume = {32},
         number = {1},
         pages = {17--22},
         year = {1978},
         note = {\url{https://doi.org/10.1080/00031305.1978.10479237}}
}

@book{Hocking2003,
      title={{Methods and Applications of Linear Models: Regression and the Analysis of Variance}},
      author={Hocking, Ronald R},
      edition={2nd},
      year={2003},
      publisher={John Wiley \& Sons},
      address={Hoboken (New Jersey)},
      note={\url{https://doi.org/10.1002/0471434159}}
}

@article{Hofmann2010,
         author = {Hofmann, Marc and  Gatu, Cristian and  Kontoghiorghes, Erricos John},
         title = {{An Exact Least Trimmed Squares Algorithm for a Range of Coverage Values}},
         journal = {Journal of Computational and Graphical Statistics},
         volume = {19},
         number = {1},
         pages = {191--204},
         year = {2010},
         note = {\url{https://doi.org/10.1198/jcgs.2009.07091}}

}

@book{Horst1995,
      title={{Handbook of Global Optimization: Volume 2}},
      author={Horst, Reiner and Pardalos, Panos M.},
      year={1995},
      edition={1st},
      series={Nonconvex Optimization and Its Applications},
      publisher={Springer},
      address={New York},
      note = {\url{https://doi.org/10.1007/978-1-4757-5362-2}}
}

@article{HOSSJER1995,
         title = {Exact computation of the least trimmed squares estimate in simple linear regression},
         journal = {Computational Statistics \& Data Analysis},
         volume = {19},
         number = {3},
         pages = {265--282},
         year = {1995},
         note = {\url{https://doi.org/10.1016/0167-9473(95)92697-V}},
         author = {Ola H{\"o}ssjer}
}

@ARTICLE{Huang2016,
         author={Huang, Dong and Cabral, Ricardo and Torre, Fernando De la},
         journal={IEEE Transactions on Pattern Analysis and Machine Intelligence}, 
         title={Robust Regression}, 
         year={2016},
         volume={38},
         number={2},
         pages={363--375},
         note={\url{https://doi.org/10.1109/TPAMI.2015.2448091}}
}

@book{huber2009,
      title={{Robust Statistics}},
      author={Huber, Peter J. and  Ronchetti, Elvezio M.},
      series={Wiley Series in Probability and Statistics},
      edition={2nd},
      year={2009},
      publisher={John Wiley \& Sons},
      address = {Hoboken},
      note={\url{https://doi.org/10.1002/9780470434697}}
}

@book{ibe2013,
      title={{Markov Processes for Stochastic Modeling}},
      author={Ibe, Oliver C.},
      year={2013},
      publisher={Elsevier},
      edition={2nd},
      address={London},
      note = {\url{https://doi.org/10.1016/C2012-0-06106-6}}
}

@book{James2021,
      title={{An Introduction to Statistical Learning with Applications in R}},
      author={James, Gareth and Witten, Daniela and Hastie, Trevor and Tibshirani, Robert},
      year={2021},
      edition={2nd},
      publisher={Springer},
      address={New York},
      series={Springer Texts in Statistics},
      note = {\url{https://doi.org/10.1007/978-1-0716-1418-1}}
}

@article{Kan2013,
         author = {Kan, Bet{\"u}l and Alpu, {\"O}zlem and  Yazıcı, Berna},
         title = {{Robust ridge and robust Liu estimator for regression based on the LTS estimator}},
         journal = {Journal of Applied Statistics},
         volume = {40},
         number = {3},
         pages = {644--655},
         year = {2013},
         note = {\url{https://doi.org/10.1080/02664763.2012.750285}}
}

@article{Kirschstein2019,
         author = {T. Kirschstein and Steffen Liebscher},
         title = {Assessing the market values of soccer players – a robust analysis of data from German 1. and 2. Bundesliga},
         journal = {Journal of Applied Statistics},
         volume = {46},
         number = {7},
         pages = {1336--1349},
         year = {2019},
         note = {\url{https://doi.org/10.1080/02664763.2018.1540689}}
}

@book{Lewis2023, 
      series={Lecture Notes in Mathematics}, 
      title={Geometric Analysis on Real Analytic Manifolds}, 
      publisher={Springer}, 
      address={Cham},
      edition={1st},
      author={Lewis, Andrew D.}, 
      year={2023}, 
      note={\url{https://doi.org/10.1007/978-3-031-37913-0}}
}

@article{LeThi2000,
         title={An efficient algorithm for globally minimizing a quadratic function under convex quadratic constraints},
         author={Le Thi, Hoai An},
         journal={Mathematical Programming},
         volume={87},
         pages={401--426},
         year={2000},
         note={\url{https://doi.org/10.1007/s101070050003}}
}

@article{LeThi2005,
         title={{The DC (Difference of Convex Functions) Programming and DCA Revisited with DC Models of Real World Nonconvex Optimization Problems}},
         author={Le Thi, Hoai An and Pham Dinh, Tao},
         journal={Annals of Operations Research},
         volume={133},
         pages={23--46},
         year={2005},
         note={\url{https://doi.org/10.1007/s10479-004-5022-1}}
}

@article{leThi2018,
         title={{DC programming and DCA: thirty years of developments}},
         author={Le Thi, Hoai An and Pham Dinh, Tao},
         journal={Mathematical Programming},
         volume={169},
         pages={5--68},
         year={2018},
         note={\url{https://doi.org/10.1007/s10107-018-1235-y}}
}

@article{leThi2019,
         title={{A unified DC programming framework and efficient DCA based approaches for large scale batch reinforcement learning}},
         author={Le Thi, Hoai An and Ho, Vinh Thanh and Pham Dinh, Tao},
         journal={Journal of Global Optimization},
         volume={73},
         pages={279--310},
         year={2019},
         note={\url{https://doi.org/10.1007/s10898-018-0698-y}}
}

@article{leThi2024,
         title={{Open issues and recent advances in DC programming and DCA}},
         author={Le Thi, Hoai An and Pham Dinh, Tao},
         journal={Journal of Global Optimization},
         volume={88},
         pages={533--590},
         year={2024},
         note={\url{https://doi.org/10.1007/s10898-023-01272-1}}
}

@article{Liu2019,
         title = {{A refined convergence analysis of ${pDCA}_{e}$ with applications to simultaneous sparse recovery and outlier detection}},
         journal = {Computational Optimization and Applications},
         volume = {73},
         pages = {69--100},
         year = {2019},
         note = {\url{https://doi.org/10.1007/s10589-019-00067-z}},
         author = {Liu, Tianxiang and Pong, Ting Kei and Takeda, Akiko}
}

@book{Lojasiewicz1965,
      title={Ensembles Semi-Analytiques}, 
      author={{\L}ojasiewicz, Stanis{\l}aw},
      year={1965},
      publisher ={Institute des Hautes Etudes Scientifiques},
      address={Bures-sur-Yvette (Seine-et-Oise), France}, 
}

@article{Mahalanobis2018,
         author = {P. C. Mahalanobis},
         journal = {Sankhyā: The Indian Journal of Statistics, Series A},
         number = {A},
         pages = {pp. S1--S7},
         title = {ON THE GENERALIZED DISTANCE IN STATISTICS},
         volume = {80},
         year = {2018},
         note={\url{https://doi.org/10.1007/s13171-019-00164-5}}
}

@article{MANGASARIAN1967,
         title = {{The Fritz John necessary optimality conditions in the presence of equality and inequality constraints}},
         journal = {Journal of Mathematical Analysis and Applications},
         volume = {17},
         number = {1},
         pages = {37--47},
         year = {1967},
         note = {\url{https://doi.org/10.1016/0022-247X(67)90163-1}},
         author = {Mangasarian, O. L. and Fromovitz, S.}
}

@book{Maronna2019,
      title={{Robust Statistics: Theory and Methods (with R)}},
      author={Maronna, Ricardo A. and Martin, R. Douglas and Yohai, Victor J. and Salibián-Barrera, Matías },
      year={2019},
      edition={2nd},
      publisher={John Wiley \& Sons},
      address={Hoboken},
      series={Wiley Series in Probability and Statistics},
      note = {\url{https://doi.org/10.1002/9781119214656}}
}

@article{Molybog2020,
         author  = {Igor Molybog and Ramtin Madani and Javad Lavaei},
         title   = {{Conic Optimization for Quadratic Regression Under Sparse Noise}},
         journal = {Journal of Machine Learning Research},
         year    = {2020},
         volume  = {21},
         number  = {195},
         pages   = {1--36},
         note     = {\url{http://jmlr.org/papers/v21/18-881.html}}
}

@book{morris2011,
      title={{Design of Experiments: An Introduction Based on Linear Models}},
      author={Morris, Max D.},
      year={2011},
      edition={1st},
      publisher={Chapman \& Hall/CRC},
      address={New York},
      note = {\url{https://doi.org/10.1201/9781439894903}}
}

@article{mount2014,
         title={{On the Least Trimmed Squares Estimator}},
         author={Mount, David M and Netanyahu, Nathan S and Piatko, Christine D and Silverman, Ruth and Wu, Angela Y},
         journal={Algorithmica},
         volume={69},
         pages={148--183},
         year={2014},
         note={\url{https://doi.org/10.1007/s00453-012-9721-8}}
}

@article{mount2016,
         title = {{A practical approximation algorithm for the LTS estimator}},
         journal = {Computational Statistics \& Data Analysis},
         volume = {99},
         pages = {148--170},
         year = {2016},
         note = {\url{https://doi.org/10.1016/j.csda.2016.01.016}},
         author = {Mount, David M. and Netanyahu, Nathan S. and Piatko, Christine D. and  Wu, Angela Y. and  Silverman, Ruth}
}

@article{MULLER2017,
         title = {{Beyond crowd judgments: Data-driven estimation of market value in association football}},
         journal = {European Journal of Operational Research},
         volume = {263},
         number = {2},
         pages = {611--624},
         year = {2017},
         note = {\url{https://doi.org/10.1016/j.ejor.2017.05.005}},
         author = {Müller, Oliver and Simons, Alexander and  Weinmann, Markus},
}

@book{ng2011,
      title={{Dirichlet and Related Distributions: Theory, Methods and Applications}},
      author={Ng, Kai Wang and Tian, Guo-Liang and Tang, Man-Lai},
      year={2011},
      publisher={John Wiley \& Sons},
      edition={1st},
      address={Chichester},
      series={Wiley Series in Probability and Statistics},
      note = {\url{https://doi.org/10.1002/9781119995784}}
}

@article{NGUYEN2010,
         title = {Outlier detection and least trimmed squares approximation using semi-definite programming},
         journal = {Computational Statistics \& Data Analysis},
         volume = {54},
         number = {12},
         pages = {3212--3226},
         year = {2010},
         note = {\url{https://doi.org/10.1016/j.csda.2009.09.037}},
         author = {T.D. Nguyen and R. Welsch}
}

@book{Nocedal2006, 
      address={New York}, 
      series={Springer Series in Operations Research and Financial Engineering}, 
      title={Numerical Optimization}, 
      publisher={Springer}, 
      author={Nocedal, Jorge and Wright, Stephen J.}, 
      year={2006}, 
      edition={2nd},
      note={\url{https://doi.org/10.1007/978-0-387-40065-5}}
}

@article{Nurunnabi2014,
         author = {Nurunnabi, A.A.M. and  Hadi, Ali S. and  Imon, A.H.M.R.},
         title = {Procedures for the identification of multiple influential observations in linear regression},
         journal = {Journal of Applied Statistics},
         volume = {41},
         number = {6},
         pages = {1315--1331},
         year = {2014},
         note = { \url{https://doi.org/10.1080/02664763.2013.868418}}
}

@Incollection{owen2007,
              title={A robust hybrid of lasso and ridge regression},
              author={Owen, Art B},
              editor={Verducci, J. S. and Shen, X. and Lafferty, J.},
              booktitle={AMS-IMS-SIAM Joint Summer Research Conference Machine and Statistical Learning: Prediction and Discovery },
              series={Contemporary Mathematics},
              publisher={American Mathematical Society},
              edition={1st},
              volume={443},
              pages={59--72},
              address={Providence, Rhode Island},
              year={2007},
              note={\url{http://dx.doi.org/10.1090/conm/443/08555}},
}

@article{pena1995,
         title={{The Detection of Influential Subsets in Linear Regression by using an Influence Matrix}},
         author={Pe{\~n}a, Daniel and Yohai, Victor J},
         journal={Journal of the Royal Statistical Society: Series B (Methodological)},
         volume={57},
         number={1},
         pages={145--156},
         year={1995},
         note={\url{https://doi.org/10.1111/j.2517-6161.1995.tb02020.x}}
}

@article{PhamDinh1997,
         title={{Convex analysis approach to DC programming: Theory, Algorithms and Applications}},
         author={Pham Dinh, Tao and Le Thi, Hoai An},
         journal={Acta Mathematica Vietnamica},
         volume={22},
         number={1},
         pages={289--355},
         year={1997}
}

@Incollection{PhamDinh2014,
              author="Pham Dinh, Tao and Le Thi, Hoai An",
              editor="Nguyen, Ngoc-Thanh and Le Thi, Hoai An",
              title={{Recent Advances in DC Programming and DCA}},
              booktitle="Transactions on Computational Intelligence XIII",
              year="2014",
              publisher="Springer",
              address="Berlin, Heidelberg",
              pages="1--37",
              note={\url{https://doi.org/10.1007/978-3-642-54455-2_1}}
}

@article{Plackett1972,
         author = {Plackett, R. L.},
         title = {{Studies in the History of Probability and Statistics. XXIX: The discovery of the method of least squares}},
         journal = {Biometrika},
         volume = {59},
         number = {2},
         pages = {239--251},
         year = {1972},
         note = {\url{https://doi.org/10.1093/biomet/59.2.239}}
}

@incollection{Pop2017,
              author="Pop, Vasile and Furdui, Ovidiu",
              title={{Applications of Cayley--Hamilton Theorem}},
              bookTitle="Square Matrices of Order 2: Theory, Applications, and Problems",
              year="2017",
              publisher="Springer International Publishing",
              address="Cham",
              pages="107--182",
              edition={1st},
              note={\url{https://doi.org/10.1007/978-3-319-54939-2_3}}
}

@article{QU2021,
         title = {{A new approach to estimating earnings forecasting models: Robust regression MM-estimation}},
         journal = {International Journal of Forecasting},
         volume = {37},
         number = {2},
         pages = {1011--1030},
         year = {2021},
         note = {\url{https://doi.org/10.1016/j.ijforecast.2020.11.003}},
         author = {Qu, Li},
}

@book{reimann2008,
      title={{Statistical Data Analysis Explained: Applied Environmental Statistics with R}},
      author={Reimann, Clemens and Filzmoser, Peter and Garrett, Robert and Dutter, Rudolf},
      year={2008},
      edition={},
      publisher={John Wiley \& Sons},
      address={Chichester (England)},
      note={\url{https://doi.org/10.1002/9780470987605}}
}

@article{rousseeuw1984,
         title={{Least Median of Squares Regression}},
         author={Rousseeuw, Peter J},
         journal={Journal of the American Statistical Association},
         volume={79},
         number={388},
         pages={871--880},
         year={1984},
         note={\url{https://doi.org/10.2307/2288718}}
}

@article{rousseeuw1985,
         title={Multivariate estimation with high breakdown point},
         author={Rousseeuw, Peter J},
         journal={Mathematical Statistics and Applications,},
         pages={283--297},
         year={1985},
         volume={}
}

@book{rousseeuw1987,
      title={{Robust Regression and Outlier Detection}},
      author={Rousseeuw, Peter J and Leroy, Annick M},
      year={1987},
      address={New York},
      publisher={John Wiley \& Sons},
      note={\url{https://doi.org/10.1002/0471725382}},
      series={Wiley Series in Probability and Statistics}
}

@article{Rousseeuw1990,
         author = {Rousseeuw, Peter J. and van Zomeren, Bert C.},
         title = {{Unmasking Multivariate Outliers and Leverage Points}},
         journal = {Journal of the American Statistical Association},
         volume = {85},
         number = {411},
         pages = {633--639},
         year = {1990},
         note = {\url{https://doi.org/10.1080/01621459.1990.10474920}}
}

@article{ROUSSEEUW1992,
         title = {A comparison of some quick algorithms for robust regression},
         journal = {Computational Statistics \& Data Analysis},
         volume = {14},
         number = {1},
         pages = {107--116},
         year = {1992},
         note = {\url{https://doi.org/10.1016/0167-9473(92)90085-T}},
         author = {Rousseeuw, Peter J. and van Zomeren, Bert C. },
}

@article{Rousseeuw1997,
         note = {\url{http://www.jstor.org/stable/4355978}},
         author = {Rousseeuw, Peter J. and  Hubert, Mia},
         journal = {Lecture Notes-Monograph Series},
         pages = {201--214},
         title = {{Recent Developments in PROGRESS}},
         volume = {31},
         year = {1997}
}

@article{Rousseeuw1999,
         author = {Rousseeuw, Peter J. and Van Driessen, Katrien},
         title = {{A Fast Algorithm for the Minimum Covariance Determinant Estimator}},
         journal = {Technometrics},
         volume = {41},
         number = {3},
         pages = {212--223},
         year = {1999},
         note = {\url{https://doi.org/10.1080/00401706.1999.10485670}},
}

@article{rousseeuw2006,
         title={{Computing LTS Regression for Large Data Sets}},
         author={Rousseeuw, Peter J and Van Driessen, Katrien},
         journal={Data Mining and Knowledge Discovery},
         volume={12},
         pages={29--45},
         year={2006},
         note={\url{https://doi.org/10.1007/s10618-005-0024-4}}
}

@article{rousseeuw2011,
         title={Robust statistics for outlier detection},
         author={Rousseeuw, Peter J and Hubert, Mia},
         journal={Wiley Interdisciplinary Reviews: Data Mining and Knowledge Discovery},
         volume={1},
         number={1},
         pages={73--79},
         year={2011},
         note={\url{https://doi.org/10.1002/widm.2}}
}

@article{SABZEKAR2021,
         title = {Robust regression using support vector regressions},
         journal = {Chaos, Solitons \& Fractals},
         volume = {144},
         pages = {110738},
         year = {2021},
         note = {\url{https://doi.org/10.1016/j.chaos.2021.110738}},
         author = {Mostafa Sabzekar and Seyed Mohammad Hossein Hasheminejad}
}

@Incollection{Salgado2016,
              author={Salgado, C{\'a}tia M and Azevedo, Carlos and Proen{\c{c}}a, Hugo and Vieira, Susana M},
              editor={{MIT Critical Data}},
              title={{Noise Versus Outliers}},
              booktitle={Secondary Analysis of Electronic Health Records},
              year="2016",
              edition="1st",
              publisher="Springer International Publishing",
              address="Cham",
              pages="163--183",
              note={\url{https://doi.org/10.1007/978-3-319-43742-2_14}}
}

@article{Salibian-Barrera2006,
         author = {Salibian-Barrera, Matías and  Yohai, Víctor J},
         title = {{A Fast Algorithm for S-Regression Estimates}},
         journal = {Journal of Computational and Graphical Statistics},
         volume = {15},
         number = {2},
         pages = {414--427},
         year = {2006},
         note = {\url{https://doi.org/10.1198/106186006X113629}},
}

@article{Satman2012,
         author = {M. Hakan Satman},
         title = {{A Genetic Algorithm Based Modification on the LTS Algorithm for Large Data Sets}},
         journal = {Communications in Statistics - Simulation and Computation},
         volume = {41},
         number = {5},
         pages = {644--652},
         year = {2012},
         note = {\url{https://doi.org/10.1080/03610918.2011.598989}},
}

@article{satman2013,
         title={{A New Algorithm for Detecting Outliers in Linear Regression}},
         author={Satman, Mehmet Hakan},
         journal={International Journal of Statistics and Probability},
         volume={2},
         number={3},
         pages={101--109},
         year={2013},
         note = {\url{https://doi.org/10.5539/ijsp.v2n3p101}}
}

@article{Schelter2018,
         author = {Schelter, Sebastian and Lange, Dustin and Schmidt, Philipp and Celikel, Meltem and Biessmann, Felix and Grafberger, Andreas},
         title = {Automating large-scale data quality verification},
         year = {2018},
         volume = {11},
         number = {12},
         note = {\url{https://doi.org/10.14778/3229863.3229867}},
         journal = {Proceedings of the VLDB Endowment},
         pages = {1781--1794},
         numpages = {14}
}

@article{scikit-learn,
         title={Scikit-learn: Machine Learning in {P}ython},
         author={Pedregosa, F. and Varoquaux, G. and Gramfort, A. and Michel, V. and Thirion, B. and Grisel, O. and Blondel, M. and Prettenhofer, P. and Weiss, R. and Dubourg, V. and Vanderplas, J. and Passos, A. and Cournapeau, D. and Brucher, M. and Perrot, M. and Duchesnay, E.},
          journal={Journal of Machine Learning Research},
          volume={12},
          pages={2825--2830},
          year={2011}
}

@article{She2011,
         author = {Yiyuan She and Art B. Owen},
         title = {{Outlier Detection Using Nonconvex Penalized Regression}},
         journal = {Journal of the American Statistical Association},
         volume = {106},
         number = {494},
         pages = {626--639},
         year = {2011},
         note = {\url{https://doi.org/10.1198/jasa.2011.tm10390}}
}

@book{SheilSmall2002, 
      address={Cambridge}, 
      series={Cambridge Studies in Advanced Mathematics}, 
      title={Complex Polynomials}, 
      publisher={Cambridge University Press}, 
      author={Sheil-Small, T.}, 
      year={2002}, 
      edition={1st},
      note={\url{https://doi.org/10.1017/CBO9780511543074}}
}

@ARTICLE{Shen2013,
         author={Shen, Fumin and Shen, Chunhua and van den Hengel, Anton and Tang, Zhenmin},
         journal={IEEE Transactions on Image Processing}, 
         title={{Approximate Least Trimmed Sum of Squares Fitting and Applications in Image Analysis}}, 
         year={2013},
         volume={22},
         number={5},
         pages={1836--1847},
         note={\url{https://doi.org/10.1109/TIP.2013.2237914}}
}

@article{SMITI2020,
         title = {A critical overview of outlier detection methods},
         journal = {Computer Science Review},
         volume = {38},
         pages = {100306},
         year = {2020},
         note = {\url{https://doi.org/10.1016/j.cosrev.2020.100306}},
         author = {Smiti, Abir}
}

@INPROCEEDINGS{Solodov2011,
               author = {Solodov, Mikhail V.},
               publisher = {John Wiley \& Sons},
               title = {{Constraint Qualifications}},
               booktitle = {Wiley Encyclopedia of Operations Research and Management Science},
               chapter = {},
               pages = {},
               note = {\url{https://doi.org/10.1002/9780470400531.eorms0978}},
               year = {2011},
}

@misc{stuart2011,
      title={{Robust Regression. Department of Mathematical Sciences, University of Durham, Durham, UK}},
      author={Stuart, Chaterine},
      year={2011},
}

@INPROCEEDINGS{Subbarao2006,
               author={Subbarao, R. and Meer, P.},
               booktitle={2006 Conference on Computer Vision and Pattern Recognition Workshop (CVPRW'06)}, 
               title={{Beyond RANSAC: User Independent Robust Regression}}, 
               year={2006},
               pages={101--101},
               note={\url{https://doi.org/10.1109/CVPRW.2006.43}}
}

@article{SudermannMerx2021,
         title={{Leveraged least trimmed absolute deviations}},
         author={Sudermann-Merx, Nathan and Rebennack, Steffen},
         journal={OR Spectrum},
         volume={43},
         pages={809--834},
         year={2021},
         note={\url{https://doi.org/10.1007/s00291-021-00627-y}}
}

@article{Thormann2026b,
         title = {{The Boosted Difference of Convex Functions Algorithm for Value-at-Risk Constrained Portfolio Optimization}},
         journal = {Mathematical Programming Computation},
         volume = {},
         number = {},
         pages = {},
         year = {2026},
         note = {\url{https://doi.org/10.1007/s12532-026-00327-0}},
         author = {Thormann, Marah-Lisanne and Vuong, Phan Tu and Zemkoho, Alain B.},
}

@article{TORTI2012,
         title = {{Benchmark testing of algorithms for very robust regression: FS, LMS and LTS}},
         journal = {Computational Statistics \& Data Analysis},
         volume = {56},
         number = {8},
         pages = {2501--2512},
         year = {2012},
         note = {\url{https://doi.org/10.1016/j.csda.2012.02.003}},
         author = {Torti, Francesca and  Perrotta, Domenico and  Atkinson, Anthony C. and  Riani, Marco},
}

@article{Tukey1962,
         note = {\url{http://www.jstor.org/stable/2237638}},
         author = {John W. Tukey},
         journal = {The Annals of Mathematical Statistics},
         number = {1},
         pages = {1--67},
         publisher = {Institute of Mathematical Statistics},
         title = {{The Future of Data Analysis}},
         volume = {33},
         year = {1962}
}

@book{Tutz2011, 
      address={Cambridge}, 
      series={Cambridge Series in Statistical and Probabilistic Mathematics}, 
      title={{Regression for Categorical Data}}, 
      publisher={Cambridge University Press}, 
      author={Tutz, Gerhard}, 
      year={2011},
      note={\url{https://doi.org/10.1017/CBO9780511842061}}
}

@book{tuy2018,
      title={{Convex Analysis and Global Optimization}},
      author={Tuy, Hoang},
      year={2016},
      edition={2nd},
      publisher={Springer},
      address={Cham},
      note={\url{https://doi.org/10.1007/978-3-319-31484-6}}
}

@article{Verardi2009,
         author = {Vincenzo Verardi and Christophe Croux},
         title ={{Robust Regression in Stata}},
         journal = {The Stata Journal},
         volume = {9},
         number = {3},
         pages = {439--453},
         year = {2009},
         note = {\url{https://doi.org/10.1177/1536867X0900900306}}
}

@article{Verboven2010,
         author = {Verboven, Sabine and Hubert, Mia},
         title ={{MATLAB library LIBRA}},
         journal = {WIREs Computational Statistics},
         volume = {2},
         number={4},
         pages = {509--515},
         year = {2010},
         note = {\url{https://doi.org/10.1002/wics.96}}
}

@misc{wang2015,
      title={{Projection onto the capped simplex, arXiv preprint arXiv:1503.01002 [cs.LG]}},
      author={Wang, Weiran and Lu, Canyi},
      year={2015},
      note={\url{https://doi.org/10.48550/arXiv.1503.01002}}
}

@article{Western1995,
         note = {\url{https://doi.org/10.2307/2111654}},
         author = {Bruce Western},
         journal = {American Journal of Political Science},
         number = {3},
         pages = {786--817},
         title = {{Concepts and Suggestions for Robust Regression Analysis}},
         volume = {39},
         year = {1995}
}

@Incollection{Wilcox2010,
              author="Wilcox, Rand R.",
              editor="Wilcox, Rand R.",
              title="Robust Regression",
              bookTitle="Fundamentals of Modern Statistical Methods: Substantially Improving Power and Accuracy",
              year="2010",
              edition="2nd",
              publisher="Springer",
              address="New York",
              pages="193--215",
              note={\url{https://doi.org/10.1007/978-1-4419-5525-8_11}}
}

@article{wilcox2023,
         title={{Regression: Identifying Good and Bad Leverage Points}},
         author={Wilcox, Rand R. and Xu, Lai},
         journal={International Journal of Statistics and Probability},
         volume={12},
         number={1},
         pages={1--9},
         year={2023},
         note={\url{https://doi.org/10.5539/ijsp.v12n1p1}}
}

@article{Wu2023,
         author = {Wu, Can and Cui, Ying and Li, Donghui and Sun, Defeng},
         title = {{Convex and Nonconvex Risk-Based Linear Regression at Scale}},
         journal = {INFORMS Journal on Computing},
         volume = {35},
         number = {4},
         pages = {797--816},
         year = {2023},
         note = {\url{https://doi.org/10.1287/ijoc.2023.1282}}
}

@article{XIE2022,
         title = {Partial least trimmed squares regression},
         journal = {{Chemometrics and Intelligent Laboratory Systems}},
         volume = {221},
         pages = {104486},
         year = {2022},
         note = {\url{https://doi.org/10.1016/j.chemolab.2021.104486}},
         author = {Zhonghao Xie and Xi'an Feng and Xiaojing Chen}
}

@book{Yagi2021, 
      series={SpringerBriefs in Mathematics}, 
      title={Abstract Parabolic Evolution Equations and \L ojasiewicz–Simon Inequality~I: Abstract Theory}, 
      publisher={Springer}, 
      address={Singapore},
      edition={1st},
      author={Yagi, Atsushi}, 
      year={2021}, 
      note={\url{https://doi.org/10.1007/978-981-16-1896-3}}
}

@article{yu2014,
         title={An adversarial optimization approach to efficient outlier removal},
         author={Yu, Jin and Eriksson, Anders and Chin, Tat-Jun and Suter, David},
         journal={Journal of Mathematical Imaging and Vision},
         volume={48},
         pages={451--466},
         year={2014},
         note = {\url{https://doi.org/10.1007/s10851-013-0418-7}}
}

@article{ZHANG2023,
         title = {A survey for solving mixed integer programming via machine learning},
         journal = {Neurocomputing},
         volume = {519},
         pages = {205--217},
         year = {2023},
         note = {\url{https://doi.org/10.1016/j.neucom.2022.11.024}},
         author = {Zhang, Jiayi and  Liu, Chang and  Li, Xijun and  Zhen, Hui-Ling and  Yuan, Mingxuan and  Li, Yawen and  Yan, Junchi}
}

@article{Zioutas2009,
         title = {Quadratic mixed integer programming and support vectors for deleting outliers in robust regression},
         journal = {Annals of Operations Reserach},
         volume = {166},
         pages = {339–-353},
         year = {2009},
         note = {\url{https://doi.org/10.1007/s10479-008-0412-4}},
         author = {Zioutas, G. and Pitsoulis, L. and Avramidis, A.}
}

@misc{zuo2022,
      title={{New algorithms for computing the least trimmed squares estimator, arXiv preprint, arXiv:2203.10387 [stat.CO]}}, 
      author={Yijun Zuo},
      year={2022},
      note={\url{https://doi.org/10.48550/arXiv.2203.10387}}, 
}


\appendix

\setcitestyle{authoryear,open={(},close={)}}





\section{Feasible Set}\label{sec:feasible_set}


In Section~\ref{sec:LTS_as_DC}, the capped simplex $\Delta_h$ was introduced as the feasible set of the optimisation problem considered in this paper.
To provide better insight into its geometric structure, this appendix includes a visualisation of $\Delta_h$ for $h \in \{1, 2\}$ from different perspectives, as shown in Figure \ref{fig:capped_simplex}.
Specifically, the three subplots are based on three variables $z_1, z_2, z_3 \in [0, 1]$, which satisfy either $z_1 + z_2 + z_3 = 2$ (green triangle), or $z_1 + z_2 + z_3 = 1$ (blue triangle). 
The latter triangle is also known as the probability or unit simplex. 
From these representations, it can be observed that the feasible set is closed, bounded, and convex, as it corresponds to a hypercube (due to the box constraints) intersected by a hyperplane (due to the equality constraint). 


\vspace{-0.2cm}

\begin{figure}[H]%
    \centering
    \captionsetup{justification=centering}
    \includegraphics[width=1\textwidth]{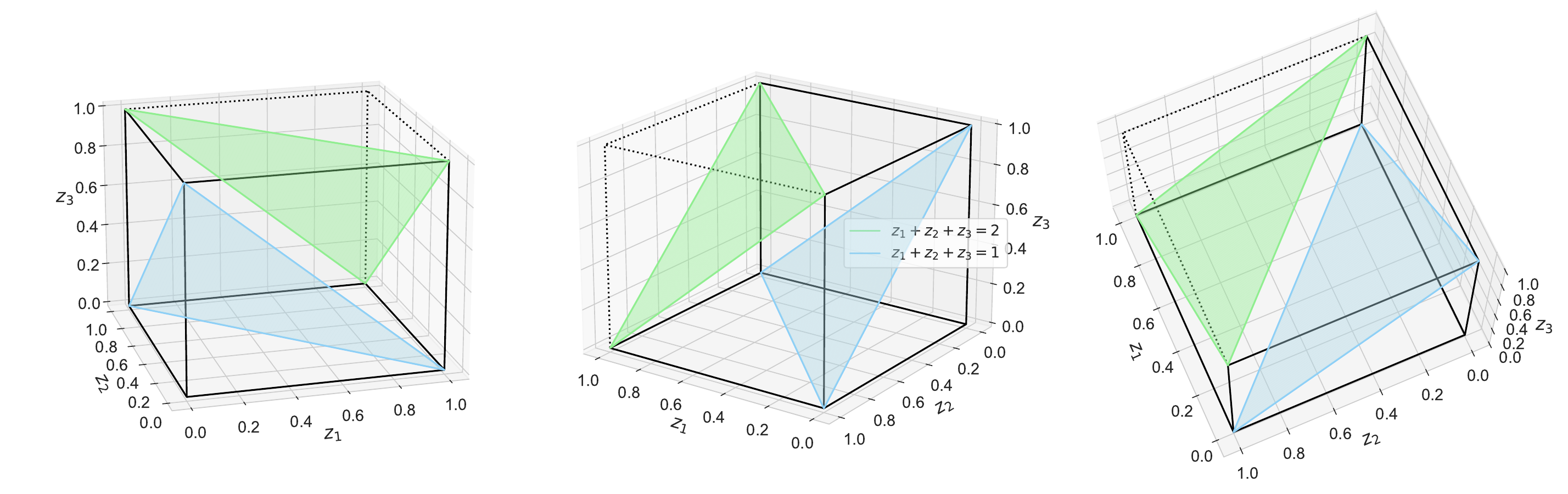}
    \caption{Different Perspectives on the Probability and Capped Simplex.}%
    \label{fig:capped_simplex}%
\end{figure}

\vspace{-0.2cm}



\section{Solving the Subproblem with Projection onto Capped Simplex}\label{sec:projection_onto_capped_simplex}


In Subsection \ref{sec:sBDCA}, the general procedure of \gls{sBDCA} for the \gls{LTS} problem was outlined.
For brevity, we omitted the details of how the projection problem \eqref{eq:projection_onto_capped_simplex} is solved in practice and how its solution is related to the minimiser of \eqref{eq:subproblem_lts}.
To complete the algorithmic description, this appendix provides the corresponding explanations. 
At each iteration, the strongly convex subproblem~\eqref{eq:subproblem_lts} must be constructed and solved anew, which significantly contributes to the overall computation time.
Favourably, for the chosen \gls{DC} decomposition, the strongly convex approximation $\Tilde{f}_t (\mathbf{z})$ around $\mathbf{z}_t$ is a quadratic function, which can be expressed in standard form as
\begin{align}\label{eq:qp}
    \begin{split}        
        \Tilde{f}_t (\mathbf{z}) &= \phi_1(\mathbf{z}) - \mathbf{z}^\top \nabla \phi_2(\mathbf{z}_t)\\
        &= \rho_t \cdot \lVert \mathbf{z} \rVert ^2 +  \mathbf{z}^\top \big (- \nabla \phi_2 (\mathbf{z}_t) \big ) \\
        &= \rho_t \cdot \mathbf{z}^\top \mathbf{z} +  \mathbf{z}^\top \big( - \nabla \phi_2 (\mathbf{z}_t) \big ) + r_t,
    \end{split}
\end{align}
where $r_t = 0$ (for all $t$) is the constant term that could influence the optimal objective function value $\Tilde{f}_t^*$ but does not change the optimal decision variables $\Tilde{\mathbf{y}}_t$. 
Consequently, the strongly convex subproblem~\eqref{eq:subproblem_lts} is a \gls{QP} that, in principle, could be solved based on various methods such as interior-point, active set or augmented Lagrangian [cf. \citet[ch.~9.3]{Giorgi2023}; \citet[ch.~16]{Nocedal2006}]. 
In our practical implementation, we experimented with various approaches.
Ultimately, we opted for the projection procedure proposed by \citet{ang2021}, since it returns a high accuracy solution for $\Tilde{\mathbf{y}}_t$, and is drastically faster than state-of-the-art commercial solvers like Gurobi that utilises an interior-point method [cf. \citet[p.~7]{ang2021}].
This performance difference becomes particularly evident in higher dimensions.


Before describing the projection procedure proposed by \citet{ang2021} in more detail, we note that their algorithm does not directly minimise the objective function $\Tilde{f}_t (\mathbf{z})$ subject to the given constraints.
Instead, it solves the following \gls{QP}
\begin{align}\label{eq:projection_onto_capped_simplex2}\tag{$\mathcal{P}_{\Delta_h}$}
 \begin{split}
      P_{\Delta_h} \left ( \Tilde{\mathbf{z}}_t \right) = \underset{\mathbf{z} \in \Delta_h}{\text{arg min}} \ \  \frac{1}{2} \cdot \bigl \lVert \mathbf{z} - \Tilde{\mathbf{z}}_t  \bigr \rVert^2,
 \end{split}
\end{align}
which is known as the projection of $\Tilde{\mathbf{z}}_t = \begin{bmatrix} \Tilde{z}_{t, 1}, \ldots, \Tilde{z}_{t, n} \end{bmatrix}^\top \in \mathbb{R}^n$ onto the capped simplex [cf.~Proposition \ref{prop:projection_theorem}]. 
By requiring that $\Tilde{\mathbf{z}}_t = \mathbf{z}_t - \eta_t \cdot \nabla f(\mathbf{z}_t)$ with $\eta_t = \frac{1}{2 \cdot \rho_t}$, it can be shown that the optimal solution of~\eqref{eq:projection_onto_capped_simplex} coincides with $\Tilde{\mathbf{y}}_t$, which is the minimiser of the convex subproblem~\eqref{eq:subproblem_lts}. 
In particular, we only have to add the constant
\begin{align*}
    \Bar{r}_t &:=  \rho_t \cdot \bigl \lVert \mathbf{z}_t \bigr \rVert ^2 -  \mathbf{z}_t^\top \nabla f(\mathbf{z}_t) + \frac{1}{4 \cdot \rho_t} \cdot \bigl \lVert \nabla f(\mathbf{z}_t) \bigr \rVert^2 = \frac{1}{4 \cdot \rho_t} \cdot \bigl \lVert \nabla \phi_2 (\mathbf{z}_t) \bigr \rVert ^2
\end{align*}
to the objective function  $\Tilde{f}_t (\mathbf{z})$ and rescale the sum by $\frac{1}{2 \cdot \rho_t}$ to obtain 
\begin{align}
    \Bar{\phi}_t (\mathbf{z}) 
    :&=  \frac{1}{2 \cdot \rho_t} \cdot \Bigl ( \Tilde{f}_t (\mathbf{z}) + \Bar{r}_t \Bigr ) \\
    &= \frac{1}{2} \cdot \Bigl\lVert \mathbf{z} -  \frac{1}{2 \cdot \rho_t} \cdot \Bigl (2 \cdot \rho_t \cdot \mathbf{z}_t - \nabla f(\mathbf{z}_t) \Bigr ) \Bigr \rVert ^2 \\
    &= \frac{1}{2} \cdot \Bigl \lVert \mathbf{z} -  \Big ( \mathbf{z}_t - \eta_t \cdot \nabla f(\mathbf{z}_t) \Big ) \Bigr \rVert ^2  \label{eq:two} \\
    &= \frac{1}{2} \cdot \bigl \lVert \mathbf{z} -  \Tilde{\mathbf{z}}_t \bigr \rVert ^2, 
\end{align}
which is the objective function of the projection procedure \eqref{eq:projection_onto_capped_simplex}. 
As adding a constant term $\Bar{r}_t$ and rescaling by a constant factor $\frac{1}{2 \cdot \rho_t}$ does not alter the optimal solution of problem~\eqref{eq:subproblem_lts}, one can equivalently solve the problem  \eqref{eq:projection_onto_capped_simplex} to obtain $\Tilde{\mathbf{y}}_t$.


\vspace{0.2cm}

\begin{prop}[{Projection {\small [\citet[p. 201]{bertsekas1999}]}}] {prop:projection_theorem}
    Let $\Omega$ be a nonempty, closed, and convex subset of $\mathbb{R}^n$, then it holds that
    \begin{itemize}
        \item For every $\mathbf{z} \in \mathbb{R}^n$, there exist a unique $\mathbf{x}^* \in \Omega$ that minimises $\lVert \mathbf{z} - \mathbf{x} \rVert$ over all $\mathbf{x} \in \Omega$. This vector is called the \textbf{projection} of $\mathbf{z}$ on $\Omega$ and is denoted by $P_{\Omega}(\mathbf{z})$.
        \item Given some $\mathbf{z} \in \mathbb{R}^n$, a vector $\mathbf{x}^* \in \Omega$ is equal to the projection $P_{\Omega}(\mathbf{z})$ if and only if
        \begin{equation*}
            \langle \mathbf{z} - \mathbf{x}^*, \mathbf{x} - \mathbf{x}^* \rangle \leq 0, \quad \forall \ \mathbf{x} \in \Omega.
        \end{equation*}
        \item The mapping $f: \mathbb{R}^n \rightarrow \Omega$ defined by $f(\mathbf{x}) = P_{\Omega}(\mathbf{x})$ is continuous and nonexpansive, that is, 
        \begin{equation*}
            \lVert P_{\Omega}(\mathbf{x}) - P_{\Omega}(\mathbf{y}) \rVert \leq \lVert \mathbf{x} - \mathbf{y} \rVert, \quad \forall \ \mathbf{x}, \mathbf{y} \in \mathbb{R}^n.
        \end{equation*}
        \item In the case that $\Omega$ is a subspace, a vector $\mathbf{x}^* \in \Omega$ is equal to the projection $P_{\Omega}(\mathbf{z})$ if and only if $\mathbf{z} - \mathbf{x}^*$ is orthogonal to $\Omega$, that is,
        \begin{equation*}
            \langle \mathbf{z} - \mathbf{x}^*, \mathbf{x} \rangle = 0 \quad \forall \ \mathbf{x} \in \Omega.
        \end{equation*}
    \end{itemize}
\end{prop}



Let us now examine the main steps of the fast projection algorithm by \citet{ang2021} shown in Algorithm~\ref{alg:fast_projection}. 
The core idea of the procedure is to reformulate problem \eqref{eq:projection_onto_capped_simplex} as a scalar minimisation with respect to the Lagrange multiplier of the equality constraint, and to solve the resulting problem using Newton’s method. 
For this purpose, we first construct a min-max problem that is equivalent to \eqref{eq:projection_onto_capped_simplex} and has the following form
\begin{equation}\label{eq:partial_lagrangian}
    \underset{\mathbf{z} \in [0, 1]^n}{\text{minimise}} \quad \underset{\gamma \in \mathbb{R}}{\text{maximise}} \quad \mathcal{L} \left( \mathbf{z}, \gamma \right) := \frac{1}{2} \cdot \bigl \lVert \mathbf{z} - \Tilde{\mathbf{z}}_t  \bigr \rVert^2 + \gamma \cdot (\mathbf{z}^\top \mathbf{1}_n - h),
\end{equation}
where $\mathcal{L}: \mathbb{R}^n \times \mathbb{R} \rightarrow \mathbb{R}$ denotes the partial Lagrangian obtained by introducing the Lagrange multiplier $\gamma \in \mathbb{R}$ associated with the equality constraint $\Tilde{h}(\mathbf{z})$. 
If we then assume that the optimal value for the inner maximisation, denoted by $\gamma^*$, is already known, the outer minimisation admits a closed-form solution of the form
\begin{align}\label{eq:closed_form_solution}
    \begin{split}
    \Tilde{\mathbf{y}}_t  &= \ \underset{\mathbf{z} \in [0, 1]^n}{\text{arg min}} \ \Tilde{\mathcal{L}}(\mathbf{z}) := \sum_{i=1}^n \frac{1}{2} \cdot z_i^2 + z_i \cdot (\gamma^* - \Tilde{z}_{t,i})  
    = \min \Bigl \{ \mathbf{1}_n, \max \big[ \mathbf{0}_n, \mathbf{v}(\gamma^*) \big] \Bigr \},
    \end{split}
\end{align}
where $\min\{\cdot, \cdot\}$ and $\max[\cdot, \cdot]$ are taken element-wise, $\Tilde{\mathcal{L}}: \mathbb{R}^n \rightarrow \mathbb{R}$, and the function $\mathbf{v}: \mathbb{R} \rightarrow \mathbb{R}^n$ is defined as
\begin{equation*}
    \mathbf{v}(\gamma) := \Tilde{\mathbf{z}}_t - \gamma \cdot \mathbf{1}_n.
\end{equation*}
The previous result then can be used to transform problem \eqref{eq:partial_lagrangian} into the following unconstrained scalar minimisation problem 
\begin{equation}\label{eq:scalar_minimisation}
    \underset{\gamma \in \mathbb{R}}{\text{minimise}} \quad \omega(\gamma) := - \mathcal{L} \Big (\min \big \{\mathbf{1}_n, \max \big[ \mathbf{0}_n, \mathbf{v}(\gamma) \big] \big\}, \gamma \Big),
\end{equation}
where firstly the order of the min-max was swapped, then the closed-form solution shown in \eqref{eq:closed_form_solution} was plugged in, and lastly the maximisation with respect to $\gamma$ was converted into a corresponding minimisation problem. 
In order to solve problem \eqref{eq:scalar_minimisation} based on Newton's method, we only have to construct the first order derivative of $\omega$ that is defined as
\begin{equation*}
    \nabla \omega \left(\gamma \right) := h - \min \Big \{ \mathbf{1}_n, \max \big[ \mathbf{0}_n, \mathbf{v} \left(\gamma \right) \big ] \Big \}^\top \mathbf{1}_n,
\end{equation*}
as well as the second order derivative of $\omega$ that is given as
\begin{equation*}
    \nabla^2 \omega \left(\gamma \right) := \sum_{i=1}^n \mathds{1}_{(0,1)} \left (v_i \right),
\end{equation*}
where $v_i \in \mathbb{R}$ is the $i^{\text{th}}$ element of the vector $\mathbf{v}(\gamma)$ for all $i \in \{1, \ldots, n\}$, and $\mathds{1}_{(0,1)}: \mathbb{R} \rightarrow \{0, 1\}$ represents an indicator function that is defined as
\begin{equation*}
    \mathds{1}_{(0,1)} \left(v_i \right) :=
    \begin{cases}
        1 \quad \text{if} \quad v_i \in (0, 1), \\
        0 \quad \text{if} \quad v_i \notin (0, 1). \\
    \end{cases}
\end{equation*}


\vspace{-0.2cm}

\begin{algorithm}[H]
  \KwInput{$\Tilde{\mathbf{z}}_t \notin \mathbf{0}_n$, $\varepsilon > 0$, $\bar{T} \in \mathbb{N}$}   
  \vspace{0.05cm}
  $\bar{t} \gets 0$ \\
  \vspace{0.05cm}
  $\delta \gets \infty$ \\
  \vspace{0.05cm}
  $\gamma_0 \gets \frac{\max \left\{\Tilde{\mathbf{z}}_t \right\} - 1 + \max \left\{\Tilde{\mathbf{z}}_t \right \}}{2}$ \\
  \vspace{0.05cm}
  \While{$\delta > \varepsilon$ \textbf{\upshape and} $\bar{t} < \bar{T}$}
   {
   \vspace{0.05cm}
   $\bar{t} \gets \bar{t} + 1$ \\
   \vspace{0.05cm}
   $\gamma_{\bar{t}} \gets \gamma_{\bar{t} - 1} - \frac{\nabla \omega(\gamma_{\bar{t} - 1})}{\nabla^2 \omega(\gamma_{\bar{t} -1})}$\\
   \vspace{0.05cm}
   $\delta \gets \lVert \gamma_{\bar{t}} - \gamma_{\bar{t}- 1} \rVert$ \\
   }
   \vspace{0.05cm}
   $\Tilde{\mathbf{y}}_t \gets \min \Big \{ \mathbf{1}_n, \max \big[ \mathbf{0}_n, \mathbf{v}(\gamma_{\bar{t}}) \big] \Big \}$ \\
   \vspace{0.05cm}
  \KwOutput{$\Tilde{\mathbf{y}}_t$}
\caption{Fast Projection onto $\Delta_h$ [cf. \citet[p. 5]{ang2021}]}\label{alg:fast_projection}

\end{algorithm} 


After examining the main steps of the procedure, it remains to discuss its runtime performance.
The computational cost of Algorithm \ref{alg:fast_projection} is primarily driven by Line~6, which incurs a cost of $\mathcal{O}(n)$ per iteration.
Consequently, the total algorithmic cost is determined by $T \cdot \mathcal{O}(n)$, where $T \in \mathbb{N}$ denotes the number of iterations required for convergence [cf. \citet[p. 5]{ang2021}]. 
Even though in theory $T$ could be arbitrarily large, empirically one can observe that the algorithm usually converges within a small number of iterations [cf. \citet[p. 5]{ang2021}]. 
Thus, in practice, the computational complexity is roughly about $\mathcal{O}(n)$, and in higher dimensions the projection procedure especially outperforms sorting-based methods as proposed by \citet{wang2015} with time complexity $\mathcal{O}(n^2)$ [cf. \citet[p. 7]{ang2021}]. 
An empirical comparison of the methods and more theoretical details can be found in \citet{ang2021}.



\section{Hat Matrix and Hat Values}\label{sec:hat_values}


In Subsection~\ref{sec:initialisation_and_preconditioner}, the hat matrix was briefly introduced to derive a preconditioning matrix for the gradient of the objective function. 
To complete the discussion, this appendix explains how the hat values can be linked to outlier scores in the $\mathbf{x}$-direction and to the standardisation of the residuals.
As outlined in Subsection~\ref{sec:initialisation_and_preconditioner}, the hat matrix $\mathbf{H}$  possesses a number of noteworthy properties [cf. \citet[p. 15]{morris2011}]. 
Due to these properties, $\mathbf{H}$ also corresponds to the \textit{orthogonal projection matrix} associated with the column space of $\mathbf{X}$ and ensures that the predictions $\Hat{\mathbf{y}} := \mathbf{H}\mathbf{y} \in \mathbb{R}^n$ are as close as possible to $\mathbf{y}$ in the least-squares sense [cf.~\citet[p. 15]{morris2011}]. 
The latter then also highlights that the \gls{OLS} predictions $\hat{\mathbf{y}}$ are linear combinations of the observed responses $\mathbf{y}$, where the weights are provided by the elements of $\mathbf{H}$ that only depends on $\mathbf{X}$ [cf.~\citet[p. 196]{Hocking2003}]. 
Therefore, the influence of $y_j$ exerted on the prediction $\hat{y}_i$ can be deduced based on element $h_{ij}$ of $\mathbf{H} = ( h_{ij} )_{1 \leq i, j \leq n}$ [cf. \citet[p. 22]{Tutz2011}].


Let us now examine the diagonal elements of the hat matrix that are also frequently referred to as \textit{leverage statistics/scores} or \textit{hat values} [cf. \citet[p. 99]{James2021}; \citet[pp. 398--399]{fox2019}]. 
In outlier diagnostic, they are often of particular interest as they give important insights into the underlying data structures, and indicate how strongly the data points influenced their own predictions [cf. \citet[p. 13]{DEMAESSCHALCK20001}]. 
Generally, the hat values are bounded from below and above as the previously mentioned properties of $\mathbf{H}$ imply that
\begin{equation*}
    0 < \frac{1}{n} \leq h_{ii} \leq \frac{\bar{n}}{n} \leq 1,
\end{equation*}
where $0 < \bar{n} \leq n \in \mathbb{N}$ denotes the number of  rows in $\mathbf{X}$ with unique $\mathbf{x}_i$ [cf. \citet[p. 120]{Fahrmeir2021}]. 
Further, if $h_{ii}$ is close to the upper bound, then the remaining entries in the $i^{th}$ row of  $\mathbf{H}$ are close to zero, and the regression hyperplane nearly intersects the point $(y_i, \mathbf{x}_i)$ [cf. \citet[pp. 180--181]{Fahrmeir2021}]. 
Accordingly, the bigger the hat value, the higher is the influence of the $i$-th data point on the estimation results. 
Apart from these facts, the diagonal elements of $\mathbf{H}$ can also be interpreted as outlier scores in $\mathbf{x}$-direction as they can be expressed as
\begin{equation*}
    h_{ii} = \frac{1}{n} + (\Tilde{\mathbf{x}}_i - \bar{\mathbf{x}})^\top  (\Tilde{\mathbf{X}}^\top \Tilde{\mathbf{X}} )^{-1} (\Tilde{\mathbf{x}}_i - \bar{\mathbf{x}})
\end{equation*}
where $\Tilde{\mathbf{x}}_i \in \mathbb{R}^{p-1}$ denotes the covariates of the $i$-th observations excluding the intercept, $\Tilde{\mathbf{X}} \in \mathcal{M}_{n \times p -1}(\mathbb{R})$ represents the centred version of $\mathbf{X}$ without intercept, and $\bar{\mathbf{x}} \in \mathbb{R}^{p - 1}$ contains the column means of $\mathbf{X}$ also excluding the intercept [cf. \citet[pp. 197--198]{Hocking2003}]. 
Therefore, the higher the hat value, the further $\mathbf{x}_i$ is away from the centre of the ellipsoid that is spanned by the covariates [cf. \citet[pp. 197--198]{Hocking2003}]. 
In this regard, it is also often reasonable to assume that the rows of $\mathbf{X}$ are realisations of a multivariate normal distribution such that the previous representation of the diagonal element $h_{ii}$ can be linked to the squared Mahalanobis distance [cf.~Definition~\ref{defi:Mahalanobis_distance}] between the $i$-th data point $\mathbf{x}_i$ and the multivariate data distribution where $\bar{\mathbf{x}}$ and $\Tilde{\mathbf{X}}^\top \Tilde{\mathbf{X}}$ are sample estimates for the true but unknown means and covariances, respectively [cf.~\citet[p. 310]{Chave2003}]. 


\vspace{0.2cm}

\begin{defi}[{Mahalanobis Distance {\small [cf. \citet{Mahalanobis2018}]}}]{defi:Mahalanobis_distance}
    Let $\mathcal{P}$ denote a (multivariate) probability distribution with mean $\boldsymbol{\mu} \in \mathbb{R}^p$ and positive semi-definite covariance matrix $\boldsymbol{\Sigma} \in \mathcal{M}_{p \times p}(\mathbb{R})$. Then the \textbf{Mahalanobis distance} between an arbitrary $p$-dimensional point $\mathbf{x} \in \mathbb{R}^p$ and the distribution $\mathcal{P}$ is defined as
    \begin{equation*}
    d_{M} \left( \mathbf{x}, \mathcal{P} \right) := \sqrt{\left( \mathbf{x} - \boldsymbol{\mu} \right)^\top \boldsymbol{\Sigma}^{-1}\left( \mathbf{x} - \boldsymbol{\mu} \right)},
    \end{equation*}
    which simplifies to the \textbf{Euclidean distance} if the covariance matrix fulfils $\boldsymbol{\Sigma} = \mathbf{I}_n$.
\end{defi}

\vspace{-0.2cm}


Besides the previously derived connection to the Mahalanobis distance, it is noteworthy that the hat values also provide important information about the variability of the residuals $\mathbf{r} := \mathbf{y} - \Hat{\mathbf{y}} \in \mathbb{R}^n$. 
In contrast to the true unknown error terms $\boldsymbol{\epsilon}$, the residuals do not necessarily have homoscedastic variances and are not uncorrelated [cf. \citet[p. 136]{Fahrmeir2021}]. 
The latter can be easily seen based on their covariance matrix
\begin{align*}
    \begin{split}
       \mathbb{V} \left(\mathbf{r} \right) = \mathbb{V} \left(\mathbf{y} - \hat{\mathbf{y}} \right) &= \left( \mathbf{I}_n - \mathbf{H} \right) \mathbb{V} \left(\mathbf{y} \right) (\mathbf{I}_n - \mathbf{H}) \\
       &= \sigma^2 \cdot \left(\mathbf{I}_n - \mathbf{H} - \mathbf{H} + \mathbf{H}^2  \right) \\
        &= \sigma^2 \cdot \left(\mathbf{I}_n - \mathbf{H} \right),
    \end{split}
\end{align*}
which also directly points out the leading role of the hat matrix [cf. \citet[p. 21]{Tutz2011}]. Accordingly, the higher the hat value $h_{ii}$, the stronger the variance of the $i$-th residual $r_i$ approaches zero, which can be even better seen based on
\begin{equation*}
    \mathbb{V} \left(r_i \right) = \sigma^2 \cdot \left(1 - h_{ii} \right),
\end{equation*}
where $\sigma \in \mathbb{R}_+$ corresponds to the constant standard deviation of the true unknown error terms $\boldsymbol{\epsilon}$ [cf.~\citet[p. 56]{Chatterjee2020}]. 
Even though it might appear slightly counterintuitive that the residual variance of a point with high leverage is close to zero, one has to keep in mind that variance is measured relative to repetitions of the data experiment with the same independent variables [cf.~\citet[p. 199]{Hocking2003}]. 
In most replications, such an unusual combination of predictors would cause that the regression fit nearly intersects the point $(y_i, \mathbf{x}_i)$. 
Therefore, the corresponding residuals are always small with low variability. 
This then highlights once more that the residuals are not good indicators for detecting leverage points and that instead one could consider the following scaling
\begin{equation*}
    \Tilde{r}_i := \frac{r_i}{\sqrt{1 - h_{ii}}} \in \mathbb{R},
\end{equation*}
to obtain standardised residuals that at least in theory all have the same variance, i.e., $\mathbb{V} \left( \Tilde{r}_i \right) = \sigma^2$, and are generally better indicators for outliers in $\mathbf{x}$-direction [cf. \citet[pp. 21--22]{Tutz2011}; \citet[pp. 199]{Hocking2003}]. 
More precisely, if the data set contained only one leverage point, such a scaling would guarantee that the residual of the respective instance is inflated [cf. \citet[p. 107]{Celik1992}]. 
However, in the case of several outliers, it is noteworthy that the scaling could not have its desired effect. 
This can be explained based on the Mahalanobis distance representation of the hat value where the sample estimates for the true unknown means and covariances can be distorted due to masking effects [cf. \citet[p. 635]{Rousseeuw1990}]. 
To address this limitation, \citet{Rousseeuw1990} proposed to replace the samples estimates with robuster versions that generally increase the reliability of the outlier diagnostic.
However, as this proposal also comes at a much higher computational cost, we restricted our approach to using the traditional hat values in the scaling matrix $\mathbf{B}_t$.



\section{Additional Numerical Results}\label{sec:further_num_results}


In this section, we provide additional graphical illustrations for the experiments conducted in Subsection~\ref{sec:synthetic_data_exp}.
In particular, Figure~\ref{fig:dmlp_atd_inf_2} compares the fraction of infeasible solutions (left y-axis, bar plots) and the \gls{ATD} of all feasible outcomes (right y-axis, box plots) across the selected settings based on the \gls{DMLP}. 
In comparison to Figure~\ref{fig:atd_s2} in Subsection~\ref{sec:synthetic_data_exp}, the \gls{ATD}-axis is not restricted to the interval $[0.05, 0.45]$. 
This allows to better observe for which data settings the Fast-\gls{LTS} estimator breaks down. 
Specifically, the high dimensional setups with many independent variables are the most challenging for the heuristic, where the deteriorating initialisation quality reveals the lack of robustness in the underlying sorting procedure. 
As a result, the median \gls{ATD} of \gls{sBDCA} with preconditioning is up to 90\% lower than the solution of Fast-\gls{LTS}.


\begin{figure}[H]%
    \centering
    \includegraphics[width=1\textwidth]{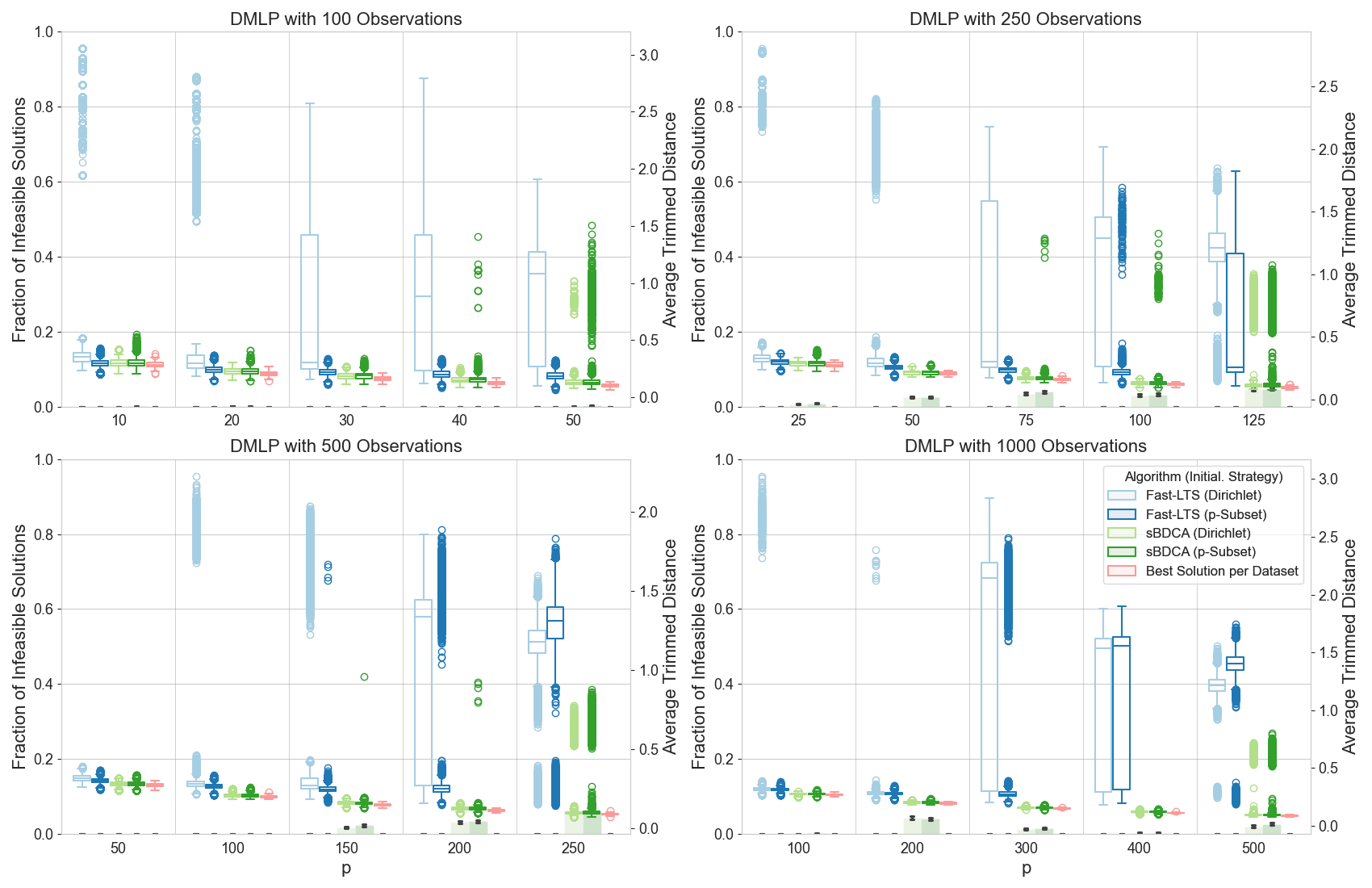}
    \caption{Performance comparison between \gls{sBDCA} with preconditioning and Fast-\gls{LTS} using the \gls{DMLP} and focusing on algorithmic reliability and output quality. 
    Each subplot displays the fraction of infeasible solutions (left y-axis, bar plots) and the \gls{ATD} of all feasible outcomes (right y-axis, box plots) across different numbers of input variables on the x-axis. 
    Results are grouped by solver type and initialisation strategy (Dirichlet, $p$-subset).
    The box plots in red display the lowest \gls{ATD} among all feasible solutions per dataset independent of the solver.}%
    \label{fig:dmlp_atd_inf_2}%
\end{figure}

\newpage


\printglossary[type=\acronymtype, title=List of Abbreviations]
\printglossary

\end{document}